\documentclass[10pt]{amsart}
\usepackage[margin=1in]{geometry}
\usepackage{amsmath, amsthm, amssymb, color, url}
\usepackage{hyperref}
\hypersetup{
    colorlinks,
    citecolor=blue,
    filecolor=black,
    linkcolor=black,
    urlcolor=black
}
\usepackage[final]{showkeys}
\usepackage[disable]{todonotes}

\makeatletter
\let\@wraptoccontribs\wraptoccontribs
\makeatother

\providecommand{\url}[1]{\url{#1}}

\newcommand{\RR}{\mathbb{R}}

\newcommand{\wD}{\wtilde{D}}
\newcommand{\lc}{{\text{lc}}}

\newcommand{\wnu}{\wtilde{\nu}}

\newcommand{\wPhi}{\wtilde{\Phi}}

\newcommand{\wgamma}{\wtilde{\gamma}}

\newcommand{\eps}{\varepsilon}

\newcommand{\wtilde}{\widetilde}

\newcommand{\FF}{\mathcal{F}}

\newcommand{\EE}{\mathbb{E}}

\newcommand{\wF}{\wtilde{F}}
\newcommand{\cD}{\mathcal{D}}

\newcommand{\BESQ}{\mathsf{BESQ}}
\newcommand{\JAC}{\mathsf{JAC}}
\newcommand{\XX}{\mathcal{X}}
\newcommand{\YY}{\mathcal{Y}}
\renewcommand{\AA}{\mathcal{A}}
\newcommand{\PP}{\mathbb{P}}
\renewcommand{\div}{\text{div}}
\newcommand{\NN}{\mathbb{N}}

\newcommand{\GT}{\mathbb{GT}}
\newcommand{\GTO}{\mathbb{GT}^{[0, 1]}}

\newcommand{\J}{\mathsf{J}}

\newcommand{\wM}{\wtilde{M}}

\newcommand{\vol}{\text{vol}}
\newcommand{\dom}{\mathsf{dom}}
\newcommand{\cL}{\mathsf{L}}
\newcommand{\cJ}{\mathsf{J}}
\newcommand{\wDelta}{\wtilde{\Delta}}
\newcommand{\cZ}{\mathcal{Z}}
\newcommand{\wLambda}{\wtilde{\Lambda}}
\newcommand{\supp}{\text{supp}}
\newcommand{\cH}{\mathcal{H}}

\newcommand{\cV}{\mathcal{V}}
\newcommand{\wU}{\wtilde{U}}
\newcommand{\fA}{\mathfrak{A}}
\newcommand{\cR}{\mathcal{R}}
\newcommand{\Ds}{D^s}
\newcommand{\wZ}{\wtilde{Z}}

\newcommand{\wphi}{\wtilde{\phi}}
\newcommand{\wb}{\wtilde{b}}
\newcommand{\wDs}{\wtilde{\Ds}}
\newcommand{\wf}{\wtilde{f}}
\newcommand{\wk}{\wtilde{k}}

\theoremstyle{definition}

\newtheorem{theorem}{Theorem}[section]
\newtheorem{corr}[theorem]{Corollary}
\newtheorem{lemma}[theorem]{Lemma}
\newtheorem{prop}[theorem]{Proposition}
\newtheorem{define}[theorem]{Definition}
\newtheorem{assump}[theorem]{Assumption}
\newtheorem*{remark}{Remark}

\numberwithin{equation}{section}

\begin{document}

\title{Laguerre and Jacobi analogues of the Warren process}
\author{Yi Sun}
\address{Y.S.: Department of Mathematics\\ Columbia University\\ 2990 Broadway\\ New York, NY 10027, USA}
\email{yisun@math.columbia.edu}

\contrib[with an appendix by]{Andrey Sarantsev}
\address{A.S.: Department of Statistics and Applied Probability, University of California, Santa Barbara}
\email{sarantsev@pstat.ucsb.edu}
\date{\today}

\begin{abstract}
We define Laguerre and Jacobi analogues of the Warren process.  That is, we construct local dynamics on a triangular array of particles so that the projections to each level recover the Laguerre and Jacobi eigenvalue processes of K\"onig-O'Connell and Doumerc and the fixed time distributions recover the joint distribution of eigenvalues in multilevel Laguerre and Jacobi random matrix ensembles.  Our techniques extend and generalize the framework of intertwining diffusions developed by Pal-Shkolnikov.  One consequence is the construction of particle systems with local interactions whose fixed time distribution recovers the hard edge of random matrix theory.  An appendix by Andrey Sarantsev establishes strong existence and uniqueness for solutions to SDER's satisfied by these processes.
\end{abstract}

\maketitle 

\tableofcontents

\section{Introduction}

The purpose of the present work is to construct and characterize Laguerre and Jacobi analogues of the Warren process.  These are stochastic dynamics on a triangular array of particles such that each particle evolves independently outside of interactions implemented by reflections of particles on level $n$ off particles on level $n - 1$.  We show that when started from Gibbs initial conditions, their projections to each level are Markovian and coincide in law with Laguerre and Jacobi analogues of Dyson Brownian motion introduced in \cite{KO} and \cite{Dou}, and their fixed time distributions recover the joint distribution of eigenvalues in multilevel Laguerre and Jacobi random matrix ensembles.  When projected to the left edge, our results yield a construction of an interacting particle system with local interactions whose fixed time distribution recovers the hard edge of random matrix theory.

Our motivation comes from the original construction of Warren in \cite{War} providing a coupling of Dyson Brownian motions via reflected Brownian motions.  In the context of interacting particle systems, multilevel intertwinings of Brownian particles have also appeared in the works \cite{WW, MOW, FF14, GS4}, while \cite{OCo, BC, GS2} study other intertwined processes whose single level projections are similar to Dyson Brownian motion but with more complicated multilevel interactions.  We note in particular that the paper \cite{FF14} gives a generalization of Warren's construction for Brownian particles with drift.  The Laguerre and Jacobi Warren processes we consider in this work differ from the previously mentioned processes in incorporating both non-Brownian diffusion terms and reflection of particles from one level on particles from lower levels.

Our method of proof proceeds via the framework of intertwining diffusions introduced by Pal-Shkolnikov in \cite{PS}; this framework gives a criterion for two diffusions to admit a Markov coupling in the spirit of the original work of Dynkin and Rogers-Pitman in \cite{Dyn, RP81}.  We extend their results by giving a general existence criterion for intertwining diffusions with non-Brownian diffusion term; our criterion involves deforming two conditions from the Brownian case.  The main results then follow from an application of our new criterion to the Laguerre and Jacobi eigenvalues processes, a general result on strong existence for stochastic differential equations with reflection on time-dependent barrier in one dimension, and a technical verification of a core criterion for the resulting process.

In the remainder of this introduction, we state our results more precisely and provide additional motivation and background. For the reader's convenience, all notations will be redefined in later sections.

\begin{remark}
After this work was completed and during the final preparation of this article, the author was made aware of the recent preprint \cite{AOW}, which obtains similar results using a different approach.
\end{remark}

\subsection{Dyson Brownian motion and the Warren process}

For each $n \geq 1$, let $X_n(t)$ be a standard Brownian motion in the space of $n \times n$ Hermitian matrices.  It was shown by Dyson in \cite{Dys} that the evolution of the ordered eigenvalues $\lambda^n_1(t) \leq \cdots \leq \lambda^n_n(t)$ of $X_n(t)$ is Markovian and solves the SDE
\[
d\lambda^n_i(t) = dB^n_i(t) + \sum_{j \neq i} \frac{1}{\lambda^n_i(t) - \lambda^n_j(t)} dt, \qquad 1 \leq i \leq n,
\]
where $B^n_i(t)$ are $n$ independent standard real Brownian motions.  The resulting process is known as \textit{Dyson Brownian motion}.  In \cite{War}, Warren introduced the following system of stochastic differential equations with reflection valued in the Gelfand-Tsetlin cone
\[
\GT_n := \{\mu^k_i \mid \mu^{k - 1}_{i - 1} \leq \mu^k_i \leq \mu^{k-1}_i\}_{1 \leq i \leq k, 1 \leq k \leq n}
\]
and given by
\begin{equation} \label{eq:war-def}
d\mu^k_i(t) = dB^k_i(t) + \frac{1}{2} dL^{k, +}_i(t) - \frac{1}{2} dL^{k, -}_i(t), \qquad 1 \leq i \leq k, 1 \leq k \leq n,
\end{equation}
where $B^k_i(t)$ are standard real Brownian motions, $L^{k, -}_i(t)$ is $0$ if $i = k$ and the local time of $\mu^k_i(t) - \mu^{k-1}_{i}(t)$ at $0$ otherwise, and $L^{k, +}_i(t)$ is $0$ if $i = 1$ and the local time of $\mu^k_i(t) - \mu^{k-1}_{i-1}(t)$ at $0$ otherwise.  He proved the following theorem, which shows that the unique weak solution to (\ref{eq:war-def}), known as the \textit{Warren process}, gives a coupling of Dyson Brownian motions with different numbers of particles.  Let $\Delta(\mu) := \prod_{i < j} (\mu_i - \mu_j)$ denote the Vandermonde determinant.

\begin{theorem}[{\cite[Section 4]{War}}] \label{thm:war-process}
The SDE (\ref{eq:war-def}) satisfies the following properties.
\begin{itemize}
\item[(a)] It admits a unique weak solution $\{\mu^k_i(t)\}$ when started at $0$ with entrance law
\[
(2\pi)^{-n} t^{-n^2/2} \Delta(\mu^n) \prod_{i = 1}^n e^{-(\mu^n_i)^2/2t} \prod_{k = 1}^n \prod_{i = 1}^k d\mu^k_i.
\]

\item[(b)] For each $k$, the projection of $\{\mu^k_i(t)\}$ to level $k$ is Markovian and coincides in law with Dyson Brownian motion started at $0$ with entrance law
\[
(2\pi)^{-k/2} t^{-k^2/2} \Delta(\mu^k)^2 \prod_{i = 1}^k e^{-(\mu^k_i)^2/2t} d\mu^k_i.
\]
\end{itemize}
\end{theorem}

\begin{remark}
While it might be tempting to conjecture that the Warren process is the joint evolution of eigenvalues of principal submatrices of $X_n(t)$, it was shown by Adler-Nordenstam-van~Moerbeke in \cite{ANvM14} that in general the joint evolution of eigenvalues of three principal submatrices is not Markovian.
\end{remark}

\subsection{Statement of the main results}

The purpose of the present work is to provide a generalization of the Warren process of Theorem \ref{thm:war-process} for Laguerre and Jacobi analogues of Dyson Brownian motion.  For the Laguerre case, it was shown in \cite{KO} that the largest $\min\{n, p\}$ singular values
\[
0 \leq \lambda^{(n)}_1(t) \leq \cdots \leq \lambda^{(n)}_{\min\{n, p\}}(t)
\]
of a $n \times p$ rectangular matrix of complex Brownian motions satisfy the stochastic differential equation
\[
d\lambda^{(n)}_i(t) = 2 \sqrt{\lambda^{(n)}_i(t)} dB^{(n)}_i(t) + 2 \Big(|n - p| + 1\Big) dt + \sum_{j \neq i} \frac{4 \lambda^{(n)}_i(t)}{\lambda^{(n)}_i(t) - \lambda^{(n)}_j(t)} dt,
\]
where $B^{(n)}_i(t)$ are standard Brownian motions, and may be started at $0$ with entrance law proportional to 
\[
\Delta(\lambda^{(n)})^2 \prod_{i = 1}^{\min\{n, p\}} (\lambda^{(n)}_i)^{|p - n|} e^{-\frac{\lambda^{(n)}_i}{2t}} d\lambda^{(n)}_i.
\]
This process is known as the Laguerre eigenvalues process of rank $p$ and level $n$.  Consider the system of SDE's with reflection
\begin{equation} \label{eq:lw-eq}
dl^n_i(t) = 2 \sqrt{l^n_i(t)} dB^n_i(t) + 2 \Big(p - n + 1\Big) dt + \frac{1}{2}dL^{n, +}_i(t) - \frac{1}{2}dL^{n, -}_i(t), 1 \leq i \leq \min\{n, p\}, 1 \leq n \leq m,
\end{equation}
where $B^n_i(t)$ are standard real Brownian motions, $L_i^{n, +}(t)$ is the local time at $0$ of $l_i^n(t) - l^{n - 1}_i(t)$ if $n > p$, the local time at $0$ of $l_i^n(t) - l^{n - 1}_{i - 1}(t)$ if $n \leq p$ and $i > 1$, and $0$ if $n \leq p$ and $i = 1$, and $L_i^{n, -}(t)$ is $0$ if $i = \min\{n, p\}$, the local time at $0$ of $l^{n - 1}_{i+1}(t) - l^n_i(t)$ if $n > p$ and $i < \min\{n, p\}$, and the local time at $0$ of $l^{n - 1}_i(t) - l^n_i(t)$ if $n > p$ and $i < \min\{n, p\}$. We say that an initial condition $\{l^n_i(0)\}$ for (\ref{eq:lw-eq}) is Gibbs if for any Borel $B$, we have
\[
\PP(\{l^n_i(0)\} \in B \mid l^m(0) = \lambda\} = \frac{(m - 1)! \cdots 1!}{(m - p - 1)_+! \cdots 1!} \frac{\vol(B)}{\Delta(\lambda) \prod_i \lambda_i^{(m - p)_+}}.
\]
Theorem \ref{thm:lwt}, our first main result, states that a solution to (\ref{eq:lw-eq}) from a Gibbs initial condition provides a simultaneous coupling of Laguerre eigenvalues processes of different levels.  It follows from Theorems \ref{thm:lag-war-inter} and \ref{thm:lag-war-exist} and Corollaries \ref{corr:lag-war} and \ref{corr:lag-entrance} in Section \ref{sec:lwp}.

\begin{theorem} \label{thm:lwt}
The SDER (\ref{eq:lw-eq}) admits a unique strong solution, known as the Laguerre Warren process, for any Gibbs initial condition.  This solution satisfies the following properties.

\begin{itemize}
\item[(a)] Its projection to level $n$ is Markovian and coincides in law with the Laguerre eigenvalues process of rank $p$ and level $n$.

\item[(b)] Its fixed time distribution at any $t > 0$ is Gibbs.

\item[(c)] It may be started from $l^n_i(0) = 0$ with entrance law proportional to 
\[
\Delta(l^m) \prod_{i = 1}^{\min\{m, p\}}(l^m_i)^{(p - m)_+} e^{-\frac{l_i^m}{2t}} \prod_{n = 1}^m \prod_{i = 1}^{\min\{n, p\}} dl^n_i.
\]
\end{itemize}
\end{theorem}

\begin{remark}
The fixed-time marginals of Theorem \ref{thm:lwt}(c) correspond to the joint distribution of eigenvalues at different levels of the Laguerre ensemble from random matrix theory.  More precisely, let $X(t)$ be a matrix of complex Brownian motions with $p$ columns and an infinite number of rows.  Letting $X_n(t)$ consist of its top $n$ rows, the distribution of Theorem \ref{thm:lwt}(c) is the joint distribution of the largest $\min\{n, p\}$ singular values of $X_n(t)$ for $1 \leq n \leq m$.
\end{remark}

For the Jacobi case, fix parameters $(p, q)$ and $n \leq p, q$, and let $N = p + q$.  In \cite{Dou}, it was shown that the singular values
\[
0 \leq \mu^{(n)}_1(t) \leq \cdots \leq \mu^{(n)}_n(t) \leq 1
\]
of the top left $n \times p$ submatrix of a Brownian motion on the space of unitary $N \times N$ matrices satisfy the stochastic differential equation
\[
d\mu^{(n)}_i(t) = 2 \sqrt{\mu^{(n)}_i(t) (1 - \mu^{(n)}_i(t))} dB^{(n)}_i(t) + 2 \Big((p - n + 1) + (p + q - 2 n + 2) \mu^{(n)}_i(t)\Big) dt + \sum_{j \neq i} \frac{4 \mu^{(n)}_i(1 - \mu^{(n)}_i(t))}{\mu^{(n)}_i(t) - \mu^{(n)}_j(t)} dt
\]
for standard real Brownian motions $B^{(n)}_i(t)$ and may be started with invariant measure proportional to 
\[
\Delta(\mu^{(n)})^2 \prod_{i = 1}^n (\mu^{(n)}_i)^{p - n} (1 - \mu^{(n)}_i)^{q - n} d\mu^{(n)}_i.
\]
This process is known as the Jacobi eigenvalues process with parameters $(p, q)$ and level $n$.  Consider the system of SDE's with reflection
\begin{multline} \label{eq:jw-eq}
dj^n_i(t) = 2 \sqrt{j^n_i(t) (1 - j^n_i(t))} dB^n_i(t) + 2\Big((p - n + 1) + (p + q - 2n + 2) j^n_i(t)\Big) dt\\ + \frac{1}{2} dL^{n, +}_i(t) - \frac{1}{2} dL^{n, -}_i(t), 1 \leq i \leq n, 1 \leq n \leq \min\{p, q\},
\end{multline}
where $B^n_i(t)$ are standard real Brownian motions, $L_i^{n, +}(t)$ is $0$ if $i = 1$ and the local time at $0$ of $j_i^n(t) - j^{n - 1}_{i - 1}(t)$ otherwise, and $L_i^{n, -}(t)$ is $0$ if $i = n$ and the local time at $0$ of $j^{n - 1}_i(t) - j^n_i(t)$ otherwise.  We say an initial condition $\{j^n_i(0)\}$ for \ref{eq:jw-eq} is Gibbs if for any Borel $B$ we have
\[
\PP(\{j^n_i(0)\} \in B \mid j^{\min\{p, q\}}(0) = \lambda) = (\min\{p, q\} - 1)! \cdots 1! \frac{\vol(B)}{\Delta(\lambda)}.
\]
Theorem \ref{thm:jwt}, our second main result, states that a solution to (\ref{eq:jw-eq}) from a Gibbs initial condition provides a simultaneous coupling of Jacobi eigenvalues processes of different levels.  It follows from Theorems \ref{thm:jac-war-inter} and \ref{thm:jac-war-exist} and Corollaries \ref{corr:jac-war} and \ref{corr:jac-inv} in Section \ref{sec:jwp}.
\begin{theorem} \label{thm:jwt}
The SDER (\ref{eq:jw-eq}) admits a unique strong solution, known as the Jacobi Warren process, for any Gibbs initial condition.  This solution satisfies the following properties.
\begin{itemize}
\item[(a)] Its projection to level $n$ is Markovian and coincides in law with the Jacobi eigenvalues process with parameters $(p, q)$ and level $n$.

\item[(b)] Its fixed time distribution at any $t > 0$ is Gibbs.

\item[(c)] It may be started with invariant measure proportional to
\[
\Delta(j^{\min\{p, q\}}) \prod_{i = 1}^{\min\{p, q\}} (j^{\min\{p, q\}}_i)^{p - \min\{p, q\}} (1 - j^{\min\{p, q\}}_i)^{q - \min\{p, q\}} \prod_{n = 1}^{\min\{p, q\}} \prod_{i = 1}^n dj^n_i.
\]
\end{itemize}
\end{theorem}

\begin{remark}
The invariant measure of Theorem \ref{thm:jwt} corresponds to the joint distribution of eigenvalues at different levels of the Jacobi ensemble from random matrix theory.  More precisely, let $X$ and $Y$ be matrices of standard complex Gaussians with infinitely many columns and $p$ rows and $q$ rows, respectively.  Let $X_n$ and $Y_n$ be the first $n$ columns of $X$ and $Y$; then the measure of Theorem \ref{thm:jwt} is the joint density of eigenvalues of $X_n^*X_n (X_n^* X_n + Y_n^*Y_n)^{-1}$ for $1 \leq n \leq \min\{p, q\}$. 
\end{remark}

\begin{remark}
In the appendix to this paper, authored by Andrey Sarantsev, it is shown that the SDER's (\ref{eq:lw-eq}) and (\ref{eq:jw-eq}) have a unique strong solution for arbitrary initial condition.  However, the properties of the solutions given in Theorems \ref{thm:lwt} and \ref{thm:jwt} hold only for Gibbs initial conditions.
\end{remark}

\subsection{Pal-Shkolnikov method of intertwining diffusions}

Our proof of Theorems \ref{thm:lwt} and \ref{thm:jwt} proceeds by studying the intertwining of Markov processes stemming from the semigroup criteria given by Dynkin in \cite{Dyn} and Rodgers-Pitman in \cite{RP81}.  In particular, we apply and extend the framework of intertwining diffusions introduced by Pal-Shkolnikov in \cite{PS}, which gives a version of intertwining adapted to our setting and shows that it applies to the Warren process of \cite{War} and the Whittaker process of \cite{OCo}.  We refer the reader to \cite{PS} for a detailed review of the literature on intertwinings. Let $X$ and $Y$ be diffusion processes with generators 
\begin{align*}
\AA^X &:= \frac{1}{2} \sum_{i = 1}^m \sum_{j = 1}^m a_{ij}(x) \partial_{x_i} \partial_{x_j} + \sum_{i = 1}^m b_i(x) \partial_{x_i}\\
\AA^Y &:= \frac{1}{2} \sum_{i = 1}^n \sum_{j = 1}^n \rho_{ij}(y) \partial_{y_i} \partial_{y_j} + \sum_{i = 1}^n \gamma_i(y) \partial_{y_i}
\end{align*}
on domains $\XX$ and $\YY$, respectively.  Let $L$ be a stochastic transition operator mapping $C_0(\XX)$ to $C_0(\YY)$.  Pal-Shkolnikov define an intertwining diffusion as follows.

{\renewcommand{\thetheorem}{\ref{def:ps-def}}
\begin{define}[{\cite[Definition 2]{PS}}] 
A process $Z = (Z_1, Z_2)$ is an intertwining of diffusions $X$ and $Y$ with link operator $L$ if:
\begin{itemize}
\item[(i)] $Z_1 \overset{d} = X$ and $Z_2 \overset{d}= Y$, where $\overset{d}=$ denotes equality in law, and 
\[
\EE[f(Z_1(0)) \mid Z_2(0) = y] = (Lf)(y),
\]
for all bounded Borel measurable functions $f$ on $D(y)$.

\item[(ii)] The transition semigroups $P_t$ and $Q_t$ of $Z_1$ and $Z_2$ are intertwined, meaning that $Q_t L = L P_t$ for all $t \geq 0$.

\item[(iii)] The process $Z_1$ is Markovian with respect to the joint filtration generated by $(Z_1, Z_2)$.

\item[(iv)] For any $s \geq 0$, conditional on $Z_2(s)$, the random variable $Z_1(s)$ is independent of $\{Z_2(u), 0 \leq u \leq s\}$ and is conditionally distributed according to $L$.
\end{itemize}
\end{define}\addtocounter{theorem}{-1}}

Let $D \subset \XX \times \YY$ be a domain with polyhedral closure, and let $D(y) := \{x \mid (x, y) \in D\}$.  We are interested in cases where $Z$ takes values in the domain $D$.  Suppose that $L$ is given by
\[
(Lf)(y) := \int_{D(y)} \Lambda(y, x) f(x) dx.
\]
for some non-negative probability density $\Lambda(y, x): D \to \RR$. Under several technical conditions given in Assumptions \ref{ass:proc}, \ref{ass:domain}, \ref{ass:link}, and \ref{ass:compat}, we give in Theorem \ref{thm:ps-ext} a new criterion for a diffusion $Z$ solving the SDER (\ref{eq:sder-ps}) with domain $D$ and reflection on moving boundaries to intertwine $X$ and $Y$ with link $L$.  Our result extends Pal-Shkolnikov's criterion in \cite[Theorem 3]{PS} to the case of non-Brownian diffusion terms and requires modifications to the conditions in \cite[Theorem 3]{PS} for this setting; in particular, conditions (\ref{eq:int-eq}) of Assumption \ref{ass:compat} requires the introduction of non-trivial terms which vanish in the case of a Brownian diffusion term.  Our proofs of Theorems \ref{thm:lwt} and \ref{thm:jwt} rest on an application of Theorem \ref{thm:ps-ext}, and we believe it may be of independent interest.

{\renewcommand{\thetheorem}{\ref{thm:ps-ext}}
\begin{theorem}
Suppose that $D, L, X, Y$ satisfy Assumptions \ref{ass:proc}, \ref{ass:domain}, \ref{ass:link}, and \ref{ass:compat} and that the SDER (\ref{eq:sder-ps}) has a weak solution $Z$ which is a regular Feller diffusion with generator $\AA^Z$ as defined in Assumption \ref{ass:proc}.  If the resulting process satisfies the initial condition
\[
\PP(Z_1(0) \in B \mid Z_2(0) = y) = \int_B \Lambda(y, x) dx \text{ for Borel $B \subset D(y)$},
\]
then $Z$ is an intertwining of $X$ and $Y$ with link $L$.
\end{theorem}
\addtocounter{theorem}{-1}}

\subsection{Organization of the paper}

The remainder of this paper is organized as follows.  In Section \ref{sec:lep-jep}, we define the Laguerre and Jacobi eigenvalues processes, their realizations via Doob $h$-transform, and their entrance law and invariant measure.  In Section \ref{sec:main}, we define the Laguerre and Jacobi Warren processes and prove Theorems \ref{thm:lag-war-inter}, \ref{thm:lag-war-exist}, \ref{thm:jac-war-inter}, and \ref{thm:jac-war-exist} showing that they give coupling of the Laguerre and Jacobi eigenvalues processes at different levels; in this section we make reference to tools developed in the next two sections.  In Section \ref{sec:sder}, we collect results from the literature to prove Theorem \ref{thm:exist-crit} giving a criterion for strong uniqueness and existence for stochastic differential equations with reflection on time-dependent boundaries with Holder regular diffusion term.  In Section \ref{sec:ps-ext}, we introduce the Pal-Shkolnikov framework of intertwining diffusions and prove Theorem \ref{thm:ps-ext} extending their results to general diffusion terms and reflection on moving boundaries and Proposition \ref{prop:geo-cond} allowing its application in a specific geometric context.  In Section \ref{sec:core-crit}, we prove Theorem \ref{thm:z-reg-cond} identifying a core for a certain class of diffusion processes.  Sections \ref{sec:sder}, \ref{sec:ps-ext}, and \ref{sec:core-crit} develop tools which we apply in Section \ref{sec:main} and may be read independently.  Appendix \ref{sec:appendix}, by Andrey Sarantsev, provides a proof that the SDER's (\ref{eq:lw-eq}) and (\ref{eq:jw-eq}) defining the Laguerre and Jacobi Warren processes have unique strong solutions for arbitrary initial conditions.

\subsection*{Acknowledgments}

The author thanks H. Altman, T. Assiotis, A. Borodin, V. Gorin, K. Ramanan, and M. Shkolnikov for helpful discussions and the organizers of the CMI workshop on Random Polymers and Algebraic Combinatorics in May 2015 for providing the environment where this work was initiated. Y.~S. was partially supported by a NSF Graduate Research Fellowship (NSF Grant 1122374), a Junior Fellow award from the Simons Foundation, and NSF Grants DMS-1637087 and DMS-1701654.

\section{The Laguerre and Jacobi eigenvalues processes} \label{sec:lep-jep}

In this section we introduce the Laguerre and Jacobi eigenvalues processes as Doob $h$-transforms of independent squared Bessel and univariate Jacobi processes and explain their realization as eigenvalues of certain matrix-valued processes.

\subsection{The Laguerre eigenvalues process} \label{sec:lep}

Let $A(t)$ be an infinite matrix of complex Brownian motions with $p$ columns and initial condition $A(0)$.  Denote by $A_n(t)$ its top $n$ rows.  The \textit{complex Wishart process} of rank $p$ and level $n$ is the process valued in $n \times n$ matrices given by $M_n(t) = A_n(t) A_n(t)^*$.  The \textit{Laguerre eigenvalues process} of rank $p$ and level $n$ consists of the $\min\{n, p\}$ largest eigenvalues
\[
0 \leq \lambda_1^{(n)}(t) \leq \cdots \leq \lambda_{\min\{n, p\}}^{(n)}(t)
\]
of $M_n(t)$.  This process was introduced in the real case in \cite{Bru} and studied in the complex case in \cite{KO}.  It was studied in \cite{Dem07} and may be identified with the radial Dunkl process of type B studied in \cite{Dem09}.  The singular SDE it satisfies was analyzed in detail in \cite{GM2014}.  We collect some of its properties below.

\begin{prop}[{\cite[Theorem 1]{Bru}, \cite{KO}, \cite[Lemma 1]{Dem07}, and \cite[Corollary 6.2]{GM2014}}] \label{prop:lep}
For any initial condition $\lambda^{(n)}(0)$, the Laguerre eigenvalues process $\lambda^{(n)}(t)$ of rank $p$ and level $n$ satisfies the following:
\begin{enumerate}
\item[(a)] The $\lambda^n(t)$ are Markovian, have all $\lambda^{(n)}_i(t)$ positive and distinct for $t > 0$, and form a diffusion with generator
\[
\cL^p_n := \sum_{i = 1}^{\min\{n, p\}} 2\lambda^{(n)}_i \partial_i^2 + \sum_{i = 1}^{\min\{n, p\}} 2 \Big(|n - p| + 1\Big)\partial_i + \sum_{i = 1}^{\min\{n, p\}} \sum_{j \neq i} \frac{4\lambda^{(n)}_i}{\lambda^{(n)}_i - \lambda^{(n)}_j} \partial_i.
\]
Equivalently, $\lambda^{(n)}(t)$ is a solution to the SDE
\[
d\lambda^{(n)}_i(t) = 2 \sqrt{\lambda^{(n)}_i(t)} dB^{(n)}_i(t) + 2 \Big(|n - p| + 1\Big) dt + \sum_{j \neq i} \frac{4\lambda^{(n)}_i(t)}{\lambda^{(n)}_i(t) - \lambda^{(n)}_j(t)} dt
\]
for standard Brownian motions $B^{(n)}_i(t)$.

\item[(b)] When started at $\lambda^{(n)}(0) = 0$, the Laguerre eigenvalues process has entrance law proportional to
\[
\Delta(\lambda^{(n)})^2 \prod_{i = 1}^{\min\{p, n\}} (\lambda^{(n)}_i)^{|p - n|} e^{-\frac{\lambda^{(n)}_i}{2t}} d\lambda^{(n)}_i.
\]
\end{enumerate}
\end{prop}

In \cite{KO}, it was observed that the Laguerre eigenvalues process may be constructed from  independent squared Bessel processes conditioned never to intersect.  More precisely, let $\BESQ^d(t)$ denote the squared Bessel process of dimension $d$; it solves the stochastic differential equation 
\[
d\BESQ^d(t) = 2 \sqrt{\BESQ^d(t)} dB(t) + d\,dt
\]
and has generator $2x \partial + d\, \partial$.  Let $\Delta(\lambda^{(n)})$ denote the Vandermonde determinant $\Delta(\lambda^{(n)}) := \prod_{i < j} (\lambda^{(n)}_i - \lambda^{(n)}_j)$, and denote the generator of $\min\{n, p\}$ independent squared Bessel processes of dimension $2\Big(|n - p| + 1\Big)$ by
\[
L^p_n := \sum_{i = 1}^{\min\{n, p\}} 2 \lambda_i^{(n)} \partial_i^2 + \sum_{i = 1}^{\min\{n, p\}} 2\Big(|n - p| + 1\Big) \partial_i.
\]
In \cite{KO}, K\"onig-O'Connell realized the Laguerre eigenvalues process as a Doob $h$-transform as follows.

\begin{prop}[\cite{KO}] \label{prop:ko}
The function $\Delta(\lambda^{(n)})$ is harmonic with respect to $L_n^p$, and the Doob $h$-transform of $\min\{n, p\}$ independent squared Bessel processes of dimension $2\Big(|n - p| + 1\Big)$ with respect to $\Delta(\lambda^{(n)})$ is the Laguerre eigenvalues process with rank $p$ and level $n$.
\end{prop}

\subsection{The Jacobi eigenvalues process}

Fix parameters $(p, q)$ and $n \leq p, q$, and let $N = p + q$.  Let $U_N(t)$ be a Brownian motion on the space of unitary $N \times N$ matrices, and let $X^{p, q}_n(t)$ denote the top left $n \times p$ submatrix of $U_N(t)$.  In \cite{Dou}, the \textit{complex matrix Jacobi process} with parameters $(p, q)$ and level $n$ was defined to be the process valued in $n \times n$ matrices given by 
\[
J^{p, q}_n(t) := X^{p, q}_n(t) X^{p, q}_n(t)^*.
\]
The \textit{Jacobi eigenvalues process} with parameters $(p, q)$ and level $n$ consists of the eigenvalues
\[
0 \leq \mu^{(n)}_1(t) \leq \cdots \leq \mu^{(n)}_n(t) \leq 1
\]
of $J^{p, q}_n(t)$.  This process was introduced in \cite{Dou} and analyzed in detail in \cite{Dem10, GM2013, GM2014}.  We collect some of its properties below.

\begin{prop}[{\cite[Proposition 9.4.7]{Dou}, \cite{Dem10}, \cite[Corollary 9]{GM2013}, \cite[Corollary 6.7]{GM2014}}] \label{prop:jep}
The Jacobi eigenvalues process $\mu^{(n)}(t)$ with parameters $(p, q)$ and level $n$ satisfies the following:
\begin{enumerate}
\item[(a)] When started at any initial condition $\mu^{(n)}(0)$, the process $\mu^{(n)}(t)$ is Markovian, has all $\mu^{(n)}_i(t)$ distinct in $(0, 1)$ for $t > 0$, and forms a diffusion with generator
\[
\cJ^{p, q}_n := 2 \sum_{i = 1}^n \mu^{(n)}_i (1 - \mu^{(n)}_i) \partial_i^2 + 2 \sum_{i = 1}^n \Big((p - n + 1) + (p + q - 2n + 2) \mu^{(n)}_i\Big) \partial_i + \sum_{i = 1}^n \sum_{j \neq i} \frac{4 \mu^{(n)}_i (1 - \mu^{(n)}_i)}{\mu^{(n)}_i - \mu^{(n)}_j} \partial_i.
\]
Equivalently, $\mu^{(n)}(t)$ is a solution to the SDE
\[
d\mu^{(n)}_i(t) = 2 \sqrt{\mu^{(n)}_i(t)(1 - \mu^{(n)}_i(t))} dB^{(n)}_i(t) + 2 \Big((p - n + 1) + (p + q - 2n + 2)\mu^{(n)}_i(t)\Big) dt + \sum_{j \neq i} \frac{4 \mu^{(n)}_i(t)(1 - \mu^{(n)}_i(t))}{\mu^{(n)}_i(t) - \mu^{(n)}_j(t)} dt
\]
for standard Brownian motions $B^{(n)}_i(t)$.

\item[(b)] The Jacobi eigenvalues process is a diffusion whose invariant measure on $[0, 1]^n$ has density proportional to
\begin{equation} \label{eq:jac-inv}
\Delta(\mu^{(n)})^2 \prod_{i = 1}^n (\mu^{(n)}_i)^{p - n} (1 - \mu^{(n)}_i)^{q - n} d\mu^{(n)}_i.
\end{equation}
\end{enumerate}
\end{prop}

\begin{remark}
The density of (\ref{eq:jac-inv}) admits an alternate realization as follows.  It is the probability density of the eigenvalues of $X^*X(X^*X + Y^*Y)^{-1}$, where $X$ and $Y$ are matrices of standard complex Gaussians of size $p \times n$ and $q \times n$, respectively.
\end{remark}

In \cite{Dou}, it was observed that the Jacobi eigenvalues process may be constructed from independent univariate Jacobi processes conditioned to never intersect.  More precisely, let $\JAC^{a, b}(t)$ denote the univariate Jacobi process with parameters $(a, b)$.  It is the diffusion which solves the stochastic differential equation
\[
d\JAC^{a, b}(t) = 2 \sqrt{\JAC^{a, b}(t) (1 - \JAC^{a, b}(t))} dB(t) + 2 \Big(a + 1 - (a + b + 2)\JAC^{a, b}(t)\Big) dt
\]
and has generator $2 x (1 - x) \partial^2 + 2 (a + 1 - (a + b + 2)x)\partial$.  Denote the generator of $n$ independent univariate Jacobi processes by
\[
J^{a, b}_n := 2 \sum_{i = 1}^n \mu^{(n)}_i (1 - \mu^{(n)}_i) \partial_{i}^2 + 2 \sum_{i = 1}^n (a + 1 - (a + b + 2)\mu^{(n)}_i)\partial_{i}.
\]
In \cite{Dou}, Doumerc realized the Jacobi eigenvalues process via Doob $h$-transform as follows.

\begin{prop}[{\cite[Proposition 9.4.7]{Dou}}] \label{prop:jac-h}
For $p, q, n$ with $n \leq \min\{p, q\}$, the function $\Delta(\mu^{(n)})$ is an eigenfunction of $J^{p - n, q - n}_n$ with eigenvalue $\frac{n (n - 1)(3p + 3q - 4n + 2)}{3}$, and the Doob $h$-transform of $n$ independent univariate Jacobi processes with parameters $(p - n, q - n)$ is the Jacobi eigenvalues process of with parameters $(p, q)$ and level $n$.
\end{prop}

\section{The Laguerre and Jacobi Warren processes via intertwining diffusions} \label{sec:main}

In this section, we define the Laguerre and Jacobi Warren processes and show that, when started at Gibbs initial conditions, their projections to each level recover the Laguerre and Jacobi eigenvalues processes.  These are the main results of this paper.

\subsection{The Laguerre Warren process} \label{sec:lwp}

In this section we define the Laguerre Warren process as the unique weak solution to a certain stochastic differential equation with reflection.  This construction is the Laguerre analogue of the Warren process defined using Brownian motions in \cite{War}.  For any positive integer $m$, denote the positive Gelfand-Tsetlin cone of rank $p$ and level $m$ by
\begin{multline*}
\GT_{m, p} := \{l_i^n \mid 0 \leq l_{i-1}^{n - 1} \leq l^n_i \leq l^{n - 1}_i \text{ if } n \leq p,\\ 0 \leq l_i^{n-1} \leq l^n_i \leq l^{n-1}_{i+1} \text{ if } n > p, \text{ and } l_1^m < \cdots < l^m_{\min\{m, p\}}\}_{1 \leq i \leq \min\{n, p\}, 1 \leq n \leq m
}
\end{multline*}
so that there are $\min\{n, p\}$ particles $l^n_{1} \leq \cdots \leq l^n_{\min\{n, p\}}$ on level $n$. Define the positive Gelfand-Tsetlin polytope of rank $p$ and level $m$ subordinate to $\lambda = (\lambda_1 < \cdots < \lambda_{\min\{m, p\}})$ by
\[
\GT_{m, p}(\lambda) := \{\{l^n_i\} \in \GT_{m, p} \mid l^m = \lambda\}.
\]
For a set of variables $x_1, \ldots, x_m$, we define the modified Vandermonde determinant by
\[
\wDelta_{m, p}(x) := \Delta(x) \prod_i x_i^{(m - p)_+}.
\]
We say that a probability distribution $\nu$ on $\GT_{m, p}$ is \textit{Gibbs} if for any $(\lambda_1 < \cdots < \lambda_{m})$ and any Borel $B \subset \GT_{m, p}(\lambda)$, we have that 
\begin{equation} \label{eq:l-gibbs-def}
\PP_\nu(\{l^n_i\} \in B \mid l^m = \lambda) = \frac{(m - 1)! \cdots 1!}{(m - p - 1)_+! \cdots 1!} \frac{\vol(B)}{\wDelta_{m, p}(\lambda)}.
\end{equation}

\begin{remark}
By repeated application of the Dixon-Anderson integral stated in \cite{Dix05, And91} and surveyed in \cite[Theorem 2.1]{Rai10}, we obtain for $n > p$ that
\[
\int_{x_1 \leq y_1 \leq \cdots \leq x_p \leq y_p} \frac{\wDelta_{n - 1, p}(x)}{\wDelta_{n, p}(y)} dx = \frac{(n - p - 1)!}{(n - 1)!},
\]
which implies by induction that
\[
\vol(\GT_{m, p}(\lambda)) = \wDelta_{m, p}(\lambda) \frac{(m - p - 1)_+! \cdots 1!}{(m - 1)! \cdots 1!}
\]
and therefore that (\ref{eq:l-gibbs-def}) defines a valid probability density.
\end{remark}

Consider the system of stochastic differential equations with reflection with domain $\GT_{m, p}$ given by
\begin{equation} \label{eq:lag-war-def}
dl_i^{n}(t) = 2 \sqrt{l_i^n(t)} dB^n_i(t) + 2(p - n + 1) dt + \frac{1}{2}dL_i^{n, +}(t) - \frac{1}{2}dL_i^{n, -}(t), \text{ $1 \leq i \leq \min\{n, p\}, 1 \leq n \leq m$},
\end{equation}
where $B^n_i(t)$ are standard real Brownian motions, $L_i^{n, +}(t)$ is the local time at $0$ of $l_i^n(t) - l^{n - 1}_i(t)$ if $n > p$, the local time at $0$ of $l_i^n(t) - l^{n - 1}_{i - 1}(t)$ if $n \leq p$ and $i > 1$, and $0$ if $n \leq p$ and $i = 1$, and $L_i^{n, -}(t)$ is $0$ if $i = \min\{n, p\}$, the local time at $0$ of $l^{n - 1}_{i+1}(t) - l^n_i(t)$ if $n > p$ and $i < \min\{n, p\}$, and the local time at $0$ of $l^{n - 1}_i(t) - l^n_i(t)$ if $n > p$ and $i < \min\{n, p\}$.

\begin{remark}
Informally, a solution to (\ref{eq:lag-war-def}) may be described as follows.  At level $n$, it consists of independent squared Bessel processes with dimension $2(p - n + 1)$ interlacing with and reflecting off the processes at level $n - 1$.  This differs from the Warren process by replacing Brownian motions by squared Bessel processes and introducing different parameters on each level.
\end{remark}

The following two theorems, whose proofs are given in Section \ref{sec:lag-proof}, are the first of our main results.  We show that (\ref{eq:lag-war-def}) admits a unique strong solution for any Gibbs initial condition and that this solution provides a coupling of Laguerre eigenvalues processes on each level.  We call the resulting process the \textit{Laguerre Warren process}.  The key technical ingredient is our extension of the theory of intertwining diffusions of \cite{PS} given in Section \ref{sec:ps-ext}.  

\begin{theorem} \label{thm:lag-war-inter}
For any Gibbs initial condition $\{\lambda^n_i(0)\}$, if the SDER (\ref{eq:lag-war-def}) admits a unique weak solution which is a regular Feller process in the sense of Assumption \ref{ass:proc}, then for $1 < n \leq m$, its projection to levels $1, \ldots, n$ is Markovian and is an intertwining of the Laguerre eigenvalues processes of rank $p$ and level $n$ and the Laguerre Warren process of rank $p$ and level $n - 1$ in the sense of Definition \ref{def:ps-def}.
\end{theorem}

\begin{theorem} \label{thm:lag-war-exist}
For any Gibbs initial condition $\{l^n_i(0)\}$, the SDER (\ref{eq:lag-war-def}) admits a unique strong solution $\{l^n_i(t)\}_{1 \leq i \leq \min\{n, p\}, 1 \leq n \leq m}$ which is a regular Feller process in the sense of Assumption \ref{ass:proc} and which we call the Laguerre Warren process.
\end{theorem}

\begin{corr} \label{corr:lag-war}
For $1 \leq n \leq m$, the projection of the Laguerre Warren process to level $n$ is Markovian and coincides in law with the Laguerre eigenvalues process of rank $p$ and level $n$.
\end{corr}
\begin{proof}
This follows by combining Theorem \ref{thm:lag-war-exist}, Theorem \ref{thm:lag-war-inter}, and the definition of intertwining.
\end{proof}

\begin{corr} \label{corr:lag-entrance}
The Laguerre Warren process may be started from $l^n_i(0) = 0$ with entrance law proportional to
\[
\Delta(l^m) \prod_{i = 1}^{\min\{m, p\}}(l^m_i)^{(p - m)_+} e^{-\frac{l_i^m}{2t}} \prod_{n = 1}^m \prod_{i = 1}^{\min\{n, p\}} dl^n_i.
\]
\end{corr}
\begin{proof}
This measure is the Gibbs measure associated to the entrance law of the Laguerre eigenvalues process from Proposition \ref{prop:lep}(b).  By Theorem \ref{thm:lag-war-inter}, the Laguerre Warren process preserves Gibbs measures and projects to the Laguerre eigenvalues process of rank $p$ and level $m$, hence the claimed measure is a valid entrance law because the measure of Proposition \ref{prop:lep}(b) is.
\end{proof}

\begin{remark}
By analogy with the Wigner case studied in \cite{ANvM14}, we do not expect the Laguerre Warren process to coincide with the process of singular values of slices of the matrix $A(t)$ from Section \ref{sec:lep}.
\end{remark}

\begin{remark}
Consider the projection of $\{l^n_i(t)\}$ to the smallest particle on the first $p$ levels.  The result is a Markovian particle process on
\[
l^1_1(t) \geq l^2_1(t) \geq \cdots \geq l^p_1(t) \geq 0
\]
solving the system of stochastic differential equations
\begin{equation} \label{eq:left-edge-sde}
dl^n_1(t) = 2 \sqrt{l^n_1(t)} dB^n_1(t) + 2(p - n + 1)dt - \frac{1}{2}dL_1^{n, -}(t),
\end{equation}
where $L_1^{n, -}(t)$ is the local time at $0$ of $l^{n - 1}_1(t) - l^n_1(t)$.  By Corollary \ref{corr:lag-war}, for each level $n$ and time $t$, the smallest particle $l^n_1(t)$ on each level is equal in law to the smallest particle of a Laguerre eigenvalues process of rank $p$ and level $n$.  Therefore, for $n \leq p$, the fixed time distribution at $t = 1$ has the law of the smallest particle of a Laguerre random matrix of shape $p \times n$, which is the hard edge of random matrix theory.  On the other hand, (\ref{eq:left-edge-sde}) only involves interactions between neighboring particles, meaning that this construction gives a method to produce the hard edge from a process with purely local interactions.
\end{remark}

\subsection{Proof of Theorem \ref{thm:lag-war-inter}} \label{sec:lag-proof}

We prove by induction on $n$ that $\{l^k_i(t)\}_{1 \leq k \leq n - 1}$ is intertwined with the Laguerre eigenvalues process of rank $p$ and level $n$.  The base case is trivial.  During the proof, we adopt the notations of Sections \ref{sec:ps-ext} and \ref{sec:core-crit}.  In particular, we consider the domains $\XX := \GT_{n - 1, p}$, $\YY := \{0 \leq y_1 \leq \cdots \leq y_{\min\{n, p\}}\}$, and
\[
D := \begin{cases} \{(\{x^k\}, y) \mid \{x^k\} \in \XX, y \in \YY, 0 \leq y_1 \leq x_1^{n-1} \leq y_2 \cdots \leq x_{n-1}^{n-1} \leq y_n\} & n \leq p \\
\{(\{x^k\}, y) \mid \{x^k\} \in \XX, y \in \YY, 0 \leq x_1^{n-1} \leq y_1 \leq \cdots \leq y_{\min\{n, p\}}\} & n > p. \end{cases}
\]
On these domains, we let $X(t)$ denote $\{l^k_i(t)\}_{1 \leq k \leq n - 1}$, $Y(t)$ denote the Laguerre eigenvalues process of rank $p$ and level $n$, and $Z(t)$ denote $\{l^k_i(t)\}_{1 \leq k \leq n}$, the weak solution to (\ref{eq:lag-war-def}) stopped at the non-smooth parts of the boundary.  Let
\begin{align*}
\AA^X &:= \sum_{k = 1}^{n - 1} \sum_{i = 1}^{\min\{k, p\}} \Big(2 x^k_i \partial^2_{x^k_i} + 2 (p - k + 1)\partial_{x^k_i}\Big)\\
\AA^Z &:= \AA^X + \sum_{i = 1}^{\min\{n, p\}} \Big(2 y_i \partial_{y_i}^2 + 2 (p - n + 1) \partial_{y_i}\Big)
\end{align*}
denote the generators of $X(t)$ and $Z(t)$, and let $\AA^Y := \cL^{p}_n$ denote the generator of $Y(t)$.  Define the kernel
\[
\Lambda(y, x) := \frac{(n - 1)! \cdots 1!}{(n - p - 1)_+! \cdots 1!} \wDelta_{n, p}(y)^{-1}.
\]
We will verify the hypotheses of Theorem \ref{thm:ps-ext} for $D$ and $\Lambda$. 
\begin{itemize}
\item \textbf{Assumption \ref{ass:proc}:} Note that $Y(t)$ is Feller as a projection under the eigenvalue map of the Wishart process, which is Feller by \cite[Lemma 2]{Bru}; it is regular as the solution to the SDE for the Laguerre eigenvalues process.  In addition, $X(t)$ is regular Feller by the given.

\item \textbf{Assumption \ref{ass:domain}:} All three points are evident from the parametrization $x^k_i = \xi^k_i x^{k+1}_i + (1 - \xi^{k_i}) x^{k + 1}_{i + 1}$ for $k < p$ and $x^k_i = \xi^k_i x^{k+1}_{i-1} + (1 - \xi^{k_i}) x^{k + 1}_{i}$ for $k \geq p$, where we adopt the conventions that $x^n = y$ and $x^k_0 = 0$.

\item \textbf{Assumption \ref{ass:link}:} Points (a), (b), and (c) follow from the definition of $\Lambda(y, x)$.  For (d), note that $\cL^{p, y}_n = \Delta(y)^{-1} \circ L^{p, y}_n \circ \Delta(y)$ by Proposition \ref{prop:ko}, so we observe directly that $\cL^{p, y}_n\Lambda(y, x) = 0$ is continuous and bounded.  For $n > p$, noting that
\[
\Big(2y_i \partial_{y_i}^2 + 2 (|p - n| + 1) \partial_{y_i}\Big) y_i^{p - n} = \Big(2(p - n)(p - n - 1) + 2(n - p + 1)(p - n)\Big) y_i^{p - n - 1} = 0,
\]
we see that
\[
\cL^{p, y}_n \Lambda(y, x) = \frac{(n - 1)! \cdots 1!}{(n - p - 1)_+! \cdots 1!} \Delta(y)^{-1} L^{p, y}_n \prod_i y_i^{p - n} = 0, 
\]
which is again continuous and bounded.  This establishes (d).

\item \textbf{Assumption \ref{ass:compat}:}  Points (a) and (d) are trivial because $Y(t)$ has no boundary reflection. Point (b) follows by definition of the kernel.  For (c), we apply Proposition \ref{prop:geo-cond} with 
\begin{align*}
\AA^{X_1} &= \cL^{p, x^1}_{n - 1} = \Delta(x^1)^{-1} \circ L^{p, x^1}_{n-1} \circ \Delta(x^1)\\
&= \sum_{i = 1}^{\min\{n - 1, p\}} 2 x^1_i \partial_{x^1_i}^2 + \sum_{i = 1}^{\min\{n - 1, p\}} 2 \Big(|n - p - 1| + 2\Big) \partial_{x^1_i} + \sum_{i = 1}^{\min\{n - 1, p\}} \sum_{j \neq i} \frac{4 x^1_i}{x^1_i - x^1_j} \partial_{x^1_i}
\end{align*}
being the generator of the Laguerre eigenvalues process of rank $p$ and level $n - 1$, and the two kernels
\[
\Lambda_1(y, x^1) = \frac{(n - 1)!}{(n - p - 1)_+!} \frac{\wDelta_{n - 1, p}(x^1)}{\wDelta_{n, p}(y)} \qquad \text{ and } \qquad \Lambda_2(x^1, x^2) = \frac{(n - 2)! \cdots 1!}{(n - p - 2)_+! \cdots 1!} \wDelta_{n - 1, p}(x^1)^{-1}.
\]
We now check the hypotheses of Proposition \ref{prop:geo-cond}.
\begin{itemize}
\item[(a)] This follows by the inductive hypothesis.
\item[(b)] We have that $\cL^{2|p - n| + 2, y}_n \Lambda_1 = 0$ by the same computation as in Assumption \ref{ass:link}.  Notice now that 
\[
(\AA^{X_1})^*\Lambda_1 = \frac{(n - 1)!}{(n - p - 1)_+!} \frac{\Delta(x^1)}{\wDelta_{n, p}(y)} (L^{p, x^1}_{n - 1})^* \prod_i (x^1_i)^{(n - p - 1)_+} = 0,
\]
where for $n \geq p + 1$ we use the univariate computation
\[
\Big(2 x^1_i \partial_{x^1_i}^2 + 2(|p - n + 1| + 1) \partial_{x^1_i}\Big)^* (x^1_i)^{n - p - 1} = (2(n - p)(n - p - 1) - 2(n - p - 1)(n - p)) (x^1_i)^{n - p - 2} = 0.
\]

\item[(c)] If $n \leq p$, faces of $\partial D_1(y)$ take the form $\{x_i^1 = y_i\}$ or $\{x_i^1 = y_{i+1}\}$ for $1 \leq i \leq n - 1$.  On $\{x_i^1 = y_i\}$, we have $\eta_k^1 = -1_i$ and $\langle \Psi^{j, 1}_k, \eta_k^1 \rangle = - \delta_{ij}$, so the equality becomes
\begin{multline*}
- 2(|n - p - 1| + 1) - \sum_{j \neq i}\frac{4x^1_i}{x^1_i - x^1_j} + 2 + \sum_{j \neq i}\frac{4x^1_i}{x^1_i - x^1_j} = - 2 (|n - p| + 1) - \sum_{j \neq i} \frac{4 y_i}{y_i - y_j} + \sum_{j \neq i} \frac{4 y_i}{y_i - y_j}.
\end{multline*}
Similarly, on $\{x_i = y_{i + 1}\}$, we have $\eta_k^1 = 1_i$ and $\langle \Psi^{j, 1}_k, \eta_k^1\rangle = \delta_{i + 1, j}$, so the equality becomes
\[
2(|n - p - 1| + 1) + \sum_{j \neq i}\frac{4x^1_i}{x^1_i - x^1_j} - 2 - \sum_{j \neq i}\frac{4x^1_i}{x^1_i - x^1_j} = 2 (|n - p| + 1) + \sum_{j \neq i + 1} \frac{4 y_{i + 1}}{y_{i + 1} - y_j} - \sum_{j \neq i + 1} \frac{4 y_{i + 1}}{y_{i + 1} - y_j}.
\]
If $n > p$, faces of $\partial D_1(y)$ take the form $\{x_i^1 = y_i\}$ or $\{x_i^1 = y_{i-1}\}$ for $1 \leq i \leq p$.  On $\{x_i^1 = y_i\}$, we have $\eta_k^1 = 1_i$ and $\langle \Psi^{j, 1}_k, \eta_k^1 \rangle = \delta_{ij}$, so the desired equality becomes
\begin{multline*}
2(|n - p - 1| + 1) + \sum_{j \neq i}\frac{4x^1_i}{x^1_i - x^1_j} - 2 - \sum_{j \neq i}\frac{4x^1_i}{x^1_i - x^1_j} - 4 (n - p - 1)\\  =  2 (|n - p| + 1) + \sum_{j \neq i} \frac{4 y_i}{y_i - y_j} - \sum_{j \neq i} \frac{4 y_i}{y_i - y_j} - 4 (n - p).
\end{multline*}
On $\{x_i^1 = y_{i - 1}\}$, we have $\eta_k^1 = -1_i$ and $\langle\Psi^{j, 1}_k, \eta_k^1\rangle = - \delta_{i - 1, j}$, so the desired equality becomes
\begin{multline*}
- 2(|n - p - 1| + 1) - \sum_{j \neq i}\frac{4x^1_i}{x^1_i - x^1_j} + 2 + \sum_{j \neq i}\frac{4x^1_i}{x^1_i - x^1_j} + 4 (n - p - 1)\\ = - 2 (|n - p| + 1) - \sum_{j \neq {i - 1}} \frac{4 y_{i - 1}}{y_{i - 1} - y_j} + \sum_{j \neq i} \frac{4 y_{i - 1}}{y_{i - 1} - y_j} + 4 (n - p).
\end{multline*}

\item[(d)] If $n \leq p$, the desired reduces to $- 4 x_i^1 = - 4y_i$ on $\{x_i = y_i\}$ and $4 x_i^1 = 4 y_{i+1}$ on $\{x_i = y_{i+1}\}$.  If $n > p$, it reduces to $4x_i = 4y_i$ on $\{x_i = y_i\}$ and $-4x_i = -4y_{i-1}$ on $\{x_i = y_{i-1}\}$.

\item[(e)] Such faces take the form $\{x^a_b = x^{a - 1}_c\}$ for $c \in \{b - 1, b, b + 1\}$ and $a < n$.  Therefore, we have $(\Psi^i_k)_{a, b} = (\Psi^i_k)_{a - 1, c}$ for all $i$, and $\eta_k$ is proportional to $1_{a, b} - 1_{a - 1, c}$, giving the desired.

\item[(f)] Notice that $\partial D(y)_{\wk}$ takes the form $\{x_a = y_b\}$ for $b \in \{a - 1, a, a + 1\}$.  In this case, $\eta_k$ is proportional to $1_a$, and $(\Psi^i_k)_a = \delta_{ib}$, so $\langle \Psi^i_k, \eta_k\rangle = \delta_{ib}$, as desired.
\end{itemize}
Having checked all the hypotheses of Proposition \ref{prop:geo-cond}, we deduce Assumption \ref{ass:compat}(c), as desired.
\end{itemize}
We conclude that Assumptions \ref{ass:proc}, \ref{ass:domain}, \ref{ass:link}, and \ref{ass:compat} hold for $(X(t), Y(t))$.  Notice now that the SDER of Theorem \ref{thm:ps-ext} is given by (\ref{eq:lag-war-def}) and by the given has a weak solution $Z(t)$ which is a regular Feller diffusion.  Noting that $\partial_{y_i} \log \Lambda(y, x) = - \sum_{j \neq i} \frac{1}{y_i - y_j} - \frac{(n - p)_+}{y_i}$, its Markov generator (\ref{eq:z-gen-def}) is
\begin{align*} 
\AA^Z &:= \cL^{p, x}_{n - 1} + \sum_{i = 1}^{\min\{n, p\}} \Big(2(|n - p| + 1) + \sum_{j \neq i} \frac{4y_i}{y_i - y_j}\Big) \partial_{y_i} - \sum_{i = 1}^{\min\{n, p\}} \Big(\sum_{j \neq i} \frac{4y_i}{y_i - y_j} + 4(n - p)_+\Big) \partial_{y_i}\\
&= \cL^{p, x}_{n - 1} + \sum_{i = 1}^{\min\{n, p\}} 2y_i \partial_{y_i}^2 + \sum_{i = 1}^{\min\{n, p\}}2(p - n + 1) \partial_{y_i}
\end{align*}
with the claimed Neumann boundary conditions on $\partial D(y)_k$.  We may therefore apply Theorem \ref{thm:ps-ext} to conclude that $(X(t), Y(t))$ are intertwined by $Z(t)$, meaning that $(\{l^k(t)\}_{1 \leq k \leq n -1}, l^n(t))$ is an intertwining of the Laguerre Warren process of rank $p$ and level $n - 1$ and the Laguerre eigenvalue process of rank $p$ and level $n$, completing the proof.

\begin{remark}
The single-level version $\Lambda_1$ of our kernel is the Dixon-Anderson kernel.  It was shown via random corank $1$ projections to correspond to the Gibbs property of fixed time distributions of random matrix eigenvalues in \cite{FR, FN} and is thus the natural one to use.
\end{remark}

\subsection{Proof of Theorem \ref{thm:lag-war-exist}} \label{sec:lag-exist}

We first check that $\cD(\AA^Z)$ is a core for $\AA^Z$ via Theorem \ref{thm:z-reg-cond}.  

\begin{prop} \label{prop:lag-regular}
The generator $\AA^Z$ of the Jacobi Warren process has core $\cD(\AA^Z)$.
\end{prop}
\begin{proof}
By Theorem \ref{thm:z-reg-cond}, it suffices to verify Assumptions \ref{ass:domain}, \ref{ass:d-cone}, \ref{ass:diag}, \ref{ass:ref-geo}, \ref{ass:estimate}, and \ref{ass:drift-bound}.  Assumptions \ref{ass:domain} and \ref{ass:d-cone}(b) hold because the faces of the Gelfand-Tsetlin cone are defined by inequalities between coordinates.  Assumptions \ref{ass:diag} (a)-(f), Assumption \ref{ass:ref-geo}(b)-(c), and Assumption \ref{ass:estimate} hold by the definition of $\AA^Z$.  For Assumption \ref{ass:ref-geo}(a), the face $F$ of minimal dimension containing any point $z$ is defined by equalities between different coordinates.  The span of $\wU(z)$ contains only vectors which are $0$ in the coordinate corresponding to the lowest level, while this coordinate is non-zero for all non-zero vectors in $T_zF$, yielding Assumption \ref{ass:ref-geo}(a).  Finally, Assumptions \ref{ass:diag}(g) and \ref{ass:drift-bound} follow from Theorem \ref{thm:existence-general-lw} and Lemma \ref{lem:drift-lag-ver} below, respectively.  This completes the verification of all assumptions of Theorem \ref{thm:z-reg-cond}, yielding the desired conclusion.
\end{proof}

\begin{lemma} \label{lem:bessel-mart}
For any $\delta \in \RR$ and $x > 0$, let $\EE^{\delta}_x[\cdot]$ denote the law of the Bessel process of dimension $\delta$ started at $x$.  For the process
\[
D^\delta_x(t) := \rho_x^\delta(t) \exp\Big(\frac{1 - \delta}{2} \int_0^t \frac{1}{\rho_x^\delta(s)^2} ds\Big) 1_{t < T_0},
\]
we have for any function $F: C([0, t], \RR_{\geq 0}) \to \RR_{\geq 0}$ that
\[
\EE^{\delta}_x\left[\frac{1}{x} D^\delta_x(t) F(\rho)\right] = \EE^{\delta + 2}_x\left[F(\rho)1_{t < T_0}\right],
\]
where $T_0$ is the hitting time at $0$.
\end{lemma}
\begin{proof}
This follows by the proof of \cite[Proposition 6]{Alt}, where we note that all steps are valid for $\delta < 0$.
\end{proof}

\begin{lemma} \label{lem:drift-lag-ver}
Let $\{l^k_i(t)\}_{1 \leq k \leq n}$ be a weak solution to (\ref{eq:lag-war-def}), and let $\tau^k_i$ be the first hitting time of $l^k_i(t)$ on a particle on a lower level. For sufficiently large $\lambda > 0$, and any $h(\{z^k_i\}) \in C^2_c(D)$ which is locally constant near $0$, the quantity
\[
\EE\left[\int_0^{\tau^k_i} e^{-\lambda t} \frac{\partial h}{\partial z^k_i} (\{l^k_i(t)^{1/2}\}) \exp\Big(\frac{1 - 2(p - k + 1)}{2} \int_0^t \frac{1}{l^k_i(r)} dr\Big) dt\right]
\]
is bounded as a function of $l^k_i(0)$.
\end{lemma}
\begin{proof}
First, if $k \leq p$, the claim follows because the norm of the integrand is bounded above by $\sup |\nabla h|$. Now, suppose that $k > p$. Notice that
\begin{multline*}
\EE\left[\int_0^{\tau^k_i} e^{-\lambda t} \frac{\partial h}{\partial z^k_i} (\{l^k_i(t)^{1/2}\}) \exp\Big(\frac{1 - 2(p - k + 1)}{2} \int_0^t \frac{1}{l^k_i(r)} dr\Big) dt\right]\\
= \EE\left[\int_0^{\infty} e^{-\lambda t} \frac{\frac{\partial h}{\partial z^k_i} (\{l^k_i(t)^{1/2}\})}{\sqrt{l^k_i(t)}}\, 1_{t < \tau^k_i} \sqrt{l^k_i(t)} \exp\Big(\frac{1 - 2(p - k + 1)}{2} \int_0^t \frac{1}{l^k_i(r)} dr\Big) dt\right].
\end{multline*}
Choose $Z > 0$ so that $|z^k_i| < Z$ for all $\{z^k_i\} \in \supp(h)$. Since $h$ is locally constant near $0$ and in $C^2_c(D)$, we see that
\[
\left|\frac{\frac{\partial h}{\partial z^k_i} (\{z^k_i\})}{z^k_i}\right| < C \cdot 1_{|z^k_i| < Z}
\]
for some $C > 0$. Consequently, it suffices for us to check that
\[
\EE\left[\int_0^\infty e^{-\lambda t}\, 1_{\sqrt{l^k_i(t)} < Z} 1_{t < \tau^k_i} \sqrt{l^k_i(t)} \exp\Big(\frac{1 - 2(p - k + 1)}{2} \int_0^t \frac{1}{l^k_i(r)} dr\Big) dt\right]
\]
is bounded.  For this, let $\delta = 2(p - k + 1)$ and $x = \sqrt{l^k_i(0)}$. For $t \leq \tau^k_i$, the law of $\sqrt{l^k_i(t)}$ coincides with that of a Bessel process $\rho^\delta_x(t)$ of dimension $\delta$ started at $\rho^\delta_x(0) = x$.  Let $T_0^\delta(x)$ be the hitting time of $\rho^\delta_x(t)$ at $0$, and define
\[
D^\delta_x(t) = 1_{t < T_0^\delta(x)} \, \rho^\delta_x(t) \cdot \exp\Big(\frac{1 - \delta}{2} \int_0^t \frac{1}{\rho^\delta_x(r)^2} dr\Big).
\]
It therefore suffices for us to check that
\[
S(x) = \EE\left[\int_0^\infty e^{-\lambda t} 1_{\rho^\delta_x(t) < Z} \, D^\delta_x(t) dt\right]
\]
is bounded as a function of $x$.  Adopting the notations of Lemma \ref{lem:bessel-mart}, we have that
\begin{equation} \label{eq:rough-bound}
S(x) = \EE^\delta_x\left[\int_0^\infty e^{-\lambda t} 1_{\rho(t) < Z} \, D^\delta_x(t) dt\right] = x \, \EE^{\delta + 2}_x\left[\int_0^\infty e^{-\lambda t} 1_{\rho(t) < Z} 1_{t < T_0} dt \right] \leq \frac{x}{\lambda}.
\end{equation}
Now, let $T^{\delta + 2}_Z(x)$ denote the first hitting time of $\rho^{d + 2}_x(t)$ at $Z$.  By \cite[Theorem 2]{BMR13}, for $x > Z$, we find that $T^{\delta + 2}_Z(x)$ has density bounded above by
\begin{align*}
p(t) dt & < C_\delta \frac{Z^{-1} x - 1}{1 + Z^{-\delta} x^\delta} \frac{e^{-\frac{(x - Z)^2}{2 Z^3 t}}}{t^{\frac{3}{2}} Z^3} \frac{x^{|\delta| - 1}}{Z^{|\delta| - 1} t^{|\delta|/2 - 1/2} + Z^{1/2 - |\delta|/2} x^{|\delta|/2 - 1/2}} dt\\
& < C_{\delta, Z} x^{|\delta|/2 + 1/2} e^{-\frac{(x - Z)^2}{2 Z^3 t}} t^{-3/2}
\end{align*}
for some constants $C_\delta, C_{\delta, Z}$, where the first equality is the result of \cite[Theorem 2]{BMR13} and the second results from some crude estimates.  We therefore have that
\begin{align*}
\PP\Big(T^{\delta + 2}_Z(x) < T\Big) &< C_{\delta, Z} x^{|\delta|/2 + 1/2} \int_0^T e^{-\frac{(x - Z)^2}{2Z^3 t}} t^{-3/2}dt\\
&= C_{\delta, Z} x^{|\delta|/2 + 1/2} \int_{T^{-1}}^\infty e^{-\frac{(x - Z)^2}{2Z^3} u} u^{-1/2} du\\
&< 2Z^3 C_{\delta, Z} \frac{x^{|\delta|/2 + 1/2} T^{-1/2} e^{- \frac{(x - Z)^2}{2 Z^3 T}}}{(x - Z)^2}.
\end{align*}
We now have for each $T > 0$ and $x > Z$ the estimate
\begin{align*}
x\EE^{\delta + 2}_x\left[\int_0^\infty e^{-\lambda t} 1_{\rho(t) < Z} 1_{t < T_0} dt \right] &\leq \frac{x e^{-\lambda T}}{\lambda} + x \PP\Big(T^{\delta + 2}_Z(x) < T\Big)\\
&< \frac{x e^{-\lambda T}}{\lambda} + C_{\delta, Z}' \frac{x^{|\delta|/2 + 3/2} T^{-1/2} e^{- \frac{(x - Z)^2}{2 Z^3 T}}}{(x - Z)^2}.
\end{align*}
where $C_{\delta, Z}' = 2 Z^3 C_{\delta, Z}$. Choosing $T = x$ and recalling (\ref{eq:rough-bound}) implies that for $x > Z$ we have
\begin{equation} \label{eq:better-bound}
S(x) < \frac{x e^{-\lambda x}}{\lambda} + C_{\delta, Z}' \frac{x^{|\delta|/2 + 1} e^{- \frac{(x - Z)^2}{2 Z^3 x}}}{(x - Z)^2}.
\end{equation}
On the other hand, for all $x$ we have $S(x) \leq \frac{x}{\lambda}$.  Since the bound in (\ref{eq:better-bound}) tends to $0$ as $x \to \infty$, combining these two yields an upper bound on $S(x)$ independent of $x$, as needed.
\end{proof}

We now prove Theorem \ref{thm:lag-war-exist} by applying Theorem \ref{thm:lag-war-inter} to reduce (\ref{eq:lag-war-def}) to the case of adjacent levels, rephrasing that case in terms of a stochastic differential equation with reflection on a time dependent boundary as described in Section \ref{sec:sder}, and finally applying the the Yamada-Watanabe type exactness criterion for SDER's provided by Theorem \ref{thm:exist-crit}.  A crucial input is given by the fact proven in Proposition \ref{prop:lep} that the Laguerre eigenvalues process has no collisions; this corresponds to the fact that the time-dependent boundaries in our SDER never touch.  

\begin{proof}[Proof of Theorem \ref{thm:lag-war-exist}]
We proceed by induction on $m$.  For the base case $m = 1$, equation (\ref{eq:lag-war-def}) is the SDE with no reflection terms for a single copy of $\BESQ^{2p}(t)$, hence admits a unique strong solution which is Feller Markov. For the inductive step, suppose that the result holds for (\ref{eq:lag-war-def}) for $m - 1$ and consider (\ref{eq:lag-war-def}) for $m$. By Theorem \ref{thm:lag-war-inter} applied to the solution $\{l^n_i(t)\}_{1 \leq i \leq n, 1 \leq n \leq m - 1}$ for $m - 1$, the projection to $\{l^{m - 1}_i(t)\}_{1 \leq i \leq m - 1}$ is equal in law to the Laguerre eigenvalues process of rank $p$ and level $m - 1$.  In particular, by Proposition \ref{prop:lep}, we have a.s. that $0 < l^{m - 1}_i(t) < l^{m - 1}_{i + 1}(t)$ for $1 \leq i < m - 2$.  Therefore, by Theorem \ref{thm:exist-crit}, there exists a unique strong solution $\{l^m_i(t)\}_{1 \leq i \leq m}$ to the SDER
\[
dl^m_i(t) = 2 \sqrt{l_i^m(t)} dB^m_i(t) + 2\Big(p - m + 1\Big) dt + d\Phi^m_i(t) - d\Psi^m_i(t)
\]
with time-dependent boundaries given by $l^{m - 1}_{i - 1}(t) < l^{m - 1}_i(t)$.  To check that $\Phi^m_i(t) = \frac{1}{2}L^{m, +}_i(t)$, we note by the It\^o-Tanaka formula of \cite[Theorem VI.1.2]{RY} that
\begin{align*}
\frac{1}{2}L^{m, +}_i(t) &= (l^m_i(t) - l^{m - 1}_{i - 1}(t))^+ - (l^m_i(0) - l^{m - 1}_{i - 1}(0))^+ - \int_0^t 1_{l^m_i(s) > l^{m - 1}_{i - 1}(s)} d(l^m_i(s) - l^{m - 1}_{i - 1}(s))\\
&= l^m_i(t) - l^m_i(0) - \int_0^t 2 \sqrt{l_i^m(s)} dB^m_i(s) - \int_0^t 2(p - m + 1) ds + \Psi^m_i(t) \\
&\phantom{====} + \int_0^t 1_{l^m_i(s) = l^{m - 1}_{i - 1}(s)} \Big(2 \sqrt{l_i^m(s)} dB^m_i(s) + 2 (p - m + 1)dt\Big)\\
&= \Phi^m_i(t) + \int_0^t 1_{l^m_i(s) = l^{m - 1}_{i - 1}(s)} \Big(2 \sqrt{l_i^m(s)} dB^m_i(s) + 2 (p - m + 1) dt\Big).
\end{align*}
Now, by the occupation time formula of \cite[Corollary VI.1.6]{RY}, we see that
\[
\int_0^t 1_{l^m_i(s) = l^{m - 1}_{i - 1}(s)} (4 l_i^m(s) + 4 l_{i - 1}^{m - 1}(s)) ds = \int_{-\infty}^\infty 1_{a = 0}\, dL^{a}_t(l^m_i - l^{m - 1}_{i - 1}) = 0.
\]
We conclude that 
\[
\PP\Big(\int_0^t 1_{l^m_i(s) = l^{m - 1}_{i - 1}(s)} ds = 0\Big) = 1,
\]
and therefore that 
\[
\int_0^t 1_{l^m_i(s) = l^{m - 1}_{i - 1}(s)} (2 \sqrt{l_i^m(s)} dB^m_i(s) + 2 (p - m + 1)dt) = 0,
\]
so $\Phi^m_i(t) = \frac{1}{2}L^{m, +}_i(t)$.  A similar argument shows that $\Psi^m_i(t) = \frac{1}{2}L^{m, -}_i(t)$.  Combining $\{l^m_i(t)\}$ with the previous solution $\{l^{n}_i(t)\}_{1 \leq i \leq n, 1 \leq n \leq m - 1}$ yields the desired unique strong solution.

We now check that $\{l^n_i(t)\}_{1 \leq i \leq n, 1 \leq n \leq m}$ is a regular Feller process.  Because $\{l^n_i(t)\}$ has continuous trajectories, to check that it is Feller, it suffices to check that its generator $\AA$ preserves the space $C_0(\GT_{m, p})$ of continuous functions on $\GT_{m, p}$ vanishing at infinity.  For this, we induct on $m$; for $m = 1$, this follows because the squared Bessel process is Feller.  For the inductive step, it suffices to show that for any $t \geq 0$ and any bounded Borel set $A \times B \subset \GT_{m + 1, p}$ with $A \subset \GT_{m, p}$, we have 
\begin{equation} \label{eq:feller-traj}
\lim_{\{\mu^n_i\} \to \infty} \PP\Big(\{l^n_i(t)\}_{1 \leq n \leq m + 1} \in A \times B \mid l^n_i(0) = \mu^n_i\Big) = 0
\end{equation}
for any trajectory $\{\mu^n_i\}_{1 \leq n \leq m + 1} \to \infty$.  If $\{\mu^n_i\}_{1 \leq n \leq m} \to \infty$, (\ref{eq:feller-traj}) follows by the Feller property for $A$ and $\{l^n_i(t)\}_{1 \leq n \leq m}$ given by the inductive hypothesis.  Otherwise, if $m + 1 \leq p$, we must have either $\mu^{m+1}_1 \to 0$ or $\mu^{m+1}_{m + 1} \to \infty$ along the trajectory.  In this case, the law of $l^{m+1}_{m+1}(t)$ stochastically dominates the law of a squared Bessel process started at $\mu^{m+1}_{m+1}$ and the law of $l^{m+1}_1(t)$ is stochastically dominated by the law of a squared Bessel process started at $\mu^{m+1}_1$.  Therefore, (\ref{eq:feller-traj}) follows by the Feller property for the squared Bessel process applied to the projection of $B$ to either its first or last coordinate.  This completes the proof that $\{l^n_i(t)\}$ is Feller.  Finally, regularity follows from Proposition \ref{prop:lag-regular}. 
\end{proof}

\subsection{The Jacobi Warren process} \label{sec:jwp}

In this section, we couple the Jacobi eigenvalue processes at different levels.  This construction is the Laguerre analogue of the Warren process for Dyson Brownian motion of \cite{War}.  Define the the $[0, 1]$ Gelfand-Tsetlin polytope with parameters $(p, q)$ and level $k \leq \min\{p, q\}$ by 
\[
\GTO_{p, q, k} := \{j^n_i \mid 0 \leq j^{n - 1}_{i - 1} \leq j^n_i \leq j^{n - 1}_i \leq 1 \text{ and } j^{\min\{p, q\}}_1 < \cdots < j^{\min\{p, q\}}_{\min\{p, q\}}\}_{1 \leq i \leq n, 1 \leq n \leq k}
\]
and the $[0, 1]$ Gelfand-Tsetlin polytope with parameters $(p, q)$ and level $k$ subordinate to $\mu = (\mu_1 < \cdots < \mu_{\min\{p, q\}})$ by
\[
\GTO_{p, q, k}(\mu) := \{\{j^n_i\} \in \GT_{p, q} \mid j^{\min\{p, q\}} = \mu\}.
\]
We say that a probability distribution $\nu$ on $\GTO_{p, q, \min\{p, q\}}$ is Gibbs if for any $\mu = (\mu_1 < \cdots < \mu_{\min\{p, q\}})$ and any Borel $B\subset \GTO_{p, q, \min\{p, q\}}(\mu)$, we have that 
\[
\PP_\nu(\{j^n_i\} \in B \mid j^{\min\{p, q\}} = \mu) = (\min\{p, q\} - 1)! \cdots 1! \frac{\vol(B)}{\Delta(\mu)}.
\]
Consider the system of stochastic differential equations with reflection with domain $\GTO_{p, q, \min\{p, q\}}$ given by
\begin{multline} \label{eq:jac-war-def}
dj^n_i(t) = 2 \sqrt{j^n_i(t)(1 - j^n_i(t))} dB^n_i(t)\\ + 2 \Big((p - n + 1) - (p + q - 2n + 2) j^n_i(t)\Big) dt + \frac{1}{2}dL^{n, +}_i(t) - \frac{1}{2}dL^{n, -}_i(t), \qquad 1 \leq i \leq n, 1 \leq n \leq p, q,
\end{multline}
where $B^n_i(t)$ are standard real Brownian motions, $L^{n, +}_i$ is $0$ if $i = 1$ and the local time of $j^n_i(t) - j^{n - 1}_{i - 1}(t)$ at $0$ otherwise, and $L^{n, -}_i$ is $0$ if $i = n$ and the local time of $j^{n - 1}_i(t) - j^n_i(t)$ at $0$ otherwise.

\begin{remark}
Informally, a solution to (\ref{eq:jac-war-def}) may be described as follows.  At level $n$, it consists of independent univariate Jacobi processes with parameters $(p - n, q - n)$ interlacing with and reflecting off the processes at level $n - 1$.  As with the SDER (\ref{eq:lag-war-def}), this differs from the Warren process by replacing Brownian motions by univariate Jacobi processes and introducing different parameters on each level.
\end{remark}

The following two theorems are analogues of Theorems \ref{thm:lag-war-inter} and \ref{thm:lag-war-exist} and are our second set of main results; proofs will be given in the following subsections.  We show that (\ref{eq:jac-war-def}) admits a unique strong solution for any Gibbs initial condition and that this solution provides a coupling of the Jacobi eigenvalues processes on each level.  We call the resulting process the \textit{Jacobi Warren process}.

\begin{theorem} \label{thm:jac-war-inter}
For any Gibbs initial condition $\{j^n_i(0)\}$, if the SDER (\ref{eq:jac-war-def}) admits a unique weak solution which is a regular Feller process in the sense of Assumption \ref{ass:proc}, then for $1 < n \leq \min\{p, q\}$, its projection to levels $n$ and $n - 1$ is Markovian and is an intertwining of the Jacobi eigenvalues processes with parameters $(p, q)$ and levels $n$ and $n - 1$ in the sense of Definition \ref{def:ps-def}.
\end{theorem}

\begin{theorem} \label{thm:jac-war-exist}
For any Gibbs initial condition $\{j^n_i(0)\}$, the SDER (\ref{eq:jac-war-def}) admits a unique strong solution $\{j^n_i(t)\}_{1 \leq i \leq n, 1 \leq n \leq \min\{p, q\}}$ which is a regular Feller process and which we call the Jacobi Warren process.
\end{theorem}

\begin{corr} \label{corr:jac-war}
For $1 \leq n \leq \min\{p, q\}$, the projection of the Jacobi Warren process to level $n$ is Markovian and coincides in law with the Jacobi eigenvalues process with parameters $(p, q)$ and level $n$.
\end{corr}
\begin{proof}
This follows by combining Theorem \ref{thm:jac-war-exist}, Theorem \ref{thm:jac-war-inter}, and the definition of intertwining.
\end{proof}

\begin{corr} \label{corr:jac-inv}
The Jacobi Warren process admits an invariant measure proportional to
\[
\Delta(j^{\min\{p, q\}}) \prod_{i = 1}^{\min\{p, q\}} (j^{\min\{p, q\}}_i)^{p - \min\{p, q\}} (1 - j^{\min\{p, q\}}_i)^{q - \min\{p, q\}} \prod_{n = 1}^{\min\{p, q\}} \prod_{i = 1}^n dj^{n}_i.
\]
\end{corr}
\begin{proof}
This measure is the Gibbs measure associated to the invariant measure (\ref{eq:jac-inv}) for the Jacobi eigenvalues process with parameters $(p, q)$ and level $\min\{p, q\}$ from Proposition \ref{prop:jep}(b).  By Theorem \ref{thm:jac-war-inter}, the Jacobi Warren process preserves Gibbs measures and projects to the Jacobi eigenvalues process.  Therefore, the claimed measure is invariant for the Jacobi Warren process because its projection to level $\min\{p, q\}$ is invariant for the Jacobi eigenvalues process.
\end{proof}

\subsection{Proof of Theorem \ref{thm:jac-war-inter}}

We prove by induction on $n$ that $\{j^k_i(t)\}_{1 \leq k \leq n - 1}$ is intertwined with the Jacobi eigenvalues process with parameters $(p, q)$ and level $n$.  The base case is trivial.  During the proof, we adopt the notations of Sections \ref{sec:ps-ext} and \ref{sec:core-crit}.  In particular, we consider the domains $\XX := \GTO_{p, q, n - 1}$, $\YY := \{0 \leq y_1 \leq \cdots \leq y_n\}$, and
\[
D := \{(\{x^k\}, y) \mid \{x^k\} \in \XX, y \in \YY, 0 \leq y_1 \leq x_1^{n-1} \leq y_2 \cdots \leq x_{n-1}^{n-1} \leq y_n \leq 1\}.
\]
On these domains, we let $X(t)$ denote $\{j^k_i(t)\}_{1 \leq k \leq n - 1}$, $Y(t)$ denote the Jacobi eigenvalues process of rank $p$ and level $n$, and $Z(t)$ denote $\{j^k_i(t)\}_{1 \leq k \leq n}$, the weak solution to (\ref{eq:jac-war-def}) stopped at the non-smooth parts of the boundary.  Let
\begin{align*}
\AA^X &:= \sum_{k = 1}^{n - 1} \sum_{i = 1}^{k} \Big(2 (1 - x^k_i)x^k_i \partial^2_{x^k_i} + 2 \Big((p - k + 1) - (p + q - 2k + 2)x^k_i\Big)\partial_{x^k_i}\Big)\\
\AA^Z &:= \AA^X + \sum_{i = 1}^n \Big[2 (1 - y_i)y_i \partial_{y_i}^2 + 2 \Big((p - n + 1) - (p + q - 2n + 2)y_i\Big)\partial_{y_i}\Big]
\end{align*}
denote the generators of $X(t)$ and $Z(t)$, and let $\AA^Y := \cJ^{p, q}_n$ denote the generator of $Y(t)$.  Define the kernel
\[
\Lambda(y, x) := (n - 1)! \cdots 1! \cdot \Delta(y)^{-1}.
\]
We will verify the hypotheses of Theorem \ref{thm:ps-ext} for $D$ and $\Lambda$. 
\begin{itemize}
\item \textbf{Assumption \ref{ass:proc}:} Note that the Jacobi eigenvalues processes at level $n$ have no boundary conditions and are Feller because they have compact domain and continuous trajectories.

\item \textbf{Assumption \ref{ass:domain}:} This follows in the same way as in the proof of Theorem \ref{thm:lag-war-inter}.

\item \textbf{Assumption \ref{ass:link}:} Points (a), (b), and (c) follow from the definition of $\Lambda(y, x)$.  For (d), by Proposition \ref{prop:jac-h}, we have that
\[
\J^{p, q, y}_{n} = \Delta(y)^{-1} \circ J^{p - n, q - n, y}_{n} \circ \Delta(y) + \frac{n(n - 1)(3p + 3q - 4n + 2)}{3},
\]
so we find that
\[
\J^{p, q, y}_{n} \Lambda(y, x) = \frac{n(n - 1)(3p + 3q - 4n + 2)}{3}\Lambda(y, x)
\]
is continuous and bounded on $\bigcup_{x \in K} D(x, -)$ for all compact $K$, giving (c).

\item \textbf{Assumption \ref{ass:compat}:}  Points (a) and (d) are trivial because $Y(t)$ has no reflecting boundary.  Point (b) follows by definition of the kernel.  For (c), we apply Proposition \ref{prop:geo-cond} with $\AA^{X_1} = \J^{p, q, x^1}_{n - 1}$ being the generator of the Jacobi eigenvalues process with parameters $(p, q)$ and level $n - 1$.  We consider the two kernels
\[
\Lambda_1(y, x^1) = (n - 1)! \frac{\Delta(x^1)}{\Delta(y)} \qquad \text{ and } \qquad \Lambda_2(x^1, x^2) = (n - 2)! \cdots 1! \frac{1}{\Delta(x^1)}.
\]
We now check the hypotheses of Proposition \ref{prop:geo-cond}.
\begin{itemize}
\item[(a)] This follows from the inductive hypothesis.

\item[(b)] By the same computation as for Assumption \ref{ass:link}, we find that $\AA^Y\Lambda_1 =  \frac{n(n - 1)(3p + 3q - 4n + 2)}{3}\Lambda_1$.  On the other hand, notice that 
\begin{align*}
(\J^{p, q, x^1}_{n - 1})^* &= \Delta(x^1) \circ J^{p - n + 1, q - n + 1, x}_{n - 1} \circ \Delta(x^1)^{-1} - 4(n - 1) \\
&\phantom{==}+ 2 (n - 1)\Big((p - n + 1) + (q - n + 1) + 2\Big) + \frac{(n - 1)(n - 2)(3p + 3q - 4n + 6)}{3}\\
&= \Delta(x^1) \circ J^{p - n + 1, q - n + 1, x^1}_{n - 1} \circ \Delta(x^1)^{-1} + \frac{n(n - 1)(3p + 3q - 4n + 2)}{3},
\end{align*}
from which we conclude that
\[
\frac{(\J^{p, q, x^1}_{n - 1})^* \Lambda_1}{\Lambda_1} = \frac{n(n - 1)(3p + 3q - 4n + 2)}{3},
\]
giving the desired equality.

\item[(c)] Faces of $\partial D_1(y)$ take the form $\{x_i^1 = y_i\}$ and $\{x_i^1 = y_{i+1}\}$.  On $\{x_i^1 = y_i\}$, we again have $\eta_k^1 = -1_i$ and $\langle \Psi^{j, 1}_k, \eta^1_k\rangle = - \delta_{ij}$, so the equality becomes 
\begin{multline*}
- 2((p - n + 2) - (p + q - 2n + 4) x^1_i) - \sum_{j \neq i} \frac{4 x^1_i(1 - x^1_i)}{x^1_i - x^1_j} + (2 - 4x^1_i) + \sum_{j \neq i} \frac{4 x^1_i(1 - x^1_i)}{x^1_i - x^1_j}\\
 = - 2((p - n + 1) - (p + q - 2n + 2) y_i) - \sum_{j \neq i} \frac{4 y_i(1 - y_i)}{y_i - y_j} + \sum_{j \neq i} \frac{4 y_i(1 - y_i)}{y_i - y_j}.
\end{multline*}
 On $\{x_i^1 = y_{i + 1}\}$, we have $\eta_k^1 = 1_i$ and $\langle \Psi^{j, 1}_k, \eta^1_k\rangle = \delta_{i + 1, j}$, so the equality becomes 
\begin{multline*}
2((p - n + 2) - (p + q - 2n + 4) x^1_i) + \sum_{j \neq i} \frac{4 x^1_i(1 - x^1_i)}{x^1_i - x^1_j} - (2 - 4x^1_i) - \sum_{j \neq i} \frac{4 x^1_i(1 - x^1_i)}{x^1_i - x^1_j}\\
 = 2((p - n + 1) - (p + q - 2n + 2) y_i) + \sum_{j \neq i+1} \frac{4 y_{i+1}(1 - y_{i+1})}{y_{i+1} - y_j} - \sum_{j \neq i+1} \frac{4 y_{i+1}(1 - y_{i+1})}{y_{i + 1} - y_j}.
\end{multline*}

\item[(d)] The condition reduces to to $-4x_i^1(1 - x_i^1) = - 4 y_i(1 - y_i)$ on $\{x_i^1 = y_i\}$ and $4x_i^1(1 - x_i^1) = 4 y_{i+1}(1 - y_{i+1})$ on $\{x_i^1 = y_{i+1}\}$. 

\item[(ef)] These conditions follow in the same way as in the proof of Theorem \ref{thm:lag-war-inter}.
\end{itemize}
Having checked all hypotheses of Proposition \ref{prop:geo-cond}, we deduce Assumption \ref{ass:compat}.
\end{itemize}
We conclude that Assumptions \ref{ass:proc}, \ref{ass:domain}, \ref{ass:link}, and \ref{ass:compat} hold for $(X(t), Y(t))$.  Now, the SDER of Theorem \ref{thm:ps-ext} is given by (\ref{eq:jac-war-def}) and by the given has a weak solution which is a regular Feller diffusion.  Its Markov generator (\ref{eq:z-gen-def}) is
\begin{align*}
\AA^Z &:= \AA^X + \J^{p, q, y}_{n} - \sum_{i = 1}^n \sum_{j \neq i} \frac{4y_i (1 - y_i)}{y_i - y_j}\\
&= \AA^X + \sum_{i = 1}^n 2y_i (1 - y_i) \partial_{y_i}^2 + \sum_{i = 1}^n 2(p - n + 1 - (p + q - 2n + 2)y_i)\partial_{y_i}
\end{align*}
with the claimed Neumann boundary conditions on $\partial D(y)_k$.  We may therefore apply Theorem \ref{thm:z-reg-cond} to conclude that $(X(t), Y(t))$ are intertwined by $Z(t)$ so that $(\{j^k(t)\}_{1 \leq k \leq n - 1}, j^n(t))$ is an intertwining of the Jacobi Warren process with parameters $(p, q)$ and level $n - 1$ and the Jacobi eigenvalues process with parameters $(p, q)$ and level $n$, completing the proof.

\subsection{Proof of Theorem \ref{thm:jac-war-exist}} \label{sec:jac-exist}

We now prove Theorem \ref{thm:jac-war-exist} in parallel to the proof of Theorem \ref{thm:lag-war-exist}.  Again, we require the fact from Proposition \ref{prop:jep} that the Jacobi eigenvalues process has no collisions and hence the time-dependent boundaries in our SDER never touch.  We begin again by verifying that $\AA^Z$ has core $\cD(\AA^Z)$.

\begin{prop} \label{prop:jac-regular}
The generator $\AA^Z$ of the Jacobi Warren process has core $\cD(\AA^Z)$.
\end{prop}
\begin{proof}
The proof is exactly parallel to that of Proposition \ref{prop:lag-regular}, with two modifications.  First, Assumption \ref{ass:diag}(g) follows from Theorem \ref{thm:existence-general-jw} instead of Theorem \ref{thm:existence-general-lw}.  Second, Assumption \ref{ass:drift-bound} follows from the following argument instead of Lemma \ref{lem:drift-lag-ver}.  Note that the coordinate transform for the Jacobi Warren process is $\chi(z) = \arcsin(\sqrt{z})$, so the transformed drift for the $i^\text{th}$ particle at level $k$ is
\[
\wgamma^k_i(\chi(j)) := \frac{(p - k)(1 - j) - (q - k) j}{2 \sqrt{j(1 - j)}}
\]
with derivative
\[
\frac{\partial \wgamma^k_i(\chi(j))}{\partial j} = - \frac{(p - k)(1 - j) + (q - k)j}{4 [(1 - j) j]^{3/2}} \leq 0.
\]
By the chain rule, this implies that $\frac{\partial \wgamma^n_i}{\partial \wtilde{j}} \leq 0$ and therefore that the quantity (\ref{eq:drift-exp}) is bounded above by 
\[
\EE\left[\int_0^\infty e^{-\lambda t} |\partial_ih(\chi(\{j^k_i(t)\}))| \exp\Big(\int_0^t  \frac{\partial \wgamma^n_i}{\partial \wtilde{j}}(\chi(j^k_i(r))) dr\Big) dt\right].
\]
The quantity inside the expectation is uniformly bounded in $\{j^k_i(0)\}$, hence the expectation itself is, yielding Assumption \ref{ass:drift-bound}.
\end{proof}

\begin{proof}[Proof of Theorem \ref{thm:jac-war-exist}]
We proceed by induction on $n$ to show that a strong solution exists for the first $n$ levels of (\ref{eq:jac-war-def}).  For $n = 1$, the equation is the SDE with no reflection terms for a single copy of $\JAC^{p - 1, q - 1}(t)$, hence admits a unique strong solution which is Feller Markov.  For the inductive step, suppose that the result holds for (\ref{eq:jac-war-def}) for $n - 1$ and consider (\ref{eq:jac-war-def}) for $n$.  By Theorem \ref{thm:jac-war-inter} applied to the solution $\{j^k_i(t)\}_{1 \leq i \leq k, 1 \leq k \leq n - 1}$ for $n - 1$, the projection to $\{j^{n - 1}_i(t)\}$ is equal in law to the Jacobi eigenvalues process with parameters $(p, q)$ and level $n - 1$.  In particular, we have a.s. that $0 < j^{n - 1}_i(t) < j^{n - 1}_{i + 1}(t) < 1$ for $1 \leq i \leq n - 2$.  Therefore, by Theorem \ref{thm:exist-crit} there exists a unique strong solution $\{j^n_i(t)\}$ to the SDER
\[
dj^n_i(t) = 2 \sqrt{j^n_i(t)(1 - j^n_i(t))} dB^n_i(t) + 2\Big((p - n + 1) + (p + q - 2n + 2)j^n_i(t)\Big) dt + d\Phi^n_i(t) - d\Psi^n_i(t)
\]
with time-dependent boundaries given by $j^{n - 1}_{i-1}(t) < j^{n - 1}_i(t)$.  The proof that $\Phi^n_i(t) = \frac{1}{2} L^{m, +}_i(t)$ and $\Psi^n_i(t) = \frac{1}{2}l^{m, -}_i(t)$ now follows in the exact same way as in the proof of Theorem \ref{thm:lag-war-exist}. Combining $\{j^n_i(t)\}$ with $\{j^k_i(t)\}_{1 \leq i \leq k, 1 \leq k \leq n - 1}$ yields the desired strong solution.  The Feller property follows because the state space of $\{j^k_i(t)\}$ is compact and trajectories are continuous, and regularity follows from Proposition \ref{prop:jac-regular}.
\end{proof}

\section{Strong existence and uniqueness for SDER's with time-dependent boundary} \label{sec:sder}

In this section, we gather existing results to prove a criterion for strong existence and uniqueness of solutions to one-dimensional SDE's with reflection on two time-dependent boundaries.  In particular, we handle the case of Lipschitz drift and Holder diffusion coefficient which is used in our applications.

\subsection{Statement of the SDER}

By a \textit{context}, we mean a complete probability space $(\Omega, \FF, \PP)$ equipped with an increasing family $\{\FF_t\}_{t \geq 0}$ of sub-$\sigma$-fields of $\FF$ and an $\FF_t$-adapted Brownian motion $B_t$.  Given such a context, let $(x_0, L_t, U_t)$ be $\FF_0$-measurable random variables on $\RR \times C[0, \infty) \times C[0, \infty)$ so that $L_t < U_t$ and $L_0 < x_0 < U_0$ a.s..  For diffusion and drift coefficients $\sigma$ and $b$, we consider the SDE with reflection
\begin{equation} \label{eq:ref-sde}
dX_t = \sigma(X_t) dB_t + b(X_t) dt + d\Phi_t - d\Psi_t.
\end{equation}
A strong solution to (\ref{eq:ref-sde}) with initial condition $x_0$ is a triple of continuous $\FF_t$-adapted processes $(X_t, \Phi_t, \Psi_t)$ so that:
\begin{itemize}
\item $X_t = \int_0^t \sigma(X_s) dB_s + \int_0^t b(X_s) ds + \Phi_t - \Psi_t$ and $X_0 = x_0$;

\item $L_t \leq X_t \leq U_t$ for all $t$;

\item $\Phi_t$ and $\Psi_t$ are non-decreasing with bounded variation and $\Phi_0 = \Psi_0 = 0$;

\item $\int_0^\infty 1_{X_t \neq L_t} d\Phi_t = \int_0^\infty 1_{X_t \neq U_t} d\Psi_t = 0$.
\end{itemize}
We say that the SDER (\ref{eq:ref-sde}) is \textit{exact} if for every choice of initial-boundary conditions $(x_0, L_t, U_t)$ there is a unique strong solution to (\ref{eq:ref-sde}).

\begin{theorem}[{\cite[Theorem 3.3]{SW13} and \cite[Corollary 2.4]{BKR}}] \label{thm:sde-exact}
If $\sigma$ and $b$ are Lipschitz, then (\ref{eq:ref-sde}) is exact.
\end{theorem}

\begin{remark}
We allow $U_t \equiv \infty$ or $L_t \equiv - \infty$, in which case we have $\Psi_t \equiv 0$ or $\Phi_t \equiv 0$, respectively, and this setting coincides with that of \cite[Definition III.1.6]{Sou}.
\end{remark}

\begin{remark}
This definition is a generalization of \cite[Definition III.1.6]{Sou} and a specialization of the definition in \cite[Section 3]{SW13} to the case where $L_t < U_t$ for all $t$ via \cite[Corollary 2.4]{BKR}.  In our setting, the random variables $L_t$ and $U_t$ do not depend on the past trajectory of $X_t$, hence are constant Lipschitz operators in the sense of \cite[Section 3]{SW13}.
\end{remark}

\subsection{Pathwise uniqueness and strong existence of solutions}

We have the following result on pathwise uniqueness for (\ref{eq:ref-sde})  following the modifications to the proof of \cite[Proposition III.5.2]{Sou} discussed in \cite[Section IV.3.1]{Sou}, 

\begin{prop} \label{prop:sou-unique}
Let $\rho: \RR_+ \to \RR$ be a function so that $\int_{0^+} \frac{1}{\rho(s)} ds = \infty$.  Suppose that $\sigma$ satisfies
\[
|\sigma(x) - \sigma(x')|^2 \leq \rho(|x - x'|)
\]
for $x, x' \in \RR$ and $b$ is Lipschitz.  Then pathwise uniqueness holds for (\ref{eq:ref-sde}).
\end{prop}
\begin{proof}
If $L_t \equiv -\infty$ or $U_t \equiv \infty$, this is exactly \cite[Proposition III.5.2]{Sou}. Consider now the case where $-\infty < L_t < U_t < \infty$.  Suppose that $(X_t, \Phi_t, \Psi_t)$ and $(X_t', \Phi_t', \Psi_t')$ are two solutions to (\ref{eq:ref-sde}) defined on the same context.  By Tanaka's formula applied to $X_t - X_t'$, we find that 
\begin{align} \label{eq:pos-decomp}
(X_t - X_t')_+ &= \int_0^t 1_{X_s > X_s'} d(X - X')_s + \frac{1}{2} L_t^0 \\ \nonumber
&= \int_0^t 1_{X_s > X_s'}(\sigma(X_s) - \sigma(X_s')) dB_s + \int_0^t1_{X_s > X_s'} (b(X_s) - b(X_s')) ds\\ \nonumber
&\phantom{=} + \int_0^t 1_{X_s > X_s'}(d\Phi_s - d\Phi_s') - \int_0^t1_{X_s > X_s'}(d\Psi_s- d\Psi_s') + \frac{1}{2}L_t^0,
\end{align}
where $L_t^a$ is the local time of $X_t - X_t'$ at $a$.  We first claim that $L_t^0 = 0$ for all $t$.  By the occupation time formula, we have that 
\[
\int_{0}^\infty \frac{1}{\rho(a)} L^a_t da = \int_0^t 1_{X_s > X_s'} \frac{1}{\rho(X_s - X_s')} d\langle X - X', X - X'\rangle_s
= \int_0^t 1_{X_s > X_s'} \frac{(\sigma(X_s) - \sigma(X_s'))^2}{\rho(X_s - X_s')} ds \leq t.
\]
On the other hand, if $L^0_t > \eps$ occurs with positive probability for some $t$ and some $\eps > 0$, then because $\lim_{a \to 0} L^a_t = L^0_t$, we may find $\delta > 0$ sufficiently small so that $L^a_t > \eps/2$ for $a < \delta$ with positive probability.  This implies that with positive probability we have
\[
t \geq \int_0^\infty \frac{1}{\rho(a)} L_t^a da > \frac{\eps}{2} \int_0^\delta \frac{1}{\rho(a)} da,
\]
a contradiction.  We conclude that $L^0_t = 0$ for all $t$.  

Now, notice that $X_s > X_s'$ implies that $X_s > L_s$ and that $X_s' < U_s$, hence we find that
\[
\int_0^t 1_{X_s > X_s'} d\Phi_s = \int_0^t 1_{X_s > X_s'} d\Psi_s' = 0.
\]
On the other hand, because $\Phi_t, \Phi_t', \Psi_t$, and $\Psi_t'$ are non-decreasing, we have that 
\[
\int_0^t 1_{X_s > X_s'} d\Phi_s' \geq 0 \qquad \text{ and } \qquad \int_0^t 1_{X_s > X_s'} d\Psi_s' \geq 0.
\]
Since $-\infty < L_s < U_s < \infty$, for $0 \leq s \leq t$, we see that $\sigma(X_s) - \sigma(X_s')$ is bounded on $[0, t]$ and hence $\int_0^t 1_{X_s > X_s'} (\sigma(X_s) - \sigma(X_s')) dB_s$ is a martingale.  Taking expectations in (\ref{eq:pos-decomp}), we conclude that 
\begin{multline*}
\EE[(X_t - X_t')_+] \leq \EE\left[\int_0^t 1_{X_s > X_s'} |b(X_s) - b(X_s')| ds\right]\\ \leq K \, \EE\left[\int_0^t 1_{X_s > X_s'} |X_s - X_s'| ds\right] = K\, \EE\left[\int_0^t (X_s - X_s')_+ ds\right],
\end{multline*}
where $K$ is a Lipschitz constant for $b$.  By Gronwall's Lemma, we conclude that $\EE[(X_t - X_t')_+] = 0$ and hence that $X_t \leq X_t'$ a.s..  On the other hand, exchanging the roles of $X_t$ and $X_t'$ implies that $X_t' \leq X_t$ and hence $X_t = X_t'$ a.s..  Since $\Phi_t$ and $\Psi_t$ never increase at the same time, they are determined by $X_t$, hence we conclude $(X_t, \Phi_t, \Psi_t) = (X_t', \Phi_t', \Psi_t')$ a.s., as desired.
\end{proof}

We now give a criterion for strong existence of solutions to (\ref{eq:ref-sde}) based on localization and Theorem \ref{thm:sde-exact}.  Let $D \subset \RR$ be a connected domain.  We say that (\ref{eq:ref-sde}) has a strong solution in $D$ if for any $x_0 \in D$ and $L_t, U_t$ lying in $D$ for all $t$, there is a strong solution to (\ref{eq:ref-sde}) for which $X_t$ lies in $D$ for all $t$.

\begin{prop} \label{prop:strong-exist}
Suppose $\sigma$ and $b$ are locally Lipschitz on $D$ and pathwise uniqueness in $D$ holds for (\ref{eq:ref-sde}) for all initial-boundary conditions.  We have the following.
\begin{enumerate}
\item[(a)] For any initial-boundary condition $(x_0, L_t, U_t)$ there exists a strong solution to (\ref{eq:ref-sde}) in $D$ up to some explosion time $\tau_E$;

\item[(b)] If strong existence in $D$ holds for the classical SDE
\[
dX_t = \sigma(X_t) dB_t + b(X_t) dt \qquad X_0 = x_0
\]
for each $x_0 \in D$, then we have $\tau_E = \infty$ a.s.
\end{enumerate}
\end{prop}
\begin{proof}
For (a), let $D_1 \subset D_2 \subset \cdots D$ be a sequence of connected subdomains which exhaust $D$ and on which $\sigma, b$ are Lipschitz.  Let $D_k = (l_k, u_k)$, and define the modified functions
\[
\sigma^k(x) := \begin{cases} \sigma(x) & x \in D_k \\ \sigma(l_k) & x \leq l_k \\ \sigma(u_k) & x \geq u_k \end{cases} \text{ and } b^k(x) := \begin{cases} b(x) & x \in D_k \\ b(l_k) & x \leq l_k \\ b(u_k) & x \geq u_k. \end{cases}
\]
Observe that $\sigma^k, b^k$ are Lipschitz on $\RR$, so by Theorem \ref{thm:sde-exact} for every initial-boundary condition, there is a unique strong solution $(X^k_t, \Phi^k_t, \Psi^k_t)$ to the reflected SDE with Lipschitz coefficients
\begin{equation} \label{eq:sde-mod}
dX^k_t = \sigma^k(X^k_t) dB_t + b^k(X^k_t) dt + d\Phi^k_t - d\Psi^k_t
\end{equation}
and initial-boundary condition $(x_0, L_t, U_t)$.  Define the stopping times $\tau_k := \inf\{t \mid X^k_t \notin D_k\}$ so that $\tau_1 \leq \tau_2 \leq \cdots$.  By strong uniqueness, the restriction of $(X^k_t, \Phi^k_t, \Psi^k_t)$ to times in $[0, \tau_l)$ for $l \leq k$ coincides with $(X^l_t, \Phi^l_t, \Psi^l_t)$, so we may glue together all such solutions to obtain a strong solution $(X^*_t, \Phi^*_t, \Psi^*_t)$ in $D$ valid until the explosion time $\tau_E := \lim_{k \to \infty} \tau_k$.

For (b), if $U_t \equiv \infty$ or $L_t \equiv -\infty$, then this follows from \cite[Lemma III.4.4]{Sou}.  Otherwise, for any $T > 0$, we may find some $k$ so that $L_t, U_t \in D_k$ for $t \in [0, T]$.  The solution $(X^k_t, \Phi^k_t, \Psi^k_t)$ to (\ref{eq:sde-mod}) given by Theorem \ref{thm:sde-exact} satisfies $L_t \leq X^k_t \leq U_t$, meaning that $X^k_t \in D_k$ for $0 \leq t \leq T$ and hence that $\tau_k > T$ and hence $\tau_E > T$.  We conclude that $\tau_E = \infty$ a.s., as desired.
\end{proof}

\subsection{An exactness criterion for reflected SDE's}

We now assemble the previous results to prove the goal of this section, the following refinement of Theorem \ref{thm:sde-exact} which provides a criterion of Yamada-Watanabe type for exactness of SDER's applicable to the setting of squared Bessel and univariate Jacobi generators.

\begin{theorem} \label{thm:exist-crit}
Suppose that $b$ is Lipschitz and for $x \neq x' \in \RR$ the diffusion $\sigma$ satisfies
\[
|\sigma(x) - \sigma(x')|^2 \leq \rho(|x - x'|)
\]
for some $\rho: \RR_+ \to \RR$ so that $\int_{0^+} \frac{1}{\rho(s)} ds = \infty$.  If $\sigma$ is locally Lipschitz and for all $x \in \RR$ we have 
\[
|\sigma(x)|^2 + |b(x)|^2 \leq K (1 + |x|^2)
\]
for some constant $K$, then (\ref{eq:ref-sde}) is exact.
\end{theorem}
\begin{proof}
First, pathwise uniqueness holds by Proposition \ref{prop:sou-unique}.  Now, strong existence holds up to some explosion time $\tau_E$ by Proposition \ref{prop:strong-exist}(a).  By \cite[Theorem IV.2.4]{IW81}, under these assumptions strong existence holds for the classical SDE associated to $\sigma, b$ for all initial conditions, so we may apply Proposition \ref{prop:strong-exist}(b) to conclude that $\tau_E = \infty$, giving strong existence and hence exactness.
\end{proof}

\section{An extension of the Pal-Shkolnikov approach to intertwining diffusions} \label{sec:ps-ext}

The goal of this section is to introduce the Pal-Shkolnikov approach to intertwining diffusions given in \cite{PS} and to prove Theorem \ref{thm:ps-ext}, which provides an extension of \cite[Theorem 3]{PS} to non-Brownian diffusion terms required in our work.

\subsection{Preliminaries on diffusion processes} \label{sec:reg}

In this section, we review a few notions on diffusion processes which will be used in the rest of the section.  For a domain $\cZ \subset \RR^m$, denote by $B(\cZ)$ the space of bounded measurable functions on $\cZ$.  Denote by $C_0(\cZ)$ the space of continuous functions $f: \cZ \to \RR$ vanishing at infinity, meaning for any $\eps > 0$, there exists a compact set $K \subset \cZ$ so that $f(z) < \eps$ on $K^c$; equip $C_0(\cZ)$ with the uniform norm.  Finally, denote by $C^2_c(\cZ)$ the space of smooth compactly supported functions on $\cZ$ which are continuously differentiable to order at least $2$.

For a diffusion $Z(t)$ on $\cZ$ with semigroup $P^Z_t$ acting on $B(\cZ)$, we recall that $Z(t)$ is Feller if $P^Z_t$ preserves $C_0(\cZ)$ for $t \geq 0$ and $\lim_{t \to 0} T_tf(x) = f(x)$ for all $x \in \cZ$ and $f \in C_0(\cZ)$.  For a Feller diffusion $Z(t)$, the generator $\AA^Z$ is a closed operator on its domain $\dom(\AA^Z) \subset C_0(\cZ)$, and $\dom(\AA^Z)$ is dense in $C_0(\cZ)$.  Recall that a linear subspace $\cD \subset \dom(\AA^Z)$ is a core for $\AA^Z$ if the closure of $\AA^Z|_\cD$ is $\AA^Z$.

\subsection{The setting of Pal-Shkolnikov}

Consider two diffusions $X$ and $Y$ on domains $\XX \subset \RR^m$ and $\YY \subset \RR^n$ with generators
\begin{align*}
\AA^X &= \frac{1}{2} \sum_{i = 1}^m \sum_{j = 1}^m a_{ij}(x) \partial_{x_i}\partial_{x_j} + \sum_{i = 1}^m b_i(x) \partial_{x_i}\\
\AA^Y &= \frac{1}{2} \sum_{i = 1}^n \sum_{j = 1}^n \rho_{ij}(y) \partial_{y_i}\partial_{y_j} + \sum_{i = 1}^n \gamma_i(y) \partial_{y_i},
\end{align*}
where $(a_{ij}), (\rho_{ij})$ are continuous on the interiors of $\XX$ and $\YY$ and take values in $m \times m$ and $n \times n$ positive semidefinite matrices and $b, \gamma$ are continuous on the interiors of $\XX$ and $\YY$. Let $L$ be a stochastic transition operator mapping $C_0(\XX)$ to $C_0(\YY)$.  We now recall the notion of an intertwining of diffusions given in \cite{PS}.  
\begin{define}[{\cite[Definition 2]{PS}}] \label{def:ps-def}
A process $Z = (Z_1, Z_2)$ is an intertwining of diffusions $X$ and $Y$ with link operator $L$ if:
\begin{itemize}
\item[(i)] $Z_1 \overset{d} = X$ and $Z_2 \overset{d}= Y$, where $\overset{d}=$ denotes equality in law, and 
\[
\EE[f(Z_1(0)) \mid Z_2(0) = y] = (Lf)(y),
\]
for all bounded Borel measurable functions $f$ on $D(y)$.

\item[(ii)] The transition semigroups $P_t$ and $Q_t$ of $Z_1$ and $Z_2$ are intertwined, meaning that $Q_t L = L P_t$ for all $t \geq 0$.

\item[(iii)] The process $Z_1$ is Markovian with respect to the joint filtration generated by $(Z_1, Z_2)$.

\item[(iv)] For any $s \geq 0$, conditional on $Z_2(s)$, the random variable $Z_1(s)$ is independent of $\{Z_2(u), 0 \leq u \leq s\}$ and is conditionally distributed according to $L$.
\end{itemize}
\end{define}

\begin{remark}
In the December 2015 version of \cite{PS}, there is the additional condition
\begin{itemize}
\item[(v)] For any $t \geq 0$, conditional on $Z_2(0)$ and $Z_1(t)$, the random variables $Z_1(0)$ and $Z_2(t)$ are independent.
\end{itemize}
in the definition of intertwining diffusion.  However, this condition will be removed in an update to \cite{PS} (private communication), and therefore we omit it here.
\end{remark}

Let $D \subset \XX \times \YY$ be a domain with polyhedral closure, and let $D(x, -) := \{y \mid (x, y) \in D\}$ and $D(y) := \{x \mid (x, y) \in D\}$.  We are interested in cases where $Z$ takes values in the domain $D$.  We consider the following three assumptions on $X$, $Y$, $D$, and $L$.

\begin{assump} \label{ass:proc}
One of the following conditions holds for processes $X$ and $Y$.  In both cases, we say that the resulting diffusion is \textit{regular}.
\begin{itemize}
\item[(a)] \textbf{No boundary conditions:} The martingale problem corresponding to $\AA^X$ on $\XX$ with no boundary conditions is well-posed in the sense of \cite{SV1, SV2}, and its solution $X$ is a Feller process whose generator admits the core $\cD(\AA^X) := C^2_c(\XX)$.

\item[(b)] \textbf{Neumann boundary conditions:} The domain $\XX$ is either (i) smooth or (ii) a convex polytope.  With $n(x)$ denoting the set of inward normal vectors to $\partial \XX$ at $x \in \partial \XX$.  For $x \in \partial \XX$, let $U_1(x)$ denote a cone in $\RR^m$ so that $\langle u, n(x) \rangle > 0$ for each $u \in U_1(x)$ which in case (i) is the span of a smooth and nowhere vanishing vector field and in case (ii) is the span of a smooth and nowhere vanishing vector field for each face of $\XX$.  The submartingale problem corresponding to $\AA^X$ and reflection cone defined by $U_1$ is well-posed in the sense of \cite{KR}, and its solution is a Feller process whose generator $\AA^X$ admits the core 
\[
\cD(\AA^X) := \{f \in C^2_c(\XX) \mid \langle \nabla f(x), u\rangle = 0 \text{ for $x\in \partial \XX$ and $u \in U_1(x)$}\}.
\]
\end{itemize}
For $Y$, replace $\AA^X$ by $\AA^Y$, $\XX$ by $\YY$, and $U_1$ by $U_2$ and impose case (i) in (b). 
\end{assump}

\begin{assump} \label{ass:domain}
We consider the following assumptions on the domain $D$.
\begin{itemize}
\item[(a)] The projection of $D$ on $\RR^m$ is $\XX$, and the projection of $D$ on $\RR^n$ is $\YY$.  

\item[(b)] For each $y \in \YY$, the domain $D(y) := \{x \mid (x,y) \in D\}$ has piecewise smooth boundary $\partial D(y)$.  Denote the pieces by $\partial D(y)_i$ so that $\partial D(y) = \bigcup_k \partial D(y)_k$. On each $\partial D(y)_k$, the projection to the $x$-coordinate of $\partial D(y)_k$ can be parametrized as the diffeomorphic image $(x_k(y, \xi), y)$ of a smooth function $x_k(y, -)$.

\item[(c)] At each point $x \in \partial D(y)_k$, the directional derivatives $\Psi^j_k$ of the boundary point $x = x_k(y, \xi)$ with respect to changes in the coordinates $y_j$ exist.  Let $\eta_k$ be the unit outward normal vector on $\partial D(y)_k$.
\end{itemize}
\end{assump}

\begin{assump} \label{ass:link}
We consider the following assumptions on the link operator $L$.
\begin{itemize}
\item[(a)] The link operator $L$ maps $C_0(\XX)$ to $C_0(\YY)$ and is given by the integral kernel
\[
(Lf)(y) := \int_{D(y)} f(x) \Lambda(y, x) dx
\]
for a non-negative link function $\Lambda(y, x): D \to \RR$.

\item[(b)] The link function $\Lambda$ is continuously differentiable in $x$ on a neighborhood of $\partial \XX \times \YY \cup \partial D$.

\item[(c)] The link function $\Lambda$ is twice continuously differentiable in $y$ on a neighborhood $U_\partial$ of $\partial D$ in $\XX \times \YY$.

\item[(d)] For every $x \in \XX$, $\Lambda(-, x)$ may be extended to a twice continuously differentiable function on $\YY$ so that $\AA^Y \Lambda$ is continuous on $D$ and bounded on $\bigcup_{x \in K} D(x, -)$ for any compact $K \subset \XX$.
\end{itemize}
\end{assump}

\subsection{Formulation of the intertwining diffusion process}

We now formulate a SDER whose weak solution will provide an intertwining diffusion of $X$ and $Y$ with link $L$ under an additional compatibility condition to be formulated later.  For $x \in \partial D(y)_k$, define the vector
\begin{equation} \label{eq:sder-push}
u_k(x, y) := \sum_{i, j = 1}^n \rho_{ij}(y) \langle \Psi^i_k, \eta_k \rangle 1_j.
\end{equation}
Define the set-valued mapping $U(x, y)$ on $\partial D(x, -)$ to be the cone with vertex at $0$ spanned by $u_k(x, y)$ for all $k$ so that $x \in \partial D(y)_k$.  Consider the SDER on domain $D$ for $Z = (Z_1, Z_2)$ given by 
\begin{align} \label{eq:sder-ps}
dZ_1(t) &= \sigma_X(Z_1(t)) dB_X(t) + b(Z_1(t)) dt + d\Phi_1(t)\\ \nonumber
dZ_2(t) &= \sigma_Y(Z_2(t)) dB_Y(t) + \Big(\gamma(Z_2(t)) + \langle \rho(Z_2(t)), \nabla_{z_2} [\log \Lambda(Z_2(t), Z_1(t))] \rangle\Big) dt + d\Phi_2(t) + d\Phi(t),
\end{align}
with the quantities satisfying
\begin{itemize}
\item $a(x) = \sigma_X(x) \sigma_X(x)^T$;
\item $\rho(y) = \sigma_Y(y) \sigma_Y(y)^T$;
\item $\Phi_1(t)$ is $0$ if Assumption \ref{ass:proc}(a) holds and otherwise is a bounded variation process with an auxiliary function $\phi_1(s)$ so that $\Phi_1(0) = 0$, $\int_0^\infty 1_{Z_1(t) \notin\partial \XX} d|\Phi_1|(t) = 0$, $\Phi_1(t) = \int_0^t \phi_1(s) d|\Phi_1|(s)$, and $\phi_1(s)$ is in the same direction as $U_1(Z_1(s))$;
\item $\Phi_2(t)$ satisfies the same conditions as $\Phi_1(t)$ with $U_1$ replaced by $U_2$ and $\XX$ replaced by $\YY$;
\item $\Phi(t)$ is a bounded variation process with an auxiliary function $\phi(s)$ so that it satisfies $\Phi(0) = 0$, $\int_0^\infty 1_{Z_2(t) \notin \partial D(Z_1(t), -)} d|\Phi|(t) = 0$, $\Phi(t) = \int_0^t \phi(s) d|\Phi|(s)$, and $\phi(s) \in U(Z_1(s), Z_2(s))$.
\end{itemize}
Note that the SDER (\ref{eq:sder-ps}) corresponds to the formulation of the Skorokhod problem given in \cite[Remark 2.3]{KR}; in particular, we require that $\Phi_1, \Phi_2, \Phi$ be bounded variation and admit auxiliary functions $\phi_1, \phi_2, \phi$. For $z = (x, y) \in \partial D$, define $\wtilde{U}(z)$ to be the cone spanned by the elements of $U(z)$, $U_1(x)$, and $U_2(y)$ which are well-defined at $z$.  Following \cite{KR}, we define the domain
\[
\cV := \partial D - \{z \mid \text{there is an inward normal vector $n$ to $\partial D$ at $z$ so that $\langle n, u \rangle > 0$ for non-zero $u \in \wtilde{U}(z)$}\}.
\]
Define now the spaces of functions
\begin{align} \nonumber
\cH(\AA^Z) &:= \{f \in C^2_c(D) \oplus \RR \mid \langle u, \nabla f(z)\rangle \geq 0 \text{ for $u \in \wtilde{U}(z)$, $z \in \partial D$, $f$ locally constant at $\cV$}\}\\ \label{eq:wch-def}
\cD(\AA^Z) &:= \{f \in C^2_c(D) \mid \langle u, \nabla f(z) \rangle = 0 \text{ for $u \in \wtilde{U}(z)$, $z \in \partial D$}\}.
\end{align}
Explicitly, functions in $\cD(\AA^Z)$ satisfy Neumann boundary conditions of $X$ on $\partial \XX \times \YY \cap D$, Neumann boundary conditions of $Y$ on $\XX \times \partial \YY \cap D$, and Neumann boundary conditions on $\partial D(x, -)$ given by the vector field
\begin{equation} \label{eq:ps-ref-bound}
\sum_{i, j = 1}^n \rho_{ij} \langle \Psi^i_k, \eta_k\rangle \partial_{y_j} \text{ at $y \in \partial D(x, -)_k$}.
\end{equation}
Following \cite[Definition 2.9]{KR}, we say the law of the process $Z(t)$ solves the submartingale problem corresponding to (\ref{eq:sder-ps}) with generator
\begin{equation} \label{eq:z-gen-def}
\AA^Z := \AA^X + \AA^Y + \sum_{i = 1}^n \sum_{j = 1}^n \rho_{ij}(y) \partial_{y_i}[\log \Lambda(y, x)] \partial_{y_j}.
\end{equation}
and initial condition $z_0$ if
\begin{itemize}
\item we have $Z(0) = z_0$ a.s.;

\item we have $Z(t) \in \overline{D}$ a.s. for all $t \geq 0$;

\item for every $f \in \cH(\AA^Z)$, the process
\[
f(Z(t)) - f(Z(0)) - \int_0^t \AA^Zf(Z(s)) ds
\]
is a submartingale;

\item we have $\int_0^\infty 1_{\cV}(X(s)) ds = 0$ a.s..
\end{itemize}
In \cite[Theorem 2]{KR}, a criterion is given for a weak solution to (\ref{eq:sder-ps}) to be a solution of the submartingale problem.  We adapt this criterion to our setting and use it to characterize the weak solution as a diffusion process in Proposition \ref{prop:kr-domain} below.

\begin{prop}[{\cite[Theorem 2]{KR}}] \label{prop:kr-domain}
Let $Z(t)$ be a weak solution to (\ref{eq:sder-ps}).  If $\cV$ is the union of finitely many closed connected sets, then $Z(t)$ is a solution to the submartingale problem with initial condition $Z(0)$.  Further, if $f \in \cD(\AA^Z)$, then
\[
f(Z(t)) - f(Z(0)) - \int_0^t \AA^Z f(Z(s)) ds
\]
is a martingale.  In particular $Z(t)$ is a diffusion with generator given by $\AA^Z$ on $\cD(\AA^Z) \subset \dom(\AA^Z)$.
\end{prop}

\begin{remark}
While \cite[Theorem 2]{KR} does not address the martingale property for $f \in \cD(\AA^Z)$ explicitly, our formulation of (\ref{eq:sder-ps}) requires the solution to solve the Skorokhod problem instead of the extended Skorokhod problem.  As a result, the application of It\^o's lemma given in \cite[Section 6]{KR} yields this extension.
\end{remark}

\begin{remark}
As discussed in \cite[Remark 3]{PS}, for $x \in \partial D(y)_k$, there exists an outward normal vector $\widehat{\eta}_k$ to $\partial D(x, -)$ and some $c > 0$ so that $- \eta_k - c \widehat{\eta}_k$ is an inward normal vector to $\partial D$.  Further, for each $j$, the vector $1_{y_j} + \Psi^j_k$ is tangent to $\partial D$, meaning that $\langle \Psi^j_k, \eta_k\rangle = - c \langle \widehat{\eta}_k, 1_{y_j}\rangle$.  We conclude that
\[
\langle u_k(x, y), - \eta - c \widehat{\eta}_k \rangle = c^2 \sum_{i, j = 1}^n \rho_{ij}(y) \langle \Psi^i_k, \eta_k\rangle \langle \Psi^j_k, \eta_k\rangle \geq 0,
\]
since $(\rho_{ij})$ is positive semi-definite.  Therefore, if $(\rho_{ij})$ is strictly positive definite, then the domain $D$ and boundary conditions $U_1(x), U_2(y), U(z)$ are piecewise $C^2$ with continuous reflection as in \cite[Definition 2.11]{KR}, and \cite[Theorem 3]{KR} implies that any solution to the submartingale problem yields a weak solution to the SDER (\ref{eq:sder-ps}) stopped at the first hitting time of $\cV$.  To avoid complications with this stopping, we work with weak solutions to the SDER directly in what follows.
\end{remark}

\subsection{Existence of intertwiners for general diffusions}

We give conditions on $D, \Lambda, X, Y$ under which a weak solution to (\ref{eq:sder-ps}) will give an intertwining diffusion of $X$ and $Y$ with link $L$.
\begin{assump} \label{ass:compat}
We consider the following compatibility conditions between $X$ and $Y$ given $D$ and $\Lambda$.
\begin{itemize}
\item[(a)] We have $\sum_{j = 1}^n \langle \Psi^j_k, \eta_k\rangle U_{2, j} = 0$ on $\partial D(y)_k$ for $y \in \partial \YY$, where $U_{2, j}$ are the coordinates of $U_2$.

\item[(b)] For every $y \in \YY$, $\Lambda(y, \cdot)$ is a probability density on $D(y)$.

\item[(c)] The function $\Lambda$ is twice continuously differentiable on $\overline{D}$ and for any $f \in \cD(\AA^Z)$ satisfies
\begin{multline} \label{eq:int-eq}
\int_{D(y)} \Lambda (\AA^Xf) dx = \int_{D(y)} (\AA^Y\Lambda) f dx + \sum_{j = 1}^n \sum_k \gamma_j \int_{\partial D(y)_k} \Lambda f \langle \Psi^j_k, \eta_k\rangle d \theta(x)\\ + \frac{1}{2} \sum_{i, j = 1}^n \rho_{ij} \sum_k \int_{\partial D(y)_k}\Big(\div_{\partial x}(\Lambda f \Psi^j_k) + 2 \partial_{y_j}(\Lambda f)\Big) \langle \Psi^i_k, \eta_k\rangle d\theta(x),
\end{multline}
where $\div_{\partial x}(\Lambda f \Psi^j_k)$ is interpreted as the divergence on the subspace $\partial D(y)_k$.

\item[(d)] The function $\Lambda$ satisfies the boundary conditions $\langle \nabla_y \Lambda, U_2 \rangle = 0$ on $\partial \YY$.
\end{itemize}
\end{assump}

\begin{remark}
In Assumption \ref{ass:compat}, condition (\ref{eq:int-eq}) generalizes a combination of \cite[Equation (2.23)]{PS}, \cite[Equation (2.24)]{PS}, and the condition
\[
\eta_k = \sum_j \Psi^j_k \langle \Psi^j_k, \eta_k\rangle \qquad \text{ on $\partial D(y)_k$}
\]
to non-Brownian diffusion terms.  In particular, if we have the pointwise equality
\[
\AA^Y(\Lambda) = (\AA^X)^*\Lambda \qquad \text{ on $D(y)$}
\]
and the boundary compatibility conditions
\begin{multline} \label{eq:ps-bound}
\Lambda \langle b, \eta_k \rangle - \frac{1}{2} \Lambda \langle \div_x a, \eta_k \rangle - \langle \langle a, \eta_k \rangle, \nabla_x \Lambda\rangle\\ = \sum_j \Lambda \gamma_j \langle \Psi^j_k, \eta_k \rangle + \sum_{i, j = 1}^n \rho_{ij} \partial_{y_j}(\Lambda)\langle \Psi^i_k, \eta_k\rangle \text{ on $\partial D(y)_k$ for each $y \in \YY$}
\end{multline}
and
\begin{equation} \label{eq:ps-ref-bound2}
\langle a, \eta_k\rangle = \sum_{i, j = 1}^n \rho_{ij} \Psi^j_k \langle \Psi^i_k, \eta_k\rangle, \qquad \text{ on $\partial D(y)_k$ for $y \in \YY$},
\end{equation}
then (\ref{eq:int-eq}) is a consequence of repeated applications of the divergence theorem and the multidimensional Leibniz rule.
\end{remark}

In the following Theorem \ref{thm:ps-ext}, we give a criterion for a solution to (\ref{eq:z-gen-def}) to be an intertwining diffusion of $X$ and $Y$ under the hypothesis that $Z$ is a regular Feller diffusion.  The remainder of this section is devoted to the statement and proof of Theorem \ref{thm:ps-ext} and Proposition \ref{prop:geo-cond}, which gives conditions under which Assumption \ref{ass:compat}(c) holds.

\begin{theorem} \label{thm:ps-ext}
Suppose that $D, L, X, Y$ satisfy Assumptions \ref{ass:proc}, \ref{ass:domain}, \ref{ass:link}, and \ref{ass:compat} and that the SDER (\ref{eq:sder-ps}) has a weak solution $Z$ which is a regular Feller diffusion with generator $\AA^Z$ as defined in Assumption \ref{ass:proc}.  If the resulting process satisfies the initial condition
\[
\PP(Z_1(0) \in B \mid Z_2(0) = y) = \int_B \Lambda(y, x) dx \text{ for Borel $B \subset D(y)$},
\]
then $Z$ is an intertwining of $X$ and $Y$ with link $L$.
\end{theorem}

\begin{remark}
Theorem \ref{thm:ps-ext} generalizes \cite[Theorem 3]{PS}, which addresses the case where $a_{ij} = \delta_{ij}$ and $\rho_{kl} = \delta_{kl}$.
\end{remark}

\begin{prop} \label{prop:geo-cond}
Suppose that Assumptions \ref{ass:domain} and \ref{ass:link} hold.  Suppose that there exist domains $D_1(y) \subset \RR^{n_1}$ and $D_2(x^1) \subset \RR^{n_2}$ so that
\[
D(y) = \{(x^2, x^1) \mid x^1 \in D_1(y), x^2 \in D_2(x^1)\},
\]
that $\Lambda(y, x) = \Lambda_1(y, x^1) \Lambda_2(x^1, x^2)$ for some non-negative link functions $\Lambda_1, \Lambda_2$ which are twice continuously differentiable, and that there is a differential operator $\AA^{X_1}$ on $\RR^{n_1}$ given by
\[
\AA^{X_1} := \frac{1}{2} \sum_{i, j = 1}^{n_1} a_{ij}^1 \partial_{x^1_i} \partial_{x^1_j} + \sum_{i = 1}^{n_1} b_i^1 \partial_{x^1_i}.
\]
Let the faces of $\partial D_1(y)$ be given by $\{\partial D_1(y)_k\}_{k \in K}$, let $\eta^1_k$ denote the unit outward normal vector to $\partial D_1(y)_k$, and let $\Psi^{j, 1}_k$ denote the derivative of a boundary point on $\partial D_1(y)_k$ with respect to $y_j$.  For a face $\partial D_1(y)_k$ of $D_1(y)$, let 
\[
\partial D(y)_{\wk} = \{(x^2, x^1) \mid x^1 \in \partial D_1(y)_k\}
\]
be the corresponding face of $\partial D(y)$.  For $x^1 \in \partial D_1(y)_k$, let $\Psi^{\bullet, 2}_k$ be the gradient of a point in $D_2(x^1)_k$ with respect to $x^1$, and let $\Psi^{j, 2}_k$ be the gradient of that point with respect to $y_j$.  If these satisfy the conditions
\begin{itemize}
\item[(a)] for any $f \in \cD(\AA^X)$, we have pointwise that
\[
\int_{D_2(x^1)} \Lambda_2 (\AA^X f) dx^2 = \AA^{X_1} \int_{D_2(x^1)} \Lambda_2 f dx^2;
\]

\item[(b)] we have $(\AA^{X_1})^*\Lambda_1 = \AA^Y \Lambda_1$ pointwise;

\item[(c)] on $\partial D_1(y)_k$ we have
\[
\langle b^1, \eta^1_k\rangle - \frac{1}{2} \langle \div_{x^1}(a^1), \eta^1_k\rangle - \langle \langle a^1, \eta^1_k\rangle, \nabla_{x^1} \log \Lambda_1\rangle = \sum_{j = 1}^n \gamma_j \langle \Psi^{j, 1}_k, \eta_k^1\rangle + \sum_{i, j = 1}^n \rho_{ij} \partial_{y_j} \log \Lambda_1 \langle \Psi^{i, 1}_k, \eta_k^1\rangle;
\]

\item[(d)] on $\partial D_1(y)_k$ we have 
\[
\langle a^1, \eta^1_k\rangle = \sum_{i, j = 1}^n \rho_{ij} \Psi^{j, 1}_k \langle \Psi^{i, 1}_k, \eta^1_k\rangle;
\]

\item[(e)] on any face of $\partial D(y)$ which does not come from a face of $\partial D_1(y)$, we have $\langle \Psi^i_k, \eta_k\rangle = 0$ for all $i$;

\item[(f)] on the face $\partial D(y)_{\wk}$ coming from $\partial D_1(y)_k$, $\langle \Psi^i_k, \eta_k\rangle$ is piecewise constant for all $i$.
\end{itemize}
then Assumption \ref{ass:compat}(c) is satisfied for $X$ and $Y$.
\end{prop}

\begin{remark}
If there is a diffusion $X_1$ with generator $\AA^{X_1}$, then Conditions (c) and (d) of Proposition \ref{prop:geo-cond} correspond to conditions (\ref{eq:ps-bound}) and (\ref{eq:ps-ref-bound2}) for the intertwining of $X_1$ and $Y$.  If $X_1$ is regular Feller, then we could simply apply Theorem \ref{thm:ps-ext} with $X_1$ instead of $X$. In our applications, we require Proposition \ref{prop:geo-cond} to handle cases where we are able to establish the regularity of $X$ but not the regularity of $X_1$.
\end{remark}

\subsection{Proof of Theorem \ref{thm:ps-ext}}

We follow the same steps as in the proof of \cite[Theorem 3]{PS}.  We work over a probability space on $C([0, \infty), D)$, the space of continuous paths $[0, \infty) \to D$, with the standard Borel $\sigma$-algebra and a probability measure $\PP$ given by the law of $Z$.  Filter this probability space by $\{\FF_t, t \geq 0\}$, the filtration generated by the coordinate maps and augmented by the common null sets under $(\PP_z, z \in D)$.  We consider also $\{\FF^X_t, t \geq 0\}$ and $\{\FF^Y_t, t \geq 0\}$, the right-continuous complete subfiltrations of $\{\FF_t, t \geq 0\}$ generated by the first $m$ and last $n$ coordinate processes in $C([0, \infty), D)$, respectively.  We first prove Lemma \ref{lem:dom-trans} on the compatibility of the Gibbs measure with domains and Lemma \ref{lem:ibv-prob}, which generalizes Step 2 in the proof of \cite[Theorem 3]{PS}.  
        
\begin{lemma} \label{lem:dom-trans}
For any $f \in \cD(\AA^Z)$, the function
\[
\wtilde{f}(y) = \int_{D(y)} \Lambda(y, x) f(x, y) dx
\]
lies in $\cD(\AA^Y)$.
\end{lemma}
\begin{proof}
Because $f \in C_c^2(D)$, we see that $\supp(f)$ is a compact subset of $D$ with $C^2$ boundary.  In particular, this implies that
\[
\wtilde{f}(y) = \int_{\supp(f) \cap D(y)} \wLambda(y, x) f(x, y) dx
\]
lies in $C_c^2(\YY)$.  In addition, on $\partial \YY$, we have that
\begin{multline*}
\langle \nabla \wtilde f, U_2(y) \rangle = \int_{\supp(f) \cap D(y)} \langle \nabla_y \Lambda, U_2(y)\rangle f dx + \int_{\supp(f) \cap D(y)} \Lambda \langle \nabla_y f, U_2(y)\rangle dx\\ + \sum_{j, k} \int_{\partial D(y)_k} \Lambda f \langle \Psi^j_k, \eta_k\rangle, U_{2, j}(y) d\theta(x) = 0,
\end{multline*}
where the terms in the sum vanish by Assumption \ref{ass:compat}(d), the fact that $f \in \cD(\AA^Z)$, and Assumption \ref{ass:compat}(a), respectively.  We conclude that $\wtilde{f} \in \cD(\AA^Y)$ as desired.
\end{proof}

\begin{lemma} \label{lem:ibv-prob}
For any $f \in \cD(\AA^Z)$, the functions
\[
u(t, y) := \EE\left[f(Z_1(t), Z_2(t)) \mid Z_2(0) = y\right], \qquad (t, y) \in [0, \infty) \times \YY
\]
and
\[
u'(t, y) := \EE\left[\int_{D(Y(t))} \Lambda(Y(t), x) f(x, Y(t)) dx \mid Y(0) = y\right], \qquad (t, y) \in [0, \infty) \times \YY
\]
coincide.
\end{lemma}
\begin{proof}
We check that both $u(t, y)$ and $u'(t, y)$ lie in $\dom(\AA^Y)$,  are continuously differentiable with respect to the uniform norm on $C_0(\YY)$, and solve the Kolmogorov forward equation $\partial_t u = \AA^Y u$ with $u(0, y) = \int_{D(y)} \Lambda(y, x) f(x, y) dx$.  The desired equality then follows from the uniqueness of \cite[Proposition II.6.2]{EN}.  

First, notice that
\[
u(0, y) = \int_{D(y)} \Lambda(y, x) f(x, y) dx = u'(0, y),
\]
and that this lies in $\cD(\AA^Y)$ by Lemma \ref{lem:dom-trans}.  Now, we have $u'(t, y) := Q_t \wtilde{f}(y)$, hence by \cite[Proposition II.6.2]{EN}, $u'(t, y)$ lies in $\dom(\AA^Y)$, is continuously differentiable, and satisfies $\partial_t u' = \AA^Y u'$. For $u(t, y)$, define
\[
v(t, x, y) := \EE[f(Z_1(t), Z_2(t)) \mid Z_1(0) = x, Z_2(0) = y].
\]
By definition, we have that $v(t, x, y) \in \dom(\AA^Z)$ and that
\begin{equation} \label{eq:uv-rel}
u(t, y) = \int_{D(y)} \Lambda(y, x) v(t, x, y) dx.
\end{equation}
Because $\cD(\AA^Z)$ is a core for $\dom(\AA^Z)$, we may find a sequence $v_l(t) \in \cD(\AA^Z)$ so that $v_l(t) \to v(t)$ and $\AA^Z v_l(t) \to \AA^Z v(t)$ uniformly.  Define the function
\[
u_l(t, y) = \int_{D(y)} \Lambda(y, x) v_l(t, x, y) dx,
\]
which lies in $\cD(\AA^Y)$ by Lemma \ref{lem:dom-trans}; observe that $u_l(t) \to u(t)$ uniformly. It suffices now to check that $\lim_{l \to \infty} \AA^Y u_l$ exists and equals $\partial_t u$, as the closedness of $\AA^Y$ would then imply that $u = \lim_{l \to \infty} u_l$ lies in $\dom(\AA^Y)$ and solves $\partial_t u = \AA^Y u$.  The remainder of the proof is devoted to checking this fact.

\noindent \textbf{Computing the time derivative:}  Because $v$ solves the Kolmogorov forward equation for $Z$ and $\AA^Z$ is closed, we have that
\begin{equation} \label{eq:u-t-dev1}
\partial_t u = \int_{D(y)} \Lambda(y, x) \partial_t v(t, x, y) dx = \int_{D(y)} \Lambda(y, x) \AA^Z v(t, x, y) dx = \lim_{l \to \infty} \int_{D(y)} \Lambda(y, x) \AA^Z v_l(t, x, y) dx
\end{equation}
Recalling Assumption \ref{ass:compat}(c) and the fact that
\[
\AA^Z = \AA^X + \AA^Y + \sum_{i = 1}^n \sum_{j = 1}^n \rho_{ij} \partial_{y_i}[\log \Lambda] \partial_{y_j},
\]
we conclude that
\begin{multline} \label{eq:u-t-dev2}
\partial_t u = \int_{D(y)} \Big((\AA^Y \Lambda) v_l + \Lambda (\AA^Y v_l) + \Lambda \sum_{i, j = 1}^n \rho_{ij} \partial_{y_i}[\log \Lambda] \partial_{y_j} v_l\Big) dx \\
+ \sum_{j = 1}^n \sum_k \gamma_j \int_{\partial D(y)_k} \Lambda v_l \langle \Psi^j_k, \eta_k\rangle d \theta(x)\\ + \frac{1}{2} \sum_{i, j = 1}^n \rho_{ij} \sum_k \int_{\partial D(y)_k}\Big(\div_{\partial x}(\Lambda v_l \Psi^j_k) + 2 \partial_{y_j}(\Lambda v_l)\Big) \langle \Psi^i_k, \eta_k\rangle d\theta(x).
\end{multline}

\noindent \textbf{Computing the action of the generator:} We now compute $\AA^Y u_l(t)$; since $u_l(t) \in \cD(\AA^Y)$, we may compute with $\AA^Y$ as a differential operator. By the multidimensional Leibniz rule, we have 
\begin{align*}
\partial_{y_j} u_l &= \sum_k \int_{\partial D(y)_k} \Lambda v_l \langle \Psi^j_k, \eta_k\rangle d\theta(x) + \int_{D(y)} \partial_{y_j}(\Lambda v_l) dx\\
\partial_{y_i} \partial_{y_j} u_l &= \sum_k \int_{\partial D(y)_k} \div_{\partial x}(\Lambda v_l\langle \Psi^j_k, \eta_k\rangle \Psi^i_k) d\theta(x) +\sum_k \int_{\partial D(y)_k} \partial_{y_i}(\Lambda v_l \langle \Psi^j_k, \eta_k\rangle) d\theta(x)\\
&\phantom{=} + \sum_k \int_{\partial D(y)_k} \partial_{y_j}(\Lambda v_l)\langle \Psi^i_k, \eta_k\rangle d\theta(x) + \int_{D(y)} \partial_{y_i}\partial_{y_j}(\Lambda v_l) dx.
\end{align*}
We conclude that
\begin{multline*}
\AA^Y u_l = \int_{D(y)} \AA^Y(\Lambda v_l) dx + \sum_{j = 1}^n \gamma_j \sum_k \int_{\partial D(y)_k} \Lambda v_l \langle \Psi^j_k, \eta_k\rangle d\theta(x) \\
+ \frac{1}{2} \sum_{i, j = 1}^n \rho_{ij} \sum_k \int_{\partial D(y)_k} \Big(\div_{\partial x}(\Lambda v_l \langle \Psi^j_k, \eta_k\rangle \Psi^i_k) + 2 \partial_{y_j}(\Lambda v_l \langle \Psi^i_k, \eta_k\rangle)\Big) d\theta(x).
\end{multline*}
Note that 
\[
\AA^Y(\Lambda v_l) = (\AA^Y \Lambda) v_l + \Lambda (\AA^Y v_l) + \sum_{i, j = 1}^n \rho_{ij} \partial_{y_i} \Lambda \partial_{y_j} v_l = (\AA^Y \Lambda) v_l + \Lambda (\AA^Y v_l) + \Lambda \sum_{i, j = 1}^n \rho_{ij} \partial_{y_i}[\log \Lambda] \partial_{y_j} v_l,
\]
which implies that
\begin{multline} \label{eq:u-y-dev2}
\AA^Y u_l = \int_{D(y)} \Big((\AA^Y\Lambda) v_l + \Lambda \AA^Yv_l + \Lambda \sum_{i, j = 1}^n \rho_{ij} \partial_{y_i}[\log \Lambda] \partial_{y_j} v_l\Big)  dx\\ + \sum_{j = 1}^n \gamma_j \sum_k \int_{\partial D(y)_k} \Lambda v_l \langle \Psi^j_k, \eta_k\rangle d\theta(x) \\
 + \frac{1}{2} \sum_{i, j = 1}^n \rho_{ij} \sum_k \int_{\partial D(y)_k} \Big(\div_{\partial x}(\Lambda v_l \Psi^j_k) + 2 \partial_{y_j}(\Lambda v_l) \Big) \langle \Psi^i_k, \eta_k\rangle d\theta(x).
\end{multline}

\noindent \textbf{Comparing the two sides:} Comparing (\ref{eq:u-t-dev2}) and (\ref{eq:u-y-dev2}), we conclude that
\[
\AA^Y u_l = \int_{D(y)} \Lambda \AA^Z v_l dx.
\]
Recalling (\ref{eq:u-t-dev1}) and taking the limit as $l \to \infty$, we conclude that $\lim_{l \to \infty} \AA^Y u_l$ exists.  Since $\AA^Y$ is closed and $u_l \to u$, we conclude that $u \in \dom(\AA^Y)$ and that
\[
\AA^Y u = \lim_{l \to \infty} \AA^Y u_l = \lim_{l \to \infty} \int_{D(y)} \Lambda \AA^Z v_l dx = \partial_t u,
\]
which completes the proof.
\end{proof}

To complete the proof, we now verify conditions (i) to (iv) in the definition of an intertwining diffusion.

\noindent \textbf{Claim 1:} We claim in (i) that $Z_1$ is Markov with respect to its own filtration and that $Z_1 \overset{d} = X$.  The proof is the same as in Step 1 of \cite[Theorem 1]{PS}.

\noindent \textbf{Claim 2:} We now establish (iii).  Because $\cD(\AA^X)$ is a core for $\AA^X$, by the monotone class and convergence theorems, it suffices to check that for $f \in \cD(\AA^X)$, we have
\[
\EE[f(Z_1(t)) \mid Z_1(0) = x, Z_2(0) = y] = \EE[f(Z_1(t)) \mid Z_1(0) = x], \qquad (t, x, y) \in [0, \infty) \times D.
\]
Write $v(t, x, y)$ for the LHS and $u(t, x)$ for the RHS.  Notice that $u(t, x) \in \dom(\AA^X)$ solves the Kolmogorov forward equation $\partial_t u = \AA^X u$ with initial condition $u(0, x) = f(x)$.  Because $\cD(\AA^X)$ is a core for $\AA^X$, for each $t$, we may choose a uniformly converging sequence $u_l(t) \to u(t)$ with $u_l(t) \in \cD(\AA^X)$ so that $\AA^X u_l(t) \to \AA^X u(t)$.  Viewed as a function on $D$, $u_l(t)$ lies in $\cD(\AA^Z)$, since the additional Neumann boundary conditions hold trivially.  Therefore, we see that $\AA^Z u_l(t) = \AA^X u_l(t)$, hence $\lim_{l \to \infty} \AA^Z u_l(t)$ exists and equals $\AA^X u(t)$.  Because $\AA^Z(t)$ is closed, we conclude that $u(t) \in \dom(\AA^Z)$, $\AA^Z u(t) = \AA^X u(t)$, and hence that $\wtilde{u}(t, x, y) = u(t, x)$ solves the Kolmogorov forward equation $\partial_t \wtilde{u} = \AA^Z \wtilde{u}$ with initial condition $\wtilde{u}(0, x, y) = f(x)$.  On the other hand, $v(t, x, y) \in \dom(\AA^Z)$ is another solution to this equation with $v(0, x, y) = f(x)$, so we conclude that $v(t, x, y) = \wtilde{u}(t, x, y) = u(t, x)$ by uniqueness of solutions to the Kolmogorov forward equation given by \cite[Proposition II.6.2]{EN}.

\noindent \textbf{Claim 3:} We now establish (ii).  Because $\cD(\AA^X)$ is a core for $\AA^X$, it suffices to check that $Q_t Lf = LP_tf$ for all $f \in \cD(\AA^X)$ and $t \geq 0$.  By (\ref{eq:uv-rel}) and Claim 2, we find that
\begin{align*}
Q_t Lf = \EE\left[\int_{D(Z_2(t))} \Lambda(Z_2(t), x) f(x) dx \mid Z_2(0) = y\right] &= \int_{D(y)} \Lambda(y, x) \EE[f(Z_1(t)) \mid Z_1(0) = x, Z_2(0) = y] dx \\
&= \int_{D(y)} \Lambda(y, x) \EE[f(Z_1(t)) \mid Z_1(0) = x] dx = LP_tf.
\end{align*}

\noindent \textbf{Claim 4:} We now establish (iv). It suffices to check that for any bounded measurable function $f$ on $\XX$, $k \in \NN$, and distinct times $0 = t_0 < t_1 < \cdots < t_k = t$, we have
\[
\EE[f(Z_1(t_k)) \mid Z_2(t_0), \ldots, Z_2(t_k)] = (Lf)(Z_2(t_k)).
\]
This is equivalent to the fact that for any bounded Borel measurable function $g$ on $\YY^{k}$ and $y \in \YY$, we have 
\[
\EE[f(Z_1(t_k)) g(Z_2(t_1), \ldots, Z_2(t_k)) \mid Z_2(t_0) = y] = \EE[(Lf)(Z_2(t_k)) g(Z_2(t_1), \ldots, Z_2(t_k)) \mid Z_2(t_0) = y].
\]
We proceed by induction on $k$.  For $k = 1$, by the monotone class theorem, it suffices to replace $f$ and $g$ by indicator functions $f(x) = 1_A(x)$ and $g(y) = 1_B(y)$ of open boxes $A \subset \XX$ and $B \subset \YY$.  By Lemma \ref{lem:ibv-prob}, for any $h \in \cD(\AA^Z)$, we have
\[
\EE[h(Z_1(t_1), Z_2(t_1)) \mid Z_2(t_0) = y] = \EE\left[ \int_{D(Z_2(t_1))} \Lambda(Z_2(t_1), z_1) h(z_1, Z_2(t_1)) dz_1 \mid Z_2(t_0) = y \right].
\]
Since $\cD(\AA^Z)$ is a core for $\AA^Z$ and $\dom(\AA^Z)$ is dense in $C_0(D)$, this holds also for any $h \in C_0(D)$ and hence for any bounded measurable $h$ on $D$.  In particular, it holds for $h(x, y) = 1_{(A \times B) \cap D}(x, y) = 1_A(x) 1_B(y)$, giving the claim.  For the inductive step, notice that 
\begin{align*}
\EE[f(Z_1(t_{k + 1})) \mid Z_2(t_0), \ldots, Z_2(t_{k+1})] &= \int_{D(Z_2(t_k))} \Lambda(Z_2(t_k), x) \EE[f(Z_1(t_{k+1})) \mid Z_1(t_k) = x, Z_2(t_k), Z_2(t_{k+1})] dx\\
&= \EE[f(Z_1(t_{k+1})) \mid Z_2(t_k), Z_2(t_{k+1})]\\
&= (Lf)(Z_2(t_{k+1})).
\end{align*}
where in the first step we use that the law of $Z_1(t_k)$ given $Z_2(t_0), \ldots, Z_2(t_k)$ is $\Lambda(Z_2(t_k), \cdot)$ by the inductive hypothesis and in the last step we use time homogeneity and the $k = 1$ case.

\noindent \textbf{Claim 5:} We show in (i) that $Z_2 \overset{d} = Y$.  It suffices to check that for any $k \in \NN$ and distinct times $0 = t_0 < t_1 < \cdots < t_k = t$ we have
\[
\EE[f(Z_2(t)) \mid Z_2(t_0), \ldots, Z_2(t_{k-1})] = Q_{t - t_{k-1}} f(Z_2(t_{k - 1})).
\]
We induct on $k$.  For $k = 1$, by the monotone class theorem, it suffices to check the condition for $f(y) = 1_A(y)$ for any open $A \subset \YY$.  By Lemma \ref{lem:ibv-prob}, for any $h \in \cD(\AA^Z)$ we have
\[
\EE[h(Z_1(t), Z_2(t)) \mid Z_2(t_0) = y] = \EE\left[\int_{D(Y(t))} \Lambda(Y(t), x) h(x, Y(t)) dx \mid Y(t_0) = y\right].
\]
Now, since $\cD(\AA^Z)$ is a core for $\AA^Z$ and $\dom(\AA^Z)$ is dense in $C_0(D)$, this holds for all $h \in C_0(D)$ and hence for all bounded measurable $h$ and in particular $h(x, y) = 1_A(y)$, giving the desired. For the inductive step, notice that 
\begin{align*}
\EE[f(Z_2(t)) \mid Z_2(t_0), \ldots, Z_2(t_{k})] &= \int_{D(Z_2(t_k))} \Lambda(Z_2(t_k), x) \EE[f(Z_2(t)) \mid Z_1(t_k) = x, Z_2(t_0), \ldots, Z_2(t_k)]\\
&= \int_{D(Z_2(t_k))} \Lambda(Z_2(t_k), x) \EE[f(Z_2(t)) \mid Z_1(t_k) = x, Z_2(t_k)]\\
&= \EE[f(Z_2(t)) \mid Z_2(t_k)]\\
&= Q_{t - t_k} f(Z_2(t_k)),
\end{align*}
where in the first equality we apply Claim 4 and in the second we use the Markov property of $Z$.

\subsection{Proof of Proposition \ref{prop:geo-cond}}

Applying given condition (a) and noting that any $f \in \cD(\AA^Z)$ lies in $\cD(\AA^X)$ when evaluated at a fixed $y$, we see that
\[
\int_{D(y)} \Lambda (\AA^X f) dx = \int_{D_1(y)} \Lambda_1 \int_{D_2(x^1)} \Lambda_2 (\AA^X f) dx^2 dx^1= \int_{D_1(y)} \Lambda_1 \AA^{X_1} \int_{D_2(x^1)} \Lambda_2 f dx^2 dx^1.
\]
Now, let $d\theta(x^1)$ be the Lebesgue surface measure on $\partial D_1(y)$.  By the divergence theorem, we find that for 
\[
\wf(x^1) := \int_{D_2(x^1)} \Lambda_2 f dx^2,
\]
we have
\[
\int_{D_1(y)} \Lambda_1 \langle b^1, \nabla_{x^1} \wf\rangle dx_1 = - \int_{D_1(y)} \Lambda_1 \wf \div_{x_1} (b^1) dx_1 - \int_{D_1(y)} \wf \langle b^1, \nabla_{x^1} \Lambda_1\rangle dx^1 + \sum_{k \in K} \int_{\partial D_1(y)_k} \wf \Lambda_1 \langle b^1, \eta^1_k\rangle d\theta(x^1)
\]
and that 
\begin{multline*}
\frac{1}{2} \sum_{i, j = 1}^{n_1} \int_{D_1(y)} \Lambda_1 a^1_{ij} \partial_{x^1_i} \partial_{x^1_j} \wf dx^1 = \frac{1}{2} \int_{D_1(y)} \wf \langle a^1, \nabla_{x^1}^2 \Lambda_1\rangle dx^1  + \frac{1}{2} \int_{D_1(y)} \Lambda_1 \wf \div_{x^1}(\div_{x^1}(a)) dx^1\\
 + \int_{D_1(y)} \wf \langle \div_{x^1}(a^1), \nabla_{x^1} \Lambda_1 \rangle dx^1
 -\frac{1}{2} \sum_{k \in K} \int_{\partial D_1(y)_k} \wf \Lambda_1 \langle \div_{x^1}(a^1), \eta^1_k\rangle d \theta(x^1)\\ 
- \frac{1}{2} \sum_{k \in K} \int_{\partial D_1(y)_k} \wf\langle \langle a^1, \eta^1_k\rangle, \nabla_{x^1} \Lambda_1 \rangle d \theta(x) - \frac{1}{2}\sum_{k \in K} \int_{\partial D_1(y)_k} \wf \langle \nabla_{x^1}\Lambda_1, \langle a^1, \eta^1_k \rangle \rangle d \theta(x^1) \\
+ \frac{1}{2} \sum_{k \in K} \int_{\partial D_1(y)_k} \div_{\partial x^1}(\wf \Lambda_1 \langle a^1, \eta_k^1\rangle) d\theta(x^1).
\end{multline*}
Putting these together, we conclude that
\begin{multline*}
\int_{D_1(y)} \Lambda_1 \AA^{X_1} \wf dx^1 = \int_{D_1(y)} (\AA^{X_1})^*(\Lambda_1) \wf dx^1 + \sum_{k \in K} \int_{\partial D_1(y)_k} \wf \Lambda_1 \langle b^1, \eta^1_k\rangle d\theta(x^1)\\ -\frac{1}{2} \sum_{k \in K} \int_{\partial D_1(y)_k} \wf \Lambda_1 \langle \div_{x^1}(a^1), \eta^1_k\rangle d \theta(x^1)- \sum_{k \in K} \int_{\partial D_1(y)_k} \wf\langle \langle a^1, \eta^1_k\rangle, \nabla_{x^1} \Lambda_1 \rangle d \theta(x^1)\\
+ \frac{1}{2} \sum_{k \in K} \int_{\partial D_1(y)_k} \div_{\partial x^1}(\wf \Lambda_1 \langle a^1, \eta_k^1\rangle) d\theta(x^1).
\end{multline*}
Noting now that $(\AA^{X_1})^*\Lambda_1 = \AA^Y\Lambda_1 = \Lambda_2^{-1} \AA^Y \Lambda$ by condition (b), applying condition (d), and substituting in the definition of $\wf$, we find that 
\begin{multline} \label{eq:mid-exp}
\int_{D_1(y)} \Lambda_1 \AA^{X_1} \wf dx^1 = \int_{D(y)} (\AA^{Y}\Lambda) f dx + \sum_{k \in K} \int_{\partial D(y)_{\wk}} \Lambda f \langle b^1, \eta^1_k\rangle d\theta(x)\\  
- \frac{1}{2} \sum_{k \in K} \int_{\partial D(y)_{\wk}} \Lambda f \langle \div_{x^1}(a^1), \eta^1_k\rangle d \theta(x)  - \sum_{k \in K} \int_{\partial D(y)_{\wk}} \Lambda f \langle \langle a^1, \eta^1_k\rangle, \nabla_{x^1} \log \Lambda_1 \rangle d \theta(x)  \\
 + \frac{1}{2}\sum_{i, j = 1}^n \rho_{ij} \sum_{k \in K} \int_{\partial D_1(y)_{k}} \div_{\partial x^1}(\wf \Lambda_1 \Psi^{j, 1}_k) \langle \Psi^{i, 1}_k, \eta^1_k\rangle d\theta(x^1).
\end{multline}
Applying the multidimensional Leibniz rule, we find that the last term in (\ref{eq:mid-exp}) is given by 
\begin{multline*}
\frac{1}{2} \sum_{i, j = 1}^n \rho_{ij} \sum_{k \in K} \int_{\partial D(y)_{\wk}} \Big(f \Lambda_2 \div_{\partial x^1}(\Lambda_1 \Psi^{j, 1}_k) + \langle \nabla_{\partial x^1}(f\Lambda_2),\Lambda_1 \Psi^{j, 1}_k\rangle\\ + \Lambda_1 \langle \div_{x^2}(\Lambda_2 f \Psi^{\bullet, 2}_k), \Psi^{j, 1}_k\rangle\Big) \langle \Psi^i_k, \eta_k\rangle d\theta(x) = \frac{1}{2} \sum_{i, j = 1}^n \rho_{ij} \sum_{k \in K} \int_{\partial D(y)_{\wk}} \div_{\partial x}(\Lambda f \Psi^j_k) \langle \Psi^i_k, \eta_k\rangle d\theta(x),
\end{multline*}
where we note that $\langle \Psi^{\bullet, 2}_k, \Psi^{j, 1}_k\rangle = \Psi^{j, 2}_k$ by the chain rule.  In addition, by condition (c), terms 2-4 of (\ref{eq:mid-exp}) are equal to 
\[
\sum_{j = 1}^n \sum_{k \in K} \gamma_j \int_{\partial D(y)_{\wk}} \Lambda f \langle \Psi^j_k, \eta_k\rangle d \theta(x) + \sum_{i, j = 1}^n \rho_{ij} \sum_{k \in K} \int_{\partial D(y)_{\wk}} f\, \partial_{y_j}(\Lambda) \langle \Psi^i_k, \eta_k\rangle d \theta(x).
\]
Finally, because $f \in \cD(\AA^Z)$, we see that on $\partial D(y)_{\wk}$ we have
\[
\sum_{i, j = 1}^n \rho_{ij} \langle \Psi^i_k, \eta_k\rangle \partial_{y_i} f = 0.
\]
Adding this in, substituting, and applying condition (e), we conclude as desired that
\begin{multline*}
\int_{D(y)} \Lambda \AA^X(f) dx = \int_{D(y)} \AA^Y(\Lambda) f dx + \sum_{j = 1}^n \sum_{k} \gamma_j \int_{\partial D(y)_k} \Lambda f \langle \Psi^j_k, \eta_k\rangle d \theta(x)\\
 + \frac{1}{2} \sum_{i, j = 1}^n \rho_{ij} \sum_{k} \int_{\partial D(y)_k}\Big(\div_{\partial x}(\Lambda f \Psi^j_k) + 2 \partial_{y_j}(\Lambda f) \Big) \langle \Psi^i_k, \eta_k\rangle d \theta(x).
\end{multline*}

\section{A core criterion for reflecting diffusions} \label{sec:core-crit}

The goal of this section is to prove Theorem \ref{thm:z-reg-cond} giving conditions under which $\cD(\AA^Z)$ is a core for $\AA^Z$.  The conditions of this criterion will apply to the cases of the Laguerre and Jacobi Warren processes.  Throughout this section, we adopt the notations of Section \ref{sec:ps-ext}.

\subsection{Statement of the criterion}

Define the subdomain
\[
\Ds := D - \{\text{nonsmooth points of $\partial D$}\}.
\]
We consider the following conditions, all of which will hold for the Laguerre and Jacobi Warren processes.  To state them, we denote the drift and diffusion coefficients for $Z$ by
\begin{equation} \label{eq:z-coeff}
\gamma^Z(x, y) = (b(x), \nu(y)) \qquad \text{ and } \qquad \rho^Z(x, y) = (a(x), \rho(y)),
\end{equation}
where the drift coefficient $\nu_j(y_j)$ in coordinate $j$ is denoted by 
\begin{equation} \label{eq:nu-coeff}
\nu_j(y_j) := \gamma_j(y) +  \sum_{i = 1}^n \rho_{ij}(y) \partial_{y_i}[\log \Lambda(y, x)].
\end{equation}
Define single-variable coordinate changes on $D$ by
\[
\chi_i(x_i) := \int_0^{x_i} \frac{du}{\sqrt{a_{ii}(u)}} \qquad \text{ and } \qquad \chi_{m + j}(y_j) := \int_0^{y_j} \frac{du}{\sqrt{\rho_{jj}(u)}},
\]
and define the mapping $\chi: \RR^{m + n} \to \RR^{m + n}$ by $\chi(z_1, \ldots, z_{m + n}) = (\chi_1(z_1), \ldots, \chi_{m + n}(z_{m + n}))$.  Define also the mappings $\chi_X(x_1, \ldots, x_m) := (\chi_1(x_1), \ldots, \chi_m(x_m))$ and $\chi_Y(y_1, \ldots, y_n) := (\chi_{m+1}(y_1), \ldots, \chi_{m + n}(y_n))$ and the transformed spaces
\[
\wD := \chi(D) \qquad \text{ and } \qquad \wDs := \chi(\Ds).
\]
Finally, define the transformed drift coefficient by
\begin{equation} \label{eq:trans-drift}
\wgamma_i^{\wZ}(z_i) = \gamma_i^Z(\chi_i^{-1}(z_i)) \chi_i'(\chi_i^{-1}(z_i)) + \frac{1}{2} \rho^Z_{ii}(\chi_i^{-1}(z_i)) \chi_i''(\chi_i^{-1}(z_i)).
\end{equation}
It will appear again later in Section \ref{sec:path-der}.

\begin{assump} \label{ass:d-cone}
Either (a) both $D$ and $\wD$ are compact convex polyhedra or (b) both $D$ and $\wD$ are polyhedral cones with vertex at $0$.
\end{assump}

\begin{assump} \label{ass:diag}
We consider the following constraints on $Z$:
\begin{itemize}
\item[(a)] the boundary conditions $U_2$ on $Y$ are trivial;

\item[(b)] the diffusion matrices $a(x)$ and $\rho(y)$ are diagonal, and $a_{ii}(x)$ and $\rho_{jj}(y)$ depend only on $x_i$ and $y_j$, respectively;

\item[(c)] the drift coefficient $b_i(x_i)$ in coordinate $x_i$ is a function of $x_i$ only;

\item[(d)] the drift coefficient $\nu_j(y_j)$ in coordinate $y_j$ defined in (\ref{eq:nu-coeff}) is a function of $y_j$ only;

\item[(e)] on each $\partial D(y)_k$, $\langle \Psi^i_k, \eta_k\rangle$ is non-zero for a unique index $i$;

\item[(f)] on the interior of each face of $\XX$, the reflection cone $U_1(x)$ is in a single coordinate direction;

\item[(g)] the process $Z(t)$ stays in $\Ds$ almost surely.
\end{itemize}
\end{assump}

\begin{remark}
Notice that Assumptions \ref{ass:diag}(b) and (e) together with the definition of the reflection directions (\ref{eq:ps-ref-bound}) imply that the directions of reflection on $\partial D(y)_k$ are constant and in a single coordinate direction along each face of $\partial D(y)_k$.
\end{remark}

\begin{assump} \label{ass:ref-geo}
We consider the following assumptions on the geometries of $D$ and $\wD$ and the reflections:
\begin{itemize}
\item[(a)] for $F$ the minimal dimension face containing $z \in \partial D$, the span of $\wU(z)$ intersects $T_zF$ only at $0$;

\item[(b)] the dimension of $\wU(z)$ is constant on the interior of each face of $\partial D$;

\item[(c)] the diffusion coefficients $\rho_{ij}$ lie in $C^2(\YY)$.
\end{itemize}
For $\wD$, we consider the same assumptions with $\wU(z)$ replaced by $\wU(\chi^{-1}(z))$.
\end{assump}

\begin{assump} \label{ass:estimate}
We assume the following estimates on the coefficients of $\AA^Z$:
\begin{itemize}
\item[(a)] The transformed drift coefficients $\wgamma_i^{\wZ}$ are locally Lipschitz on $\wDs$.

\item[(b)] The drift coefficients $\gamma^Z(z)$ grow sublinearly so that
\[
|\gamma^Z(z)| < \wtilde{\gamma} (1 + |z|)
\]
for some $\wtilde{\gamma} > 0$.

\item[(c)] The diffusion coefficients $\rho^Z(z)$ grow subquadratically so that
\[
|\rho_{ii}^Z(z)| < \wtilde{\rho} (1 + |z|^2)
\]
for some $\wtilde{\rho} > 0$.

\item[(d)] The transformed drift coefficient $\wgamma_i^{\wZ}$ is $C^2$ on the interior of $\wD$.
\end{itemize}
\end{assump}

\begin{assump} \label{ass:drift-bound}
If Assumption \ref{ass:diag} holds, let $\sigma_i$ be the first hitting time of a face of $D$ on which the reflection is along coordinate $i$.  For some $\lambda > 0$ and any $h \in C^2_c(\wD)$ which is locally constant on a neighborhood of the vertices of $\wD$, we assume that
\begin{equation} \label{eq:drift-exp}
\EE\left[\int_0^{\sigma_i} e^{-\lambda t}\, \partial_i h(\chi(Z(t))\, \exp\Big(\int_0^t \frac{\partial \wgamma_i^{\wZ}}{\partial z_i}(\chi_i(Z(r)_i)) dr\Big) dt \mid Z(0) = z \right]
\end{equation}
is bounded as a function of $z$.
\end{assump}

We are ready to state Theorem \ref{thm:z-reg-cond} giving a core for $\AA^Z$ under Assumptions \ref{ass:d-cone}, \ref{ass:diag}, \ref{ass:ref-geo}, \ref{ass:estimate}, and \ref{ass:drift-bound}.

\begin{theorem} \label{thm:z-reg-cond}
Suppose Assumptions \ref{ass:domain}, \ref{ass:d-cone}, \ref{ass:diag}, \ref{ass:ref-geo}, \ref{ass:estimate}, and \ref{ass:drift-bound} hold. Then $\cD(\AA^Z)$ is a core for $\dom(\AA^Z)$.
\end{theorem}

Consider the enlarged space of test functions
\[
\cD^{2, s}_0(\AA^Z) := \{f \in C_0(D) \cap C^2(\Ds) \mid f \circ \chi^{-1} \in C^1_b(\wDs), \langle \nabla f(z), u \rangle = 0 \text{ for $z \in \partial D$ and $u \in \wtilde{U}(z)$}\}.
\]
Theorem \ref{thm:z-reg-cond} follows from the following lemmas, whose proofs will be the aim of the subsequent subsections.

\begin{lemma} \label{lem:dens-dom}
Under the conditions of Theorem \ref{thm:z-reg-cond}, for any $f \in \cD^{2, s}_0(\AA^Z)$, there exists $f_l \in \cD(\AA^Z)$ so that $f_l \to f$ and $\AA^Z f_l \to \AA^Z f$ in the uniform norm.
\end{lemma}

\begin{lemma} \label{lem:dom-core}
Under the conditions of Theorem \ref{thm:z-reg-cond}, $\cD^{2, s}_0(\AA^Z)$ is a core for $\dom(\AA^Z)$.
\end{lemma}

\begin{proof}[Proof of Theorem \ref{thm:z-reg-cond}]
Since $\AA^Z$ is a closed operator on $\dom(\AA^Z)$, to check that $\cD(\AA^Z)$ is a core it suffices to check that the closure of $\AA^Z$ on $\cD(\AA^Z)$ is defined on $\dom(\AA^Z)$.  This closure is defined on $\cD^{2, s}_0(\AA^Z)$ by Lemma \ref{lem:dens-dom} and therefore on all of $\dom(\AA^Z)$ because $\cD^{2, s}_0(\AA^Z)$ is a core by Lemma \ref{lem:dom-core}.
\end{proof}

\subsection{Proof of Lemma \ref{lem:dens-dom}}

Our proof will rely on the following construction of a sequence of bump functions on $D$ with good properties.

\begin{lemma} \label{lem:phi-exist}
If $D$ satisfies Assumption \ref{ass:d-cone}(b), for any compact $K$ and open $V$ with $0 \in K \subset V \subset D$ so that $\overline{V}$ is compact, there exists a function $\phi \in C^2_c(D)$ so that
\begin{itemize}
\item[1.] $\langle \nabla \phi(z), u \rangle = 0$ for $z \in \partial D$, $u \in \wtilde{U}(z)$;

\item[2.] $\phi \equiv 1$ on $K$ and $\phi \equiv 0$ on $V^c$.
\end{itemize}
\end{lemma}
\begin{proof}
For $0 \leq d \leq m + n$, denote by $\{F_{d, i}\}$ the collection of dimension $d$ faces of $D$.  We induct upwards on $d$ to show that there exists a collection of function $\{\phi_{d, i}, \phi^\alpha_{d, i}, \phi^{\alpha, \beta}_{d, i}\}$ for $\partial_\alpha, \partial_\beta \in T(D)$ with $\phi_{d, i} \in C^2_c(F_{d, i})$, $\phi^\alpha_{d, i} \in C^1_c(F_{d, i})$, and $\phi^{\alpha, \beta}_{d, i} \in C_c(F_{d, i})$ so that
\begin{itemize}
\item[1.] if $F_{d, i}$ is contained in $F_{d', j}$, then $\phi_{d, i} = \phi_{d', j}|_{F_{d, i}}$, $\phi_{d, i}^\alpha = \phi_{d', j}^\alpha|_{F_{d, i}}$, and $\phi_{d, i}^{\alpha, \beta} = \phi_{d', j}^{\alpha, \beta}|_{F_{d, i}}$;

\item[2.] $\phi_{d, i}^\alpha$ is linear in $\alpha$, $\phi_{d, i}^{\alpha, \beta}$ is bilinear in $\alpha, \beta$, and we have $\phi_{d, i}^{\alpha, \beta} = \phi_{d, i}^{\beta, \alpha}$;

\item[3.] for $\partial_\alpha \in T(F_{d, i})$, we have $\phi_{d, i}^\alpha = \partial_\alpha f_{d, i}$ and $\phi_{d, i}^{\alpha, \beta} = \partial_\alpha f_{d, i}^\beta$;

\item[4.] we have $\phi_{d, i}^u(z) = 0$ for $z \in F_{d, i}$ and $u \in \wU(z)$;

\item[5.] $\phi_{d, i} \equiv 1$ on $K \cap F_{d, i}$, $\phi_{d, i} \equiv 0$ on $V^c \cap F_{d, i}$, $\phi^\alpha_{d, i} \equiv 0$ on $(K \cup V^c) \cap F_{d, i}$, and $\phi^{\alpha, \beta}_{d, i} \equiv 0$ on $(K \cup V^c) \cap F_{d, i}$.
\end{itemize}
This is sufficient, since taking $\phi := \phi_{m + n, 1}$ will give the desired function by Properties 4 and 5.

For the base case $d = 0$, there is a unique vertex $0 = F_{0, 1}$ by Assumption \ref{ass:d-cone}, and we may set $\phi_{0, 1}(0) = 1$ and $\phi_{0, 1}^\alpha(0) = \phi_{0, 1}^{\alpha, \beta}(0) = 0$.  For the inductive step, suppose we have constructed functions satisfying the conditions for all faces of dimension less than $d$.  Fix a $d$-dimensional face $F_{d, i}$ whose boundary is the union of dimension $d - 1$ faces $\{F_{d - 1, j}\}_{j \in J}$.

First, we specify $C^2$ Whitney data for $h_{d, i}$ on $(V - K) \cap F_{d, i}$ by
\begin{itemize}
\item the boundary values $h_{d, i} = \phi_{d - 1, j}$ on $(V - K) \cap F_{d - 1, j}$, $h_{d, i} \equiv 1$ on $\partial K \cap F_{d, i}$, and $h_{d, i} \equiv 0$ on $\partial V \cap F_{d, i}$;

\item the boundary derivatives $\partial_\alpha h_{d, i} = \phi^\alpha_{d - 1, j}$ on $F_{d - 1, j}$ and $\partial_\alpha h_{d, i} = 0$ on $\partial K \cap F_{d, i}$ and $\partial V \cap F_{d, i}$ for $\partial_\alpha \in T(F_{d, i})$;

\item the boundary second derivatives $\partial_{\alpha} \partial_{\beta} h_{d, i} = \phi^{\alpha, \beta}_{d - 1, j}$ on $F_{d - 1, j}$ and $\partial_\alpha\partial_\beta h_{d, i} = 0$ on $\partial K \cap F_{d, i}$ and $\partial V \cap F_{d, i}$ for $\partial_\alpha, \partial_\beta \in T(F_{d, i})$.
\end{itemize}
This data is well defined on the boundary of $(V - K) \cap F_{d, i}$ by Properties 1 and 5 and is valid $C^2$ Whitney data by Property 3.  By the Whitney extension theorem, there is a function $h_{d, i} \in C^2((V - K) \cap F_{d, i})$ agreeing with the boundary conditions.  Since $\overline{V}$ is compact, by the boundary conditions we may extend $h_{d, i}$ to a function $\phi_{d, i}$ on $F_{d, i}$ satisfying Properties 1 and 5.  Now, for $\partial_\alpha, \partial_\beta \in T(F_{d, i})$, we define $\phi^{\alpha}_{d, i} := \partial_\alpha \phi_{d, i}$ and $\phi^{\alpha, \beta}_{d, i} := \partial_{\alpha} \partial_\beta T(F_{d, i})$.

Second, by Assumptions \ref{ass:ref-geo}(ab), for $z$ in the interior of $F_{d, i}$, $\wU(z)$ has constant dimension and intersects $T_z(F_{d, i})$ only at $0$.  This implies that for a basis $\{\partial_{\beta_k}\}$ of $T(F_{d, i})^\perp$ in $T(D)$, the closure of
\[
M_{d, i} := \{(z, \phi^{\beta_1}_{d, i}(z), \ldots, \phi^{\beta_{m + n - d}}_{d, i}(z)) \mid \phi^u_{d, i}(z) = 0 \text{ for $u \in \wU(z)$}\}
\]
is the kernel of a surjective affine bundle map between trivial vector bundles over $F_{d, i}$.  This map is $C^2$ by Assumption \ref{ass:ref-geo}(c), hence $M_{d, i}$ forms a $C^2$ affine bundle over $F_{d, i}$ with non-empty fibers by Assumption \ref{ass:ref-geo}(a).  We give $C^1$ Whitney data for sections $h^\nabla_{d, i} : (V - K) \cap F_{d, i} \to M_{d, i}$ by specifying for any $\partial_\beta \in T(F_{d, i})^\perp$ conditions on $h^\beta_{d, i} := \langle \beta, h^\nabla_{d, i}\rangle$ given by
\begin{itemize}
\item the boundary values $h^{\beta}_{d, i} = \phi^\beta_{d - 1, j}$ on $F_{d - 1, j}$;

\item the boundary derivatives $\partial_\alpha h^\beta_{d, i} = \phi^{\alpha, \beta}_{d - 1, j}$ on $F_{d - 1, j}$ for $\partial_\alpha \in T(F_{d, i})$;

\item identically zero boundary conditions on $\partial V \cap F_{d, i}$ and $\partial K \cap F_{d, i}$.
\end{itemize}
The data is well-defined and lies in $M_{d, i}$ by Properties 1 and 4, and it is valid $C^1$ Whitney data by Property 3.  Passing to a partition of unity on which $M_{d, i}$ is locally trivial and applying Whitney extension, we find a compactly supported section $h^\nabla_{d, i}: (V - K) \cap F_{d, i} \to M_{d, i}$ with the desired boundary conditions.  Its extension to $F_{d, i}$ by $0$ yields $\phi^\beta_{d, i} \in C^1_c(F_{d, i})$ satisfying the desired properties.  We define also $\phi^{\alpha, \beta}_{d, i} := \partial_\alpha \phi^\beta_{d, i}$ for $\partial_\alpha \in T(F_{d, i})$.

Third, for a basis $\{\partial_{\beta_k}\}$ of $T(F_{d, i})^\perp$, we define $C^0$ Whitney data for $h^{\beta_k, \beta_{k'}}_{d, i}$ on $(V - K) \cap F_{d, i}$ by
\begin{itemize}
\item the boundary values $h^{\beta_k, \beta_{k'}}_{d, i} = \phi^{\beta_k, \beta_{k'}}_{d - 1, j}$ on $F_{d - 1, j}$;

\item identically zero boundary conditions on $\partial V \cap F_{d, i}$ and $\partial K \cap F_{d, i}$.
\end{itemize}
The data is valid $C^0$ Whitney data by Properties 1 and 3, so Whitney extension yields a $h^{\beta_k, \beta_{k'}}_{d, i} \in C((V - K) \cap F_{d, i})$ which may be extended by $0$ to $\phi^{\beta_k, \beta_{k'}}_{d, i} \in C_c(F_{d, i})$.

To complete the induction, it remains only to verify Properties 1-5 on $F_{d, i}$ for the functions we just constructed.  Indeed, Properties 1-3 and 5 hold by our choice of boundary values, and Property 4 holds by our definition of $M_{d, i}$.  This completes the proof.
\end{proof}

We now proceed to the proof of Lemma \ref{lem:dens-dom}.  If $D$ is compact in Assumption \ref{ass:d-cone}, then $\cD^{2, s}_0(\AA^Z) = \cD(\AA^Z)$, so there is nothing to show.  Otherwise, we construct such an approximation explicitly.  Define the compact set $K \subset D$ by $K = \{z \in D \mid |z| \leq 1\}$, and let $V$ denote the interior of $2 K$.  Let $\phi \in C^2_c(D)$ be the function provided by applying Lemma \ref{lem:phi-exist} to $K \subset V$, and define $\phi_l(z) := \phi(z/l)$ and
\[
f_l(z) := f(z) \cdot \phi_l(z).
\]
Notice that $f_l \in C^2_c(D)$ and that for $z \in \partial D$ and $u \in \wU(z)$, we have that $z/l \in \partial D$ by Assumption \ref{ass:d-cone}(b), so an application of Condition 1 from Lemma \ref{lem:phi-exist} implies that
\[
\langle u, \nabla f_l(z)\rangle = \langle u, \nabla f(z) \rangle \phi_l(z) + \langle u, \nabla \phi(z/l)\rangle f(z)/l = 0.
\]
We conclude that $f_l \in \cD(\AA^Z)$, so it suffices to check that $f_l \to f$ and $\AA^Z f_l \to \AA^Z f$ uniformly.  

To check that $f_l \to f$, choose $C > 0$ so that $|\phi| < C - 1$. For any $\eps > 0$, since $f \in C_0(D)$, we may choose $l$ large enough so that $||f|| < \eps/C$ on $(lK)^c$.  Observe now that $f_l - f$ vanishes on $lK$ and that
\[
|f_l(z) - f(z)| = |f(z)| \cdot |\phi_l(z) - 1| \leq |f(z)| C < \eps
\]
on $(lK)^c$, meaning that $||f_l - f|| < \eps$ for this $l$.  We conclude that $f_l \to f$ uniformly.

To check that $\AA^Z f_l \to \AA^Zf$, choose $C_1 > 0$ so that $|\phi| < C_1 - 1$, $C_2 > 0$ so that $|\partial_i \phi| < C_2$, and $C_3 > 0$ so that $|\partial_i^2 \phi| < C_3$.  Observe now that 
\[
\AA^Z f_l(z) - \AA^Z f(z) = \AA^Z f(z) (\phi_l(z) - 1) + E_l(z),
\]
where
\[
E_l(z) = \frac{f(z)}{l} \langle \gamma^Z(z), \nabla \phi(z/l)\rangle + \sum_{i = 1}^{m + n} \rho_{ii}^Z(z) \Big(\frac{2}{l} \partial_i f(z) \partial_i \phi(z/l) + \frac{1}{l^2} \partial_i^2 \phi(z/l) f(z)\Big)
\]
and we recall that the coefficients $\gamma^Z$ and $\rho^Z$ were defined in (\ref{eq:z-coeff}).  Observe that $E_l(z)$ vanishes on $lK$ and $(lV)^c$ and that
\[
|E_l(z)| \leq C_2\, |f(z)| \cdot \frac{|\gamma^Z(z)|}{l} + \frac{2 C_2}{l} \left| \sum_{i = 1}^{m +n} \rho_{ii}^Z(z) \partial_if(z)\right| + \sum_{i = 1}^{m + n} C_3 |f(z)| \cdot \frac{|\rho_{ii}^Z(z)|}{l^2}.
\]
For any $\eps > 0$, choose $l$ large enough so that
\begin{itemize}
\item[1.] $||f|| < \min\{\frac{\eps}{16 C_2 \wtilde{\gamma}}, \frac{\eps}{16 C_3 (m + n) \wtilde{\rho}}\}$ on $(lK)^c$;
\item[2.] $||\AA^Z f || < \frac{\eps}{4C_1}$ on $(lK)^c$;
\item[3.] $\frac{1}{l} \left|\sum_{i = 1}^{m + n} \rho^Z_{ii}(z) \partial_i f(z)\right| < \frac{\eps}{8 C_2}$ on $(lK)^c$,
\end{itemize}
where $\wtilde{\gamma}$ and $\wtilde{\rho}$ are given in Assumptions \ref{ass:estimate}(b) and (c) and the last condition is possible because $f \circ \chi^{-1} \in C^1_b(\wDs)$.  Now, we have $\phi_l(z) \equiv 1$ on $lK$, and the support of $E_l(z)$ is contained in $l(V - K)$.  Therefore, for this choice of $l$, we have on $(lK)^c$ that
\begin{align*}
|\AA^Z f(z)(\phi_l(z) - 1)| &< \eps / 4\\
C_2 |f(z)| \frac{|\gamma^Z(z)|}{l} &< \eps / 4\\
\frac{2 C_2}{l} \left| \sum_{i = 1}^{m +n} \rho_{ii}^Z(z) \partial_if(z)\right| &< \eps/4\\
C_3 |f(z)| \cdot \frac{|\rho_{ii}^Z(z)|}{l^2} &< \eps/4,
\end{align*}
where we apply Assumptions \ref{ass:estimate}(b) and (c) in the second and forth estimates.  Putting these together, we find that $|\AA^Zf_l(z) - \AA^Zf(z)| < \eps$ for this $l$.  We conclude that $\AA^Zf_l \to \AA^Zf$ uniformly, completing the proof.

\subsection{Coordinate transforms and pathwise derivatives} \label{sec:path-der}

In this section we introduce some results on pathwise derivatives which we will later use to prove Lemma \ref{lem:dom-core}.  Define the coordinate change $\wZ(t)$ of $Z(t)$ by
\[
\wZ(t) := \chi(Z(t))
\]
and its coordinate projections $\wZ_1(t) := \chi_X(Z_1(t))$ and $\wZ_2(t) := \chi_Y(Z_2(t))$.  Under Assumption \ref{ass:diag}, by It\^o's lemma, $\wZ(t)$ solves the SDER
\begin{align*}
d\wZ_1(t) &= dB_X(t) + \wb(\wZ_1(t)) dt + d\wPhi_1(t)\\
d\wZ_2(t) &= dB_Y(t) + \wnu(\wZ_2(t)) dt + d\wPhi(t),
\end{align*}
where the drift coefficients are given by
\begin{align*}
\wb_i(\wZ_1(t)) &:= b_i(\chi_i^{-1}(\wZ_{1, i}(t))) \chi'_i(\chi_i^{-1}(\wZ_{1, i}(t))) + \frac{1}{2} a_{ii}(\chi_i^{-1}(\wZ_{1, i}(t))) \chi''_i(\chi_i^{-1}(\wZ_{1, i}(t))) \\
\wnu_j(\wZ_2(t)) &:= \nu_j(\chi_{m + j}^{-1}(\wZ_{2, j}(t))) \chi'_{m + j}(\chi_{m + j}^{-1}(\wZ_{2, j}(t))) + \frac{1}{2} \rho_{jj}(\chi_{m + j}^{-1}(\wZ_{2, j}(t))) \chi''_{m+j}(\chi_{m+j}^{-1}(\wZ_{2, j}(t)))
\end{align*}
and $\wPhi_1(t)$ and $\wPhi(t)$ are bounded variation processes with auxiliary functions $\wphi_1(s)$ and $\wphi(s)$ so that
\begin{itemize}
\item $\wPhi_1(0) = 0$, $\int_0^\infty 1_{\wZ_1(t) \notin \chi(\partial \XX)} d|\wPhi_1|(t) = 0$, $\wPhi_1(t) = \int_0^t \wphi_1(s) d|\wPhi_1|(s)$, and $\phi_1(s)$ is in the same direction as $U_1(\chi^{-1}_X(\wZ_1(s)))$;
\item $\wPhi(0) = 0$, $\int_0^\infty 1_{\wZ(t) \notin \chi(\partial D)} d|\Phi|(t) = 0$, $\wPhi(t) = \int_0^t \wphi(s) d|\wPhi|(s)$, and $\phi(s) \in U(\chi^{-1}(\wZ(s)))$.
\end{itemize}
Notice that we omit reflection for $\wZ_2(t)$ by Assumption \ref{ass:diag}(a) and that $U_1(\chi^{-1}_X(\wZ_1(s)))$ and $U(\chi^{-1}(\wZ(s)))$ are the correct cones of reflection for the transformed process by Assumption \ref{ass:diag}(e) and (f).  Denote the generator of the corresponding Feller process on $\wD$ by $\AA^{\wZ}$, and note that the combined drift coefficient of (\ref{eq:trans-drift}) is given by
\[
\wgamma^{\wZ}(x, y) := (\wb(x), \wnu(y)).
\]
We wish to consider the pathwise derivative of $\wZ(t)$ with respect to its starting point.  By Assumptions \ref{ass:d-cone} and \ref{ass:diag}, once localized to Lipschitz drift coefficients $\wb$ and $\wnu$, the process $\wZ(t)$ satisfies the assumptions of \cite[Theorem 2.2]{And09}.  Let $K_1 \subset K_2 \subset \cdots \subset \wDs$ be a sequence of convex polyhedral domains exhausting $\wDs$, and let $\wZ^l_z(t)$ denote $\wZ(t)$ started at $z$ and stopped at $\tau_l := \inf\{t \mid \wZ(t) \notin K_l\}$.  Let $\wZ^\infty_z(t)$ denote $\wZ(t)$ started at $z$ and stopped at $\tau_\infty := \inf\{t \mid \wZ(t) \notin \wDs\}$.  By Assumption \ref{ass:estimate}(a), we see that $\wgamma^{\wZ}$ is Lipschitz on $K_l$.  Further, the domain $\wD$ is a convex polyhedral cone, and $\wZ(t)$ avoids $\wDs$, so we may apply the framework of \cite{And09}, which we elaborate on presently.

Let $\{\wF_a\}$ denote the set of faces of $\partial \wD$, and let $n_a$ and $v_a$ denote the unit normal vector and the constant reflection vector to $\wF_a$ provided by Assumptions \ref{ass:diag}(e) and (f).  Normalize $n_a$ and $v_a$ so that $\langle n_a, v_a \rangle = 1$, and define vectors $v_a^\perp, n_a^\perp \in \text{span}\{v_a, n_a\}$ so that $\langle v_a^\perp, v_a\rangle = \langle n_a^\perp, n_a\rangle = 0$, $\langle n_a^\perp, v_a\rangle = \langle v_a^\perp, n_a\rangle = 1$, and $\langle n_a, v_a^\perp\rangle > 0$.  Let $\{n_a^k\}_{3 \leq k \leq m + n}$ be a completion of $\{n_a^\perp, n_a\}$ to an orthonormal basis of $\RR^{m + n}$.  For each $t \in [0, \tau_l)$, define
\[
r_a(t) := \sup\{r \leq t \mid \wZ^l_z(r) \in \wF_a\}
\]
and $r(t) = \max_a\{r_a(t)\}$, and 
\[
s(t) = \begin{cases} \emptyset & \wZ^l_z(r) \notin \partial \wD \text{ for $r \leq t$}\\
a & r_a(t) = r(t).
\end{cases}
\]
Finally, define
\[
\mu(t) := \sup\{r < t \mid \wZ^l_z(r) \in \partial\wD\}.
\]
With these definitions, we have the following results of Andres on pathwise differentiability.

\begin{theorem}[{\cite[Theorem 2.2]{And09}}] \label{thm:and-der}
The map $z \mapsto \wZ_z^l(t)$ is a.s. differentiable for all $t \in [0, \tau_l)$ for which $\wZ_z^l(t) \notin \partial D$.  Further, the process $\eta^{ij}_l(t, z) := \frac{\partial \wZ_z^l(t)_i}{\partial z_j}$ admits a right continuous extension on $[0, \tau_l)$ which a.s. solves
\begin{align*}
\eta^{ij}_l(t, z) &= \delta_{ij} + \int_0^t \sum_{k = 1}^{m + n} \frac{\partial \wgamma_i}{\partial z_k}(\wZ_z^l(r))\eta^{kj}_l(r, z) dr \text{ if $s(t) = \emptyset$}\\
\eta^{ij}_l(t, z) &= \langle \eta^{\bullet j}_l(\mu(t)^-, z), v_a^\perp\rangle (n_a^\perp)_i + \sum_{k = 3}^{m + n} \langle \eta^{\bullet j}_l(\mu(t)^-, z), n^k_a\rangle (n^k_a)_i + \int_{r_a(t)}^t \sum_{k = 1}^{m + n} \frac{\partial \wgamma_i}{\partial z_k}(\wZ_z^l(r)) \eta^{kj}_l(r, z) dr \text{ if $s(t) = a$},
\end{align*}
where $\eta^{ij}_l(\mu(t)^-, z) := \lim_{r \to \mu(t)^-} \eta^{ij}_l(r, z)$.
\end{theorem}

\begin{prop}[{\cite[Corollary 2.4]{And09}}] \label{prop:and-exp}
If $h \in C^1_b(\wDs)$, we have 
\[
\frac{\partial \EE[h(\wZ_z^l(t))]}{\partial z_i} = \sum_{k = 1}^{m + n} \EE[\partial_k h(\wZ_z^l(t)) \eta^{ki}_l(t, z)].
\]
\end{prop}

\begin{prop}[{\cite[Corollary 3.3]{And09}}] \label{prop:and-trans}
For any $h \in C_b(\wDs)$, we have 
\[
\langle v_a, \nabla_z \EE[h(\wZ_z^l(t))]\rangle = 0 \qquad \text{ for $z \in \wF_a$}.
\]
\end{prop}

\begin{remark}
By patching the processes $\wZ^l_z(t)$, we obtain that Theorem \ref{thm:and-der} and Proposition \ref{prop:and-exp} hold for $l = \infty$ for all $t \in [0, \tau_\infty)$.  We denote the resulting patched derivative process by $\eta^{ij}_\infty(t, z)$.
\end{remark}

In what follows, we use the following implications of these results.

\begin{lemma} \label{lem:der-coord}
If Assumptions \ref{ass:diag}(b), (c), (d), (e), and (f) hold, then we have the following:
\begin{itemize}
\item[(a)] $\eta^{ij}_l(t, z) = 0$ unless $i = j$;

\item[(b)] $\eta^{ii}_l(t, z) = 0$ if $t \geq \sigma_i$, where $\sigma_i$ is the first time the reflection on $\partial \wD$ is in coordinate $i$.
\end{itemize}
\end{lemma}
\begin{proof}
For (a), by Assumptions \ref{ass:diag}(b), (c), and (d), we see that $\wgamma_i(z)$ depends only on $z_i$; in particular, $\frac{\partial \wgamma_i}{\partial z_k} = 0$ for $k \neq i$, so the conclusion follows from Theorem \ref{thm:and-der}.

For (b), by Assumptions \ref{ass:diag}(e) and (f), the reflection on $\partial \wD$ is in a single coordinate direction.  Therefore, if $s(t) = a$ and the reflection off $F_a$ in the direction of $z_i$, then $v_a$ is given by $1_i$, the unit basis vector in the $z_i$-coordinate.  Substituting this into the result of Theorem \ref{thm:and-der} yields the result.
\end{proof}

\begin{corr} \label{corr:grad-bound}
For some $\lambda > 0$, if the quantity
\[
\int_0^{\tau_l} e^{-\lambda t} \nabla_z\EE[h(\wZ^l_z(t))] dt
\]
is uniformly bounded in $l$ as a function of $z$, then the function
\[
\int_0^\infty e^{-\lambda t}\, \nabla_z \EE[h(\wZ_z^\infty(t))] dt
\]
is continuous and bounded on $\wDs$ and satisfies
\[
\left\langle v_a, \int_0^\infty e^{-\lambda t}\, \nabla_z \EE[h(\wZ_z^\infty(t))] dt \right \rangle = 0
\]
for $z \in \wF_a$. 
\end{corr}
\begin{proof}
First, by the dominated convergence theorem, we conclude that
\[
\int_0^\infty e^{-\lambda t}\, \nabla_{z} \EE[h(\wZ_z^\infty(t))] dt
\]
is bounded and continuous on $\wDs$.  Now, for $z \in F_a$ notice that
\begin{align*}
\left\langle v_a, \int_0^\infty e^{-\lambda t}\, \nabla_z \EE[h(\wZ_z^\infty(t))] dt\right\rangle
&= \sum_{k, i = 1}^{m + n}  v_{a, i} \EE\left[\int_0^\infty e^{-\lambda t} \partial_k h(\wZ_z^\infty(t)) \eta^{ki}_\infty(t, z) dt\right]\\
&= \int_0^\infty e^{-\lambda t} \EE\left[ \sum_{k, i = 1}^{m + n} \partial_k h(\wZ_z^\infty(t)) v_{a, i} \eta^{ki}_\infty(t, z) \right] dt,
\end{align*}
so it suffices to check that $\sum_{i = 1}^{m + n} v_{a, i} \eta^{ki}_\infty(t, z) = 0$, which again follows by patching from the fact that $\sum_{i = 1}^{m + n} v_{a, i} \eta^{ki}_l(t, z) = 0$ from the proof of Proposition \ref{prop:and-trans} in \cite[Corollary 3.3]{And09}.
\end{proof}

\begin{lemma} \label{lem:grad-check2}
If Assumption \ref{ass:drift-bound} holds for some $\lambda > 0$, then for any $h \in C^2_c(\wD)$ which is locally constant near the vertices of $\wD$, the quantity
\[
\int_0^{\tau_l} e^{-\lambda t} \nabla_z \EE[h(\wZ^l_z(t))] dt
\]
is uniformly bounded in $l$ as a function of $z$.
\end{lemma}
\begin{proof}
By Theorem \ref{thm:and-der} and Lemma \ref{lem:der-coord}, we see that $\eta^{ki}_l(t, z) = 0$ if $k \neq i$ and
\[
\eta^{ii}_l(t, z) = \exp\Big(\int_0^t \frac{\partial\wgamma_i^{\wZ}}{\partial z_i}(\wZ^l_z(r)_i) dr\Big) 1_{t < \sigma_i},
\]
where we recall that $\sigma_i$ is the first time that the reflection on $\partial \wD$ is in the $z_i$ coordinate direction.  In these terms, by Theorem \ref{thm:and-der}, Proposition \ref{prop:and-exp}, and Lemma \ref{lem:der-coord}, we have
\begin{align*}
\int_0^{\tau_l} e^{-\lambda t} \partial_{z_i} \EE[h(\wZ^l_z(t))] dt &= \sum_{k = 1}^{m + n} \int_0^{\tau_l} e^{-\lambda t} \EE[\partial_k h(\wZ^l_z(t)) \eta^{ki}_l(t, z)]\\
&= \EE\left[ \int_0^{\tau_l \wedge \sigma_i} e^{-\lambda t} \partial_i h(\wZ^l_z(t)) \exp\Big(\int_0^t \frac{\partial \wgamma_i^{\wZ}}{\partial z_i}(\wZ^l_z(r)_i)) dr \Big) dt \right].
\end{align*}
Now, write $h = h_1 - h_2$ with $h_1, h_2 \in C^2_c(\wD)$ locally constant at the vertices of $\wD$ and $h_1, h_2$ non-negative.  Noting the bound
\begin{multline*}
\int_0^{\tau_l} e^{-\lambda t} \partial_{z_i} \EE[h(\wZ^l_z(t))] dt \leq \EE\left[ \int_0^{\sigma_i} e^{-\lambda t} |\partial_i h_1(\wZ^\infty_z(t))| \exp\Big(\int_0^t \frac{\partial \wgamma_i^{\wZ}}{\partial z_i}(\wZ^\infty_z(r)_i)) dr \Big) dt \right]\\ +\EE\left[ \int_0^{\sigma_i} e^{-\lambda t} |\partial_i h_2(\wZ^\infty_z(t))| \exp\Big(\int_0^t \frac{\partial \wgamma_i^{\wZ}}{\partial z_i}(\wZ^\infty_z(r)_i)) dr \Big) dt \right]
\end{multline*}
and applying Assumption \ref{ass:drift-bound} to $h_1$ and $h_2$ separately, we conclude that 
\[
\int_0^{\tau_l} e^{-\lambda t} \partial_{z_i} \EE[h(\wZ^l_z(t))] dt
\]
is uniformly bounded in $l$ as a function of $z$, as needed.
\end{proof}

\subsection{A density property for $C_0(\wD)$}

We now identify a particular dense subset of $C_0(\wD)$.  Define
\[
\cD_\lc(\AA^{\wZ}) := \{\wf \in C^2_c(\wD) \mid \langle \nabla \wf, u\rangle = 0 \text{ for $u \in \wU(\chi^{-1}(z))$ and $z \in \partial \wD$, $f$ is locally constant near vertices of $\wD$}\}.
\]
The goal of this section is to prove the following lemma.

\begin{lemma} \label{lem:dom-dens-trans}
If $\wD$ satisfies Assumptions \ref{ass:d-cone} and \ref{ass:ref-geo}, then $\cD_\lc(\AA^{\wZ})$ is dense in $C_0(\wD)$ in the uniform norm.
\end{lemma}
\begin{proof}
Because $C^2_c(\wD)$ is dense in $C_0(\wD)$, it suffices to show that $\cD_\lc(\AA^{\wZ})$ is dense in $C^2_c(\wD)$ in the uniform norm.  More precisely, for any $g \in C^2_c(\wD)$ and any $\eps > 0$, it suffices to show that there exists $f \in \cD_\lc(\AA^{\wZ})$ so that $||f - g|| \leq \eps$.  For this, our approach is similar to the proof of Lemma \ref{lem:phi-exist}.

For $0 \leq d \leq m + n$, let $\{\wF_{d, i}\}$ denote the collection of dimension $d$ faces of $\wD$.  Choose $\eps = \eps_{m + n} > \cdots > \eps_1 > \eps_0 = 0$.  For each vertex $\wF_{0, i}$ of $\wD$, truncate the convex polyhedral cone $\wD$ with a face $\wF_i'$ of dimension $m + n - 1$ which intersects only the faces meeting at $\wF_{0, i}$ and defines an open neighborhood $V_i \ni F_{0, i}$.  Since $g$ is continuous, we may choose $\wF_i'$ so that
\begin{itemize}
\item the closures of the $V_i$ do not intersect;

\item we have $|g(z) - g(F_{0, i})| < \eps_1$ for $z \in V_i$.
\end{itemize}
We claim by induction on $d$ that there exists a collection of functions $\{f_{d, i}, f_{d, i}^\alpha, f_{d, i}^{\alpha, \beta}\}$ for $0 \leq d \leq m + n$ and $\partial_\alpha, \partial_\beta \in T(\wF_{d, i})$ with $f_{d, i} \in C^2_c(\wF_{d, i}), f_{d, i}^\alpha \in C^1_c(\wF_{d, i}), f_{d, i}^{\alpha, \beta} \in C_c(\wF_{d, i})$ so that
\begin{itemize}
\item[1.] if $\wF_{d, i}$ is contained in $\wF_{d', j}$, then $f_{d, i} = f_{d', j}|_{\wF_{d, i}}$, $f_{d, i}^\alpha = f_{d', j}^\alpha|_{\wF_{d, i}}$, and $f_{d, i}^{\alpha, \beta} = f_{d', j}^{\alpha, \beta}|_{\wF_{d, i}}$;

\item[2.] $f_{l, i}^\alpha$ is linear in $\alpha$, $f_{l, i}^{\alpha, \beta}$ is bilinear in $\alpha, \beta$, and we have $f_{l, i}^{\alpha, \beta} = f_{l, i}^{\beta, \alpha}$;

\item[3.] for $\partial_\alpha \in T(\wF_{l, i})$, we have $f_{l, i}^\alpha = \partial_\alpha f_{l, i}$ and $f_{l, i}^{\alpha, \beta} = \partial_\alpha f_{l, i}^\beta$;

\item[4.] we have $f_{l, i}^u(z) = 0$ for $z \in \wF_{l, i}$ and $u \in \wU(\chi^{-1}(z))$;

\item[5.] we have $||f_{d, i} - g|| \leq \eps_d$ on $\wF_{d, i}$;

\item[6.] each $f_{d, i}$ is constant on all the neighborhoods $V_j$.
\end{itemize}
Taking $f := f_{m + n, 1}$ will then yield the claim.

For the base case $d = 0$, the faces $\wF_{0, i}$ are vertices, and for any vertex $v = \wF_{0, i}$ we may set $f_{0, i}(v) = g(v)$ and $f_{0, i}^\alpha(v) = f_{0, i}^{\alpha, \alpha'}(v) = 0$.  For the inductive step, suppose we have constructed functions satisfying the conditions for all dimensions less than $d$.  Fix a face $\wF_{d, i}$ of dimension $d$ whose boundary consists of the dimension $d - 1$ faces $\{\wF_{d - 1, j}\}_{j \in J}$.

First, let $U_{d, i}$ be a bounded open subset in $\RR^{m + n}$ containing $\supp(g) \cap \wF_{d, i}$ and the supports of all $f_{d - 1, j}, f_{d - 1, j}^\alpha, f_{d - 1, j}^{\alpha, \beta}$.  We specify $C^2$ Whitney data for $h_{d, i}$ on $U_{d, i} \cap (\wF_{d, i} - \bigcup_k V_k)$ by
\begin{itemize}
\item the boundary values $h_{d, i} = f_{d - 1, j}$ on $\wF_{d - 1, j}$ and $h_{d, i} = g(\wF_{0, k})$ on $\wF_k$;

\item the boundary derivatives $\partial_\alpha h_{d, i} = f^\alpha_{d - 1, j}$ on $\wF_{d - 1, j}$ and $\partial_\alpha h_{d, i} = 0$ on $\wF_k$ for $\partial_\alpha \in T(\wF_{d, i})$;

\item the boundary second derivatives $\partial_\alpha \partial_{\alpha'} h_{d, i} = f^{\alpha, \alpha'}_{d - 1, j}$ on $\wF_{d - 1, j}$ and $\partial_\alpha \partial_{\alpha'} h_{d, i} = 0$ on $\wF'_k$ for $\partial_\alpha, \partial_{\alpha'} \in T(\wF_{d, i})$;

\item identically zero boundary conditions on $\partial U_{d, i} \cap \wF_{d, i}$.
\end{itemize}
This data is well-defined on $\partial \wF_{d, i}$ by Properties 1 and 6 and the definition of $U_{d, i}$; it is valid $C^2$ Whitney data by Properties 3 and 6.  The Whitney extension theorem therefore yields a function $h_{d, i} \in C^2_c(\wF_{d, i} - \bigcup_k V_k)$ agreeing with the boundary conditions; extend it to a function in $C^2_c(\wF_{d, i})$ by setting it to be constant and equal to $g(\wF_{0, k})$ on $V_k$.  Now, let $\phi_r: \RR \to [0, 1]$ be a smooth cutoff function so that $\phi_r(x) = 1$ for $x \leq 0$ and $\phi_r(x) = 0$ for $x \geq r$, let $\wF_{d, i}^r \subset \wF_{d, i}$ be a subset of the interior of $\wF_{d, i}$ with smooth boundary containing all $z \in \wF_{d, i}$ of distance at least $r$ from $\partial \wF_{d, i}$ and no $z \in \wF_{d, i}$ of distance at most $r/2$ from $\partial \wF_{d, i}$, and let $\rho^r_{d, i}(z)$ denote the distance from $z$ to $\wF_{d, i}^r$. Define $f^r_{d, i} \in C^2_c(\wF_{d, i})$ by
\begin{equation} \label{eq:def-f}
f_{d, i}^r(z) := \phi_{r/2}(\rho^r_{d, i}(z)) g(z) + (1 - \phi_{r / 2}(\rho^r_{d, i}(z))) h_{d, i}(z)
\end{equation}
so that $f_{d, i}^r \equiv h_{d, i}$ on $\partial \wF_{d, i}$ and $f_{d, i}^r(z) = g(z)$ for $z$ at distance at least $r$ from $\partial \wF_{d, i}$.  Since $h_{d, i} \in C^2_c(\wF_{d, i})$ and $\eps_d > \eps_{d - 1}$, by the inductive hypothesis we may choose some $\delta > r^* > 0$ so that $||f^{r^*}_{d, i} - g|| < \eps_d$ on $\wF_{d, i}$.  We define $f_{d, i} := f^{r^*}_{d, i}$, $f^\alpha_{d, i} = \partial_\alpha f_{d, i}$, and $f^{\alpha, \alpha'}_{d, i} = \partial_{\alpha}\partial_{\alpha'} f_{d, i}$ for $\partial_\alpha, \partial_{\alpha'} \in T(\wF_{d, i})$.

Second, by Assumption \ref{ass:ref-geo}(ab), for $z$ in the interior of $\wF_{d, i} - \bigcup_k V_k$, $\wU(z)$ has constant dimension and intersects $T_z(\wF_{d, i})$ only at $0$.  Further, for a basis $\{\partial_{\beta_k}\}$ of $T(\wF_{d, i})^\perp$, the closure of the set 
\begin{equation} \label{eq:def-aff-bundle}
\wM_{d, i} := \{(z, f^{\beta_1}_{d, i}(z), \ldots, f^{\beta_{m + n - d}}_{d, i}(z)) \mid f^u_{l, i}(z) = 0 \text{ for $u \in U(z)$}\}
\end{equation}
is the kernel of a surjective affine map between trivial vector bundles over $\wF_{d, i} - \bigcup_k V_k$ which is $C^2$ by Assumption \ref{ass:ref-geo}(c).  Therefore, it forms a $C^2$ affine bundle over $\wF_{d, i} - \bigcup_k V_k$ whose fibers are non-empty by Assumption \ref{ass:ref-geo}(a).  We give $C^1$ Whitney data for sections $h_{d, i}^\nabla: U_{d, i} \cap \wF_{d, i} \to \wM_{d, i}$ on $U_{d, i} \cap (\wF_{d, i} - \bigcup_k V_k)$ by specifying for any $\partial_\beta \in T(\wF_{d, i})^\perp$ conditions on $h^\beta_{d, i} := \langle \beta, h^\nabla_{d, i}\rangle$ given by
\begin{itemize}
\item the boundary values $h_{d, i}^\beta = f_{d - 1, j}^\beta$ on $\wF_{d - 1, j}$ and $h_{d, i}^\beta = 0$ on $\wF_k$;

\item the boundary derivatives $\partial_\alpha h_{d, i}^\beta = f_{d - 1, j}^{\alpha, \beta}$ on $\wF_{d - 1, j}$ and $\partial_\alpha h_{d, i}^\beta = 0$ on $\wF_k$ for $\partial_\alpha \in T(\wF_{d, i})$;

\item identically zero boundary conditions on $\partial U_{d, i} \cap \wF_{d, i}$.
\end{itemize}
By Property 4, the boundary values lie in $\wM_{d, i}$.  The data is well-defined by Properties 1 and 6 and is valid $C^1$ Whitney data by Properties $3$ and 6.  Passing to a partition of unity on which $\wM_{d, i}$ is locally trivial and applying the Whitney extension theorem, there exists a compactly supported $C^1$ section $h^\nabla_{d, i}: \wF_{d, i} - \bigcup_k V_k \to \wM_{d, i}$ satisfying the boundary conditions; extend it by $0$ to all of $\wF_{d, i}$.  It gives rise to $h^\beta_{d, i} := \langle \beta, h^\nabla_{d, i}\rangle \in C^1_c(\wF_{d, i})$ satisfying the boundary conditions.  We define $f^\beta_{d, i} := h^\beta_{d, i}$ and $f^{\alpha, \beta}_{d, i} := \partial_\alpha h^\beta_{d, i}$ for $\partial_\alpha \in T(\wF_{d, i})$.

Third, for a basis $\{\partial_{\beta_k}\}$ of $T(\wF_{d, i})^\perp$ and $k \leq k'$, define the $C^0$ Whitney data for $h^{\beta_k, \beta_{k'}}_{d, i}$ on $U_{d, i} \cap (\wF_{d, i} - \bigcup_k V_k)$ by
\begin{itemize}
\item the boundary values $h_{d, i}^{\beta_k, \beta_{k'}} = f_{d - 1, j}^{\beta_k, \beta_{k'}}$ on $\wF_{d - 1, j}$ and $h_{d, i}^{\beta_k, \beta_{k'}} = 0$ on $\wF_k'$;

\item identically zero boundary conditions on $\partial U_{d, i} \cap \wF_{d, i}$.
\end{itemize}
The data is well-defined by Properties 1 and 6, so Whitney extension gives a $h^{\beta_{k}, \beta_{k'}}_{d, i} \in C_c(\wF_{d, i} - \bigcup_k V_k)$ satisfying the boundary conditions.  We may again extend this by $0$ to all of $\wF_{d, i}$ and define $f^{\beta_{k}, \beta_{k'}}_{d, i} := h^{\beta_{k}, \beta_{k'}}_{d, i}$.

It remains only to verify Properties 1-6 on $F_{d, i}$ for the functions we have just constructed.  Properties 1, 2, 3, and 6 hold by the boundary values we chose in our applications of the Whitney extension theorem.  Property 4 holds because we defined the domain of $f^\beta_{d, i}$ for $\partial_\beta \in T(\wF_{d, i})^\perp$ to be $\wM_{d, i}$ in (\ref{eq:def-aff-bundle}).  Finally, Property 5 holds by the definition of $f_{d, i}$ in (\ref{eq:def-f}), which completes the proof.
\end{proof}

\subsection{Proof of Lemma \ref{lem:dom-core}}

Notice that $f \in \dom(\AA^Z)$ if and only if $f \circ \chi^{-1} \in \dom(\AA^{\wZ})$.  Therefore, it suffices to check that
\[
\cD^{2, s}_0(\AA^{\wZ}) := \{\wf \in C_0(\wD) \cap C^2(\wDs) \mid \wf \in C^1_b(\wDs), \langle \nabla \wf, u\rangle = 0 \text{ for $u \in \wU(\chi^{-1}(z))$}\}
\]
is a core for $\AA^{\wZ}$.  First, we show that it lies in $\dom(\AA^{\wZ})$.  

\begin{lemma} \label{lem:dens-dom1}
Under the conditions of Theorem \ref{thm:z-reg-cond}, we have $\cD^{2, s}_0(\AA^{\wZ}) \subset \dom(\AA^{\wZ})$.
\end{lemma}
\begin{proof}
Since $\nabla \wf$ is bounded for $\wf \in \cD^{2, s}_0(\AA^{\wZ})$, the claim follows by It\^o's lemma.
\end{proof}

To establish that $\cD^{2, s}_0(\AA^{\wZ})$ is a core, it suffices by \cite[Lemma 17.8]{Kal}, to check that $(\lambda - \AA^{\wZ})\cD^{2, s}_0(\AA^{\wZ})$ is dense in $C_0(\wD)$ for some $\lambda > 0$. Since $\cD_\lc(\AA^{\wZ})$ is dense in $C_0(\wD)$ by Lemma \ref{lem:dom-dens-trans}, it suffices for us to prove the following lemma, to which the remainder of this subsection is devoted.

\begin{lemma} \label{lem:dom-contained}
The space $\cD_\lc(\AA^{\wZ})$ is contained in $(\lambda - \AA^{\wZ})\cD^{2, s}_0(\AA^{\wZ})$.
\end{lemma}
\begin{proof}
Recall by \cite[Theorem 17.4]{Kal} that $\lambda - \AA^{\wZ}$ is the inverse of the resolvent
\[
\cR_\lambda \wf (z) := \int_0^\infty e^{-\lambda t} \EE[\wf(\wZ(t)) \mid \wZ(0) = z] dt.
\]
Choose any $g \in \cD_\lc(\AA^{\wZ})$, and let $f = \cR_\lambda g$ so that $(\lambda - \AA^{\wZ})f = g$; we wish to check that $f \in \cD^{2, s}_0(\AA^{\wZ})$.

First, we check that $f \in C^2(\wDs)$.  For a function $\varphi$ on open sets, we write
\[
a = \lim_{U \downarrow z} \varphi(U)
\]
if for any $\eps > 0$ there exists an open set $U \ni z$ so that for any open set $V$ with $z \in V \subseteq U$ we have $|\varphi(U) - a| < \eps$.  Following \cite[Section 5.7]{Dyn}, define the characteristic operator
\[
\fA^{\wZ} f (z) := \lim_{U \downarrow z} \frac{\EE[f(\wZ(\tau_U)) \mid \wZ(0) = z] - f(z)}{\EE[\tau_U \mid \wZ(0) = z]}
\]
acting on $C_0(\wD)$, where $\tau_U$ is the first exit time from $U$. Notice that $\lambda - \fA^{\wZ}$ satisfies the strict maximum principle, meaning that if $f(z) > 0$ is a local maximum, then $(\lambda - \fA^{\wZ})f(z) > 0$.  Further, by \cite[Theorem 5.5]{Dyn}, the operator $\fA^{\wZ}$ extends $\AA^{\wZ}$, so we have that $(\lambda - \fA^{\wZ})f = g$.  We now check for each $z \in \wDs$ that $f$ is $C^2$ in a neighborhood of $z$.  We now consider two cases.

\noindent \textbf{Case 1:} Suppose $z$ lies in the interior of $\wDs$.  Choose a neighborhood $U \ni z$ which is a closed ball satisfying $U \subset \wDs$, and consider the Dirichlet problem
\[
(\lambda - \AA^Z) \wtilde{f} = g
\]
on $U$ with boundary condition $\wtilde{f} = f$ on $\partial U$. By \cite[Theorem 11.2.8]{Hel}, it has solution $\wtilde{f} \in C^2(U)$, where the regularity conditions of \cite[Theorem 11.2.8]{Hel} hold by Assumption \ref{ass:estimate}(d) and the fact that $g \in C^2_c(\wD)$.  Again using that $\fA^Z$ extends $\AA^Z$, we conclude that on $U$ we have
\[
(\lambda - \fA^Z) (\wtilde{f} - f) = 0
\]
with $\wtilde{f} - f = 0$ on $\partial U$.  Applying the strict maximum principle to any potential interior maxima of $\wtilde{f} - f$ or $f - \wtilde{f} = 0$, we conclude that $\wtilde{f} = f$ on $U$, implying that $f$ is $C^2$ on $U$.  

\noindent \textbf{Case 2:} Suppose $z$ lies on $\partial \wDs$.  Choose a neighborhood $\wDs \supset U \ni z$ which is an admissible spherical chip in the sense of \cite[Definition 11.2.3]{Hel} which intersects only the face of $\wD$ containing $z$.  Consider the mixed problem
\[
(\lambda - \AA^Z)\wtilde{f} = g
\]
on $U$ with mixed boundary condition $\wtilde{f} = f$ on $\partial U - (\partial \wD \cap U)$ and $\langle \nabla f(z), u \rangle = 0$ for $z \in \partial \wD \cap U$ and $u \in \wU(z)$.  By \cite[Theorem 11.2.8]{Hel}, it has a solution $\wtilde{f} \in C^2(U^{\text{int}} \cup (\partial \wD \cap U))$, where $U^\text{int}$ denotes the interior of $U$ and again the regularity conditions hold by Assumption \ref{ass:estimate}(d) and the fact that $g \in C^2_c(\wD)$.  As in Case 1, we conclude on $U$ that
\[
(\lambda - \fA^Z)(\wtilde{f} - f) = 0
\]
with $\wtilde{f} - f = 0$ on $\partial U - (\partial \wD \cap U)$ and $\langle \nabla (\wtilde{f} - f)(z), u\rangle = 0$ for $z \in \partial \wD \cap U$ and $u \in \wU(z)$.  Applying the strict maximum principle again implies that $\wtilde{f} = f$ on $U$ and hence that $f$ is $C^2$ on $U$, as desired.

We now check that $f \in C^1_b(\wDs)$ and that $\langle u, \nabla f(z) \rangle = 0$ for $u \in \wU(\chi^{-1}(z))$ and $z \in \partial \wDs$, which will complete the proof.  First, since Assumption \ref{ass:drift-bound} holds, by Lemma \ref{lem:grad-check2} the hypothesis of Corollary \ref{corr:grad-bound} holds for $g$. Now, notice that $e^{-\lambda t} \EE[g(\wZ(t)) \mid \wZ(0) = z]$ is $C^1$ in $z$ by Proposition \ref{prop:and-exp} and is integrable with integrable $z$-derivative by the conclusions of Corollary \ref{corr:grad-bound}.  We may therefore differentiate under the integral to see that
\[
\nabla f (z) = \int_0^\infty e^{-\lambda t} \nabla_z \EE[g(\wZ(t)) \mid \wZ(0) = z] dt.
\]
The conclusions of Corollary \ref{corr:grad-bound} now imply that $f \in C^1_b(\wDs)$ and $\langle u, \nabla f(z) \rangle = 0$ for $u \in \wU(\chi^{-1}(z))$ and $z \in \partial \wDs$, completing the proof.
\end{proof}

\appendix

{
\newcommand{\mr}{\mathbf}
\newcommand{\MZ}{\mathbb{Z}}
\newcommand{\vk}{\varkappa}
\newcommand{\BQ}{\mathbb{Q}}
\newcommand{\BR}{\mathbb{R}}
\newcommand{\BC}{\mathbb{C}}
\newcommand{\SL}{\sum\limits}
\newcommand{\IL}{\int\limits}
\newcommand{\al}{\alpha}
\newcommand{\be}{\beta}
\newcommand{\ga}{\gamma}
\newcommand{\de}{\delta}
\newcommand{\De}{\Delta}
\newcommand{\La}{\Lambda}
\newcommand{\ME}{\mathbf E}
\newcommand{\vn}{\varnothing}
\newcommand{\ls}{\varlimsup}
\newcommand{\CF}{\mathcal F}
\newcommand{\CH}{\mathcal H}
\newcommand{\CG}{\mathcal G}
\newcommand{\MP}{\mathbf P}
\newcommand{\CE}{\mathcal E}
\newcommand{\CA}{\mathcal A}
\newcommand{\CL}{\mathcal L}
\newcommand{\CS}{\mathcal S}
\newcommand{\CD}{\mathcal D}
\newcommand{\CK}{\mathcal K}
\newcommand{\CN}{\mathcal N}
\newcommand{\cl}{\mathcal I}
\newcommand{\MQ}{\mathbb Q}
\newcommand{\MY}{\mathcal Y}
\newcommand{\CW}{\mathcal W}
\newcommand{\CX}{\mathcal X}
\newcommand{\CY}{\mathcal Y}
\newcommand{\CP}{\mathcal P}
\newcommand{\MU}{\mathbf U}
\newcommand{\Oa}{\Omega}
\newcommand{\oa}{\omega}
\newcommand{\si}{\sigma}
\newcommand{\Ga}{\Gamma}
\newcommand{\pa}{\partial}
\renewcommand{\phi}{\varphi}
\newcommand{\ta}{\theta}
\newcommand{\la}{\lambda}
\newcommand{\fn}{\mathfrak{n}}
\newcommand{\Ra}{\Rightarrow}
\newcommand{\Lra}{\Leftrightarrow}
\newcommand{\TC}{\tilde{C}}
\renewcommand{\Re}{\mathop{\mathrm{Re}}\nolimits}
\renewcommand{\Im}{\mathop{\mathrm{Im}}\nolimits}
\newcommand{\ol}{\overline}
\newcommand{\oR}{\overline{R}}
\newcommand{\om}{\overline{R}^{-1}}
\newcommand{\Tau}{\mathcal T}
\newcommand{\CM}{\mathcal M}
\newcommand{\Si}{\Sigma}
\newcommand{\norm}[1]{\lVert#1\rVert}
\newcommand{\mU}{\mathbf{U}}
\newcommand{\MC}{\mathcal C}
\newcommand{\dd}{\mathrm{d}}
\newcommand{\md}{\mathrm{d}}
\newcommand{\CV}{\mathcal V}
\newcommand{\munit}{\mathbf{1}}
\newcommand{\mP}{\mathbf{p}}
\newcommand{\tr}{\text{tr}}

\section{Existence and uniqueness for systems~\eqref{eq:lw-eq} and~\eqref{eq:jw-eq}} \label{sec:appendix}

In this appendix, authored by Andrey Sarantsev, we prove strong existence and pathwise uniqueness results for reflected systems given by~\eqref{eq:lw-eq}  and~\eqref{eq:jw-eq} for general initial conditions. We stress that these initial conditions may not satisfy the Gibbs condition, and therefore fixed-time marginals of these systems no longer coincide with the distribution of eigenvalues of certain random matrix models.  More precisely, we consider initial conditions for which no particle lies on the boundary (at $0$ for the Laguerre case and at $0$ or $1$ for the Jacobi case) and no two particles on the same or adjacent levels of the Gelfand-Tsetlin pattern coincide. Our goal is to establish the following two theorems.

\begin{theorem} If the initial conditions $\{l^n_i(0)\}$ for the system~\eqref{eq:lw-eq} satisfy
$$
0 < l^n_{i}(0) < l^{n-1}_i(0) < l^{n}_{i+1}(0) \ \ \mbox{for all}\ \ n, i,
$$
then the SDER~\eqref{eq:lw-eq} admits a unique strong solution.  Further, at most one equality of the form $l^n_i(t) = l^{n-i}_i(t)$ or $l^{n-1}_i(t) = l^n_{i+1}(t)$ holds at once.  
\label{thm:existence-general-lw}
\end{theorem}

\begin{theorem} If the initial conditions $\{j^n_i(0)\}$ for the system~\eqref{eq:jw-eq}, satisfy
$$
0 < j_i^{n}(0) < j_i^{n-1}(0) < j_{i+1}^{n}(0) < 1,\ \ \mbox{for all}\ \ n, i,
$$
then the SDER~\eqref{eq:jw-eq} admits a unique strong solution.  Further, at most one equality of the form $j^n_i(t) = j^{n-i}_i(t)$ or $j^{n-1}_i(t) = j^n_{i+1}(t)$ holds at once.  
\label{thm:existence-general-jw}
\end{theorem} 

As the proofs of Theorem~\ref{thm:existence-general-lw} and Theorem~\ref{thm:existence-general-jw} are similar, we present only the proof of Theorem~\ref{thm:existence-general-jw}.  We will give first some background on reflected Brownian motions which will later be used in the proofs.

\subsection{Background on reflected Brownian motion} 

Take a positive quadrant $S = \BR^2_+$, a $2\times 2$ positive definite symmetric matrix $A$, and another $2\times 2$ matrix $R$ with units on the main diagonal. 

\begin{define} Consider an $S$-valued continuous adapted process $Z = (Z(t),\, t \ge 0)$ satisfying $Z(t) = W(t) + RL(t)$, where
\begin{itemize}
\item[(a)] $W$ is a two-dimensional driftless Brownian motion with covariance matrix $A$, starting from $W(0) \in S$; and 

\item[(b)] $L$ is a two-dimensional continuous process with $L(0) = 0$, and with nondecreasing components $L_1, L_2$, such that $L_k$ can increase only when $Z_k = 0$, for $k = 1, 2$. 
\end{itemize}
Then $Z$ is a {\it reflected Brownian motion in the quadrant} with {\it covariance matrix} $A$, and {\it reflection matrix} $R$. 
\end{define}

For the theory of reflected Brownian motion in the quadrant (and higher-dimensional positive orthant), see the survey \cite{WilliamsSurvey} and references therein. Take a continuous function $F : \BR_+ \to \BR$. A continuous adapted process $R = (R(t),\, t \ge 0)$ is called a {\it reflected Brownian motion upon the function} $F$ if there exists a Brownian motion $B$ and a continuous nondecreasing adapted process $L$ with $L(0) = 0$, which can increase only when $R(t) = F(t)$, such that one of the two cases
$$
R(t) \le F(t),\ R(t) = B(t) - L(t),\ t \ge 0;
$$
$$
R(t) \ge F(t),\ R(t) = B(t) + L(t),\ t \ge 0
$$
holds. We consider five types of systems of three {\it reflected Brownian motions} $(Y_0, Y_1, Y_2)$, together with the corresponding {\it gap processes} $Z = (Z_1, Z_2)$. We consider also a system of four reflected Brownian motions $(X_0, X_1, X_2, X_3)$, together with gap processes $(Z_1, Z_2)$.

\noindent {\bf Type A.} $Y_1(t) \le Y_0(t) \le Y_2(t)$, where $Y_0$ is a (non-reflected) Brownian motion, $Y_1$ and $Y_2$ are Brownian motions reflected upon $Y_0$; and $Z_1 = Y_0 - Y_1$, $Z_2 = Y_2 - Y_0$. 

In the next four cases, $Y_0$ is a (non-reflected) Brownian motion, $Y_1$ is a Brownian motion reflected upon $Y_0$, and $Y_2$ is a Brownian motion reflected upon $Y_1$. In addition, $Z_1(t) = |Y_1(t) - Y_0(t)|$, and $Z_2(t) = |Y_2(t) - Y_1(t)|$. 

\noindent {\bf Type B1.} $Y_2(t) \le Y_1(t) \le Y_0(t)$.  

\noindent {\bf Type B2.} $Y_0(t) \le Y_1(t) \le Y_2(t)$. 

\noindent {\bf Type C1.} $Y_1(t) \le Y_0(t)$ and $Y_2(t) \ge Y_1(t)$.

\noindent {\bf Type C2.} $Y_1(t) \ge Y_0(t)$ and $Y_2(t) \le Y_1(t)$. 

\noindent {\bf Type D.} $X_0(t) \leq X_1(t)$ and $X_2(t) \leq X_3(t)$ with $X_0(t)$ and $X_2(t)$ non-reflected Brownian motions and $X_1(t)$ and $X_3(t)$ Brownian motions reflected on $X_0(t)$ and $X_2(t)$, respectively.  The gap processes are $Z_1(t) = |X_0(t) - X_1(t)|$ and $Z_2(t) = |X_2(t) - X_3(t)|$.

There is a {\it triple collision} in a system $(Y_0, Y_1, Y_2)$ at time $t \ge 0$ if $Y_0(t) = Y_1(t) = Y_2(t)$, which is equivalent to $Z_1(t) = Z_2(t) = 0$.  There is a \textit{double collision} in the system $(X_0, X_1, X_2, X_3)$ at time $t \geq 0$ if $X_0(t) = X_1(t)$ and $X_2(t) = X_3(t)$.

\begin{lemma} 
\label{lemma:no-triple-collisions}
In each of the six systems above, there are a.s. no triple or double collisions. That is, 
\begin{equation}
\label{eq:no-triple-collisions}
\MP\left(\exists\, t > 0:\ Y_0(t) = Y_1(t) = Y_2(t)\right) = 0 \text{ and } \MP\left(\exists\, t > 0:\ X_0(t) = X_1(t) \text{ and } X_2(t) = X_3(t)\right) = 0.
\end{equation}
Equivalently, this means that
\begin{equation}
\label{eq:no-corner}
\MP\left(\exists\, t > 0:\ Z_1(t) = Z_2(t) = 0\right) = 0.
\end{equation}
\end{lemma}
\begin{proof} For Type {\bf A}, we have:
$$
Y_0(t) = W_0(t),\ Y_1(t) = W_1(t) - L_1(t),\ Y_2(t) = W_2(t) + L_2(t),\ t \ge 0,
$$
where $W_0, W_1, W_2$ are i.i.d. Brownian motions, and $L_1, L_2$ are continuous non-decreasing processes which start from zero and can grow only when $Y_0 = Y_1$, and $Y_0 = Y_2$, respectively. Therefore, the gap process satisfies
$$
Z_1(t) = W_0(t) - W_1(t) + L_1(t),\ \ Z_2(t) = W_2(t) - W_0(t) + L_2(t).
$$
The process $(W_0 - W_1, W_2 - W_0)$ is a two-dimensional driftless Brownian motion with covariance matrix
\begin{equation}
\label{eq:covariance}
A = 
\begin{bmatrix}
2 & -1\\
-1 & 2
\end{bmatrix}.
\end{equation}
Therefore, the gap process $Z$ is a reflected Brownian motion in the quadrant with covariance matrix $A$ from~\eqref{eq:covariance}, and reflection matrix
$$
R = 
\begin{bmatrix}
1 & 0\\
0 & 1
\end{bmatrix}.
$$
Applying~\cite[Theorem 2.12]{Sarantsev15}, we get~\eqref{eq:no-corner}.  

For Type {\bf B1}, we have:
$$
Y_0(t) = W_0(t),\ Y_1(t) = W_1(t) - L_1(t),\ Y_2(t) = W_2(t) - L_2(t),
$$
where $W_0, W_1, W_2$ are the same as for Type {\bf A}, and $L_1, L_2$ are continuous non-decreasing processes which can grow only when $Y_0 = Y_1$ and $Y_1 = Y_2$, respectively. Therefore,
$$
Z_1(t) = W_0(t) - W_1(t) + L_1(t),\ \ Z_2(t) = W_1(t) - W_2(t) + L_2(t) - L_1(t).
$$
Similarly to Type {\bf A}, the process $(W_0 - W_1, W_1 - W_2)$ is a two-dimensional driftless Brownian motion with covariance matrix~\eqref{eq:covariance}, and the gap process $Z$ is a reflected Brownian motion in the quadrant with covariance matrix~\eqref{eq:covariance}, and reflection matrix 
$$
R = 
\begin{bmatrix}
1 & 0\\
-1 & 1
\end{bmatrix}.
$$
Again applying~\cite[Theorem 2.12]{Sarantsev15}, we get~\eqref{eq:no-corner}. Systems of Type {\bf B2} are treated similarly. 

For Type {\bf C1}, after calculations very similar to the ones for Type {\bf B1}, we get that the gap process is a reflected Brownian motion in the quadrant with covariance matrix~\eqref{eq:covariance}, and reflection matrix
$$
R = 
\begin{bmatrix}
1 & 0\\
1 & 1
\end{bmatrix}.
$$
Apply \cite[Corollary 3.6]{BruggemanSarantsev}. In the notation of that corollary, we have: $\tilde{R} = I_2$, and $C = I_2$; therefore, $\ol{R} = I_2$, $\tr\left(\ol{R}^{-1}A\right) = \tr(A) = 4$, and 
$$
c_+(\ol{R}^{-1}) = \max\limits_{x \in \BR^2_+\setminus\{0\}}\frac{Ax\cdot x}{x\cdot x} = \max\limits_{\substack{x_1, x_2 \ge 0\\(x_1, x_2) \ne (0, 0)}}\frac{2x_1^2 + 2x_2^2 - 2x_1x_2}{x_1^2 + x_2^2} \le 2.
$$
Thus, the assumption of \cite[Corollary 3.6]{BruggemanSarantsev} is true, and we arrive at~\eqref{eq:no-corner}.  Systems of Type \textbf{C2} are treated similarly.

Finally, for Type \textbf{D}, we have
\[
X_0(t) = W_0(t), X_1(t) = W_1(t) + L_1(t), X_2(t) = W_2(t), X_3(t) = W_3(t) + L_2(t),
\]
where $W_0, W_1, W_2, W_3$ are i.i.d. Brownian motions and $L_1(t)$ and $L_2(t)$ are continuous non-decreasing processes which can grow only when $X_0 = X_1$ and $X_2 = X_3$, respectively.  We find that
\[
Z_1(t) = W_1(t) - W_0(t) + L_1(t) \text{ and } Z_2(t) = W_3(t) - W_2(t) + L_2(t),
\]
hence the gap process $(Z_1, Z_2)$ is a reflected Brownian motion in the quadrant with covariance and reflection matrices given by 
\[
A = \left[\begin{matrix} 2 & 0 \\ 0 & 2 \end{matrix}\right] \text{ and } R = \left[\begin{matrix} 1 & 0 \\ 0 & 1 \end{matrix}\right],
\]
so we get (\ref{eq:no-corner}) by \cite[Theorem 2.12]{Sarantsev15}. 
\end{proof}

\subsection{Proof of Theorem~\ref{thm:existence-general-jw}}
Let us introduce new pieces of notation:
\begin{equation}
\label{eq:f}
f(x) = \int_0^x\frac{\md u}{\phi(u)},\ \ 0 \le x \le 1, \ \mbox{where}\ \phi(u) := 2\sqrt{u(1-u)},\ \ 0 < u < 1;
\end{equation}
\begin{equation}
\label{eq:h}
h_l(x) := 2\left((p-l+1) + (p+q - 2l+2)x\right),\ \ l = 1, \ldots, p\wedge q;
\end{equation}
\begin{equation}
\label{eq:g}
g_l(x) := \frac12f''(x)\phi^2(x) + f'(x)h_l(x),\ \ l  = 1, \ldots, p\wedge q.
\end{equation}
Let us call the process $j^{n}_i(t)$ {\it a particle of level $n$ with rank $i$}. We will first prove the following three statements:
\begin{equation}
\label{eq:successive-GT}
\text{there exists a pathwise unique strong version of }
\left(j^{k}_i(t),\ k = 1, \ldots, n,\ i = 1, \ldots, k,\ t \ge 0\right);
\end{equation}
\begin{equation}
\label{eq:no-collisions}
\PP\left(\exists\, t \ge 0: j^{n}_i(t) = j^{n}_{i+1}(t)\right) = 0,\qquad i = 1, \ldots, n-1;
\end{equation}
\begin{equation}
\label{eq:no-hitting-0-1}
\PP\left(\exists\, t \ge 0: j^{n-1}_i(t) \in (0, 1)\right) = 1,\qquad i = 1, \ldots, n-1.
\end{equation}
Strong existence and pathwise uniqueness follows trivially from statements~\eqref{eq:successive-GT},~\eqref{eq:no-collisions}, and \eqref{eq:no-hitting-0-1}.

We proceed by induction on $n$. For the base case $n = 1$, there is nothing to prove. Suppose we proved this statement for $n-1$ instead of $n$, and let us prove it for $n$.  If we establish~\eqref{eq:no-collisions} and \eqref{eq:no-hitting-0-1}, existence and uniqueness for SDER~\eqref{eq:jw-eq} at level $n$ follows by applying Theorem~\ref{thm:exist-crit} to the non-colliding time-dependent boundaries given by $0$, $1$, and $j^{n - 1}_i(t)$.  We will first show~\eqref{eq:no-hitting-0-1} and use it to establish~\eqref{eq:no-collisions}.

Assume there exists a $t_0 > 0$ and an $i = 1, \ldots, n-1$ such that $j^{n-1}_i(t_0) \in \{0, 1\}$. After the mapping $x \mapsto 1 - x$ the system~\eqref{eq:jw-eq} turns into a system governed by the same SDER~\eqref{eq:jw-eq} with $p$ and $q$ exchanged, so we may assume without loss of generality that $j^{n-1}_i(t_0) = 0$. By the ranking of the particles on the $(n-1)^{\text{st}}$ level, we have that $i = 1$. Then there exists a $t_0 > 0$ such that $j^{n-1}_1(t) > 0$ for $t \in [0, t_0)$. Because $j^{n-2}_1(t) > 0$ for all $t \ge 0$ a.s.~by the induction assumption, there exists a rational $q \in (0, t_0)$ such that $j^{n-1}_1(t) < j^{n-2}_2(t)$ for all $t \in I := [q, t_0]$. Therefore, the collision term $L^{n-1, -}_1(t)$ stays constant on the time interval $I$, and the dynamics of the particle $j^{n-1}_1(t)$ on $I$ is given by
\begin{equation}
\label{eq:n-1-1}
\md j^{n-1}_1(t) = \phi\left(j^{n-1}_1(t)\right)\,\md B^{n-1}_1(t) + h_{n-1}\left(j^{n-1}_1(t)\right)\,\md t.
\end{equation}
Apply the function $f$ from~\eqref{eq:f} and get the new process 
$$
Y = (Y(t),\, t \ge 0),\ \ Y(t) = f\left(j^{n-1}_1(t)\right),
$$
with dynamics given by the following stochastic equation:
$$
\md Y(t) = \md B^{n-1}_1(t) + g_{n-1}\left(j^{n-1}_1(t)\right)\,\md t = 
 \md B^{n-1}_1(t) + v(Y(t))\,\md t,
$$
where $v(x) := g_{n-1}(f^{\leftarrow}(x))$, and $f^{\leftarrow}$ is the inverse mapping of $f$. Now, $f(x) \sim 2x^{1/2}$ as $x \downarrow 0$, and therefore $f^{\leftarrow}(x) \sim x^2/4$ as $x \downarrow 0$. Furthermore, $f'(x) = \phi^{-1}(x) \sim x^{-1/2}/2$, $f''(x) = -\phi'(x)/\phi^2(x) \sim -x^{-3/2}/4$ as $x \downarrow 0$. Also, we get from~\eqref{eq:h} that $h_{n-1}(0+) = h_{n-1}(0) = 2(p - (n-1) + 1) = 2(p - n) + 4$. Combining these asymptotic results in light of~\eqref{eq:g}, we have
$$
g_{n-1}(x) = \frac12\left(-\frac{x^{-3/2}}{4}\right)\left(2\sqrt{x(1-x)}\right)^2 + \frac1{2\sqrt{x(1 - x)}}h_{n-1}(x) \sim  (p - n + 3/2)x^{-1/2},
$$
and therefore 
$$
v(x) \sim g_{n-1}\left(\frac{x^2}4\right) \sim (p - n + 1/2)\left(\frac{x^2}4\right)^{-1/2} = \frac{2p - 2n + 3}{x},\ \ x \downarrow 0.
$$
Since $n \leq p$, it is a standard exercise in stochastic calculus to prove that $\PP(\exists\, t \in I:\, Y(t) = 0) = 0$, meaning the diffusion $Y$ a.s.~does not hit zero. Taking a countable union of probability zero events over all positive rational $q$, we complete the proof of~\eqref{eq:no-hitting-0-1}. 

Next, let us show~\eqref{eq:no-collisions}. Assume there exists an $i = 1, \ldots, n$ such that there exists a $t_0 \ge 0$ such that $j^{n}_i(t_0) = j^{n}_{i+1}(t_0)$. By the induction hypothesis applied to~\eqref{eq:no-collisions} together with~\eqref{eq:no-hitting-0-1}, we have a.s. that
$$
0 < j_{i-1}^{n-1}(t) < j_i^{n-1}(t) < j_{i+1}^{n-1}(t) < 1\qquad \mbox{for all}\qquad t \ge 0.
$$
Therefore, we get the string of inequalities
\begin{equation}
\label{eq:first-inequality}
0 < j_{i-1}^{n-1}(t_0) < j_i^{n}(t_0) = j_{i}^{n-1}(t_0) = j_{i+1}^{n}(t_0) = M < j^{n-1}_{i+1}(t_0) < 1.
\end{equation}
Now, there can be two cases.

\noindent {\bf Case 1.} Assume that the particle $j^{n-1}_{i}(t)$ does not collide with the two adjacent particles on the level $n-2$ at time $t_0$, meaning that
\begin{equation}
\label{eq:case-1}
j_{i-1}^{n-2}(t_0) < j_{i}^{n - 1}(t_0) < j_{i}^{n-2}(t_0).
\end{equation}
In light of~\eqref{eq:first-inequality} and~\eqref{eq:case-1}, together with continuity of all processes, there exist a rational $q \in (0, t_0)$, and a positive integer $m$ such that for $t \in I := [q, t_0]$, we have:
\begin{equation}
\label{eq:inequality}
m^{-1} \le j^{n-1}_{i-1}(t) < j_i^{n}(t) \le j_i^{n-1}(t) \le j_{i+1}^{n}(t) < j_{i+1}^{n-1}(t) \le 1 - m^{-1},
\end{equation}
and, in addition,
\begin{equation}
\label{eq:lower-order-inequality}
j_{i-1}^{n-2}(t) < j_{i}^{n-1}(t) < j_{i}^{n-2}(t).
\end{equation}
From~\eqref{eq:inequality}, it follows that there is no collision between $j_{i-1}^{n-1}(t)$ and $j_i^{n}(t)$, as well as between $j_{i+1}^{n-1}(t)$ and $j_{i+1}^{n}(t)$, on the time interval $I$. Therefore, on this time interval, the terms $L_{i-1}^{n, +}(t)$ and $L_{i+1}^{n, -}(t)$ are constant. Similarly, from~\eqref{eq:lower-order-inequality}, we get that the terms $L_{i}^{n-1, \pm}(t)$ are constant on this time interval. The processes $j_i^{n}(t)$, $j_i^{n-1}(t)$, $j_{i+1}^{n}(t)$ satisfy the following SDE's on the time interval $I$:
\begin{align}
\label{eq:level-n-1}
\mathrm{d}j^{n-1}_{i}(t) &= \phi\left(j_i^{n-1}(t)\right)\,\mathrm{d} B_i^{n-1}(t) + h_{n-1}\left(j^{n-1}_i(t)\right)\,\md t \\
\label{eq:level-n-plus}
\md j_{i+1}^{n}(t) &= \phi\left(j_{i+1}^{n}(t)\right)\,\md B_{i+1}^{n}(t) + h_n(j_{i+1}^{n}(t))\,\md t + \frac12\,\md L_{i+1}^{n, +}(t)\\
\label{eq:level-n-minus}
\md j_{i}^{n}(t) &= \phi\left(j_{i}^{n}(t)\right)\,\md B_{i}^{n}(t) + h_n(j_{i+1}^{n}(t))\,\md t - \frac12\,\md L_{i}^{n, -}(t).
\end{align}
Recall that $B_{i}^{n-1}(t)$, $B_i^{n}(t)$, and $B_{i+1}^{n}(t)$ are i.i.d. Brownian motions, $L_i^{n, -}(t)$ is a continuous non-decreasing process which can increase only when $j_i^{n}(t) = j_i^{n-1}(t)$, and $L_{i+1}^{n, +}(t)$ is a continuous non-decreasing process which can increase only when $j_{i+1}^{n}(t) = j_i^{n-1}(t)$. Next, we wish to get rid of the diffusion coefficients in~\eqref{eq:level-n-1},~\eqref{eq:level-n-plus},~\eqref{eq:level-n-minus}. Apply to these processes the function $f$ from~\eqref{eq:f}. Define the new processes
$$
Y(t) := f\left(j^{n-1}_{i}(t)\right),\ \ X_-(t) = f\left(j_i^{n}(t)\right),\ \ X_+(t) = f\left(j_{i+1}^{n}(t)\right).
$$
Because $f$ is strictly increasing, we have on the time interval $I$:
$X_-(t) \le Y(t) \le X_+(t)$. Applying It\^o's formula to~\eqref{eq:level-n-1},~\eqref{eq:level-n-plus},~\eqref{eq:level-n-minus}, we get for $t \in I$ that
\begin{align}
\label{eq:Y}
\md Y(t) &= \md B_i^{n-1}(t) + g_{n-1}\left(j_i^{n-1}(t)\right)\,\md t\\
\label{eq:X-}
\md X_-(t) &= \md B^{n}_i(t) + g_n\left(j_i^{n}(t)\right)\, \md t - \md\ell_-(t)\\
\label{eq:X+}
\md X_+(t) &= \md B^{n}_{i+1}(t) + g_n\left(j^{n}_{i+1}(t)\right)\, \md t + \md\ell_+(t).
\end{align}
Here, we use the notation
$$
\ell_+(t) = \frac12\int_0^tf'\left(j^{n}_{i+1}(s)\right)\,\md L_{i+1}^{n, +}(s),
$$
for a continuous nondecreasing process which can increase only when $j_{i+1}^{n}(t) = j_i^{n-1}(t)$, or, equivalently, when $X_+(t) = Z(t)$, and 
$$
\ell_-(t) = \frac12\int_0^tf'\left(j^{(n)}_i(s)\right)\,\md L_{i}^{n, -}(s)
$$
for a continuous nondecreasing process which can increase only when $j_{i}^{n}(t) = j_i^{n-1}(t)$, or, equivalently, when $X_-(t) = Z(t)$.  Comparing~\eqref{eq:g}, ~\eqref{eq:f}, and~\eqref{eq:h}, we see that $g_l$ is a continuous function on $(0, 1)$, and therefore, for $l = 1, \ldots, n$, we have
\begin{equation}
\label{eq:g-bound}
C_l := \max\limits_{x \in [m^{-1}, 1 - m^{-1}]}|g_l(x)| < \infty.
\end{equation}
Comparing~\eqref{eq:g-bound} with~\eqref{eq:inequality}, we have a.s.~that
\begin{equation}
\label{eq:Girsanov-g}
\max\limits_{t \in I}\left|g_n\left(j_i^{n}(t)\right)\right| \le C_n,\qquad \mbox{and}\qquad
\max\limits_{t \in I}\left|g_n\left(j_{i+1}^{n}(t)\right)\right| \le C_n.
\end{equation}
Apply Girsanov's theorem to~\eqref{eq:Y},~\eqref{eq:X+},~\eqref{eq:X-} in light of~\eqref{eq:Girsanov-g}. We conclude that there exists an equivalent measure $\mathbb Q$ such that under this measure, $Y(t)$ is a Brownian motion on the time interval $I$, and $X_-(t)$, $X_+(t)$ are Brownian motions reflected upon $Y(t)$. In other words, this is a system of type {\bf A}. By Lemma~\ref{lemma:no-triple-collisions}, this system a.s. does not have triple collisions where $X_-(t)  = X_+(t) = Y(t)$. Since we always have $X_-(t) \le Y(t) \le X_+(t)$, this is equivalent to saying that $\mathbb Q$-a.s. there does not exist a $t \ge 0$ such that $X_-(t) = X_+(t)$.  By equivalence, this statement is also true $\PP$-a.s. Taking a countable union of events of probability zero over all positive rationals $q$ and positive integers $m$ gives a contradiction and completes the proof of~\eqref{eq:no-collisions} for Case 1. 

\noindent {\bf Case 2.} At least one particle on the level $n-2$ hits $j^{n-1}_i(t)$ at time $t_0$. In other words, condition~\eqref{eq:case-1} is violated. This particle must be either $j^{n-2}_{i-1}(t)$ or $j^{n-2}_i(t)$. By the induction hypothesis, there are no collisions of particles on level $n - 2$. Therefore, only one particle of level $n-2$ can collide with $j^{n-1}_i(t)$ at time $t_0$. Now, let $k$ be the smallest of $l = 1, \ldots, n$ such that 
for every $j = l, \ldots, n-2$, there exists a particle from level $j$ which collides with $j_{i}^{n-1}(t)$ at time $t_0$.  Applying the induction hypothesis again, we argue that for every level $l$, there exists precisely one particle of level $l$ which collides with $j^{n-1}_i(t)$ at time $t_0$. Consider the colliding particles $X_0(t), X_1(t), X_2(t)$, on levels $k$, $k+1$, and $k+2$. If $k + 2 = n$, then pick any of the two particles $j^{n}_i(t)$ and $j^{n}_{i+1}(t)$. 

Notice that the particle $X_0(t)$ does not collide with any particle on level $k-1$ at time $t_0$. The particle $X_1(t)$ collides with only one particle of level $k$ at time $t_0$: namely, with the particle $X_0(t)$. The particle $X_2(t)$ collides with only one particle of level $k+1$ at time $t_0$, namely the particle $X_1(t)$. Similarly to Case 1, by continuity there exists a rational $q \in (0, t_0)$ and a positive integer $m$ such that on the time interval $I = [q, t_0]$, the same statements as in the previous three sentences about collisions of particles $X_0(t)$, $X_1(t)$, and $X_2(t)$ are true, and these three particles stay inside the interval $[m^{-1}, 1 - m^{-1}]$. 

Depending on the lower indices of particles $X_0(t)$, $X_1(t)$, and $X_2(t)$ in the Gelfand-Tsetlin pattern, we have one of the four orderings
\begin{align}
\label{eq:B1}
X_0(t) &\le X_1(t) \le X_2(t)\\
\label{eq:B2}
X_2(t) &\le X_1(t) \le X_0(t)\\
\label{eq:C1}
X_0(t) &\le X_1(t) \ge X_2(t)\\
\label{eq:C2}
X_0(t) &\ge X_1(t) \le X_2(t).
\end{align}
Consider the case~\eqref{eq:B1}. By the previous remark, on the time interval $I$, the particles $X_0$, $X_1$, and $X_2$ satisfy the equations
\begin{align}
\label{eq:X0}
\md X_0(t) &= \phi\left(X_0(t)\right)\,\md W_0(t) + h_k\left(X_0(t)\right)\,\md t\\
\label{eq:X1}
\md X_1(t) &= \phi\left(X_1(t)\right)\,\md t + h_{k+1}\left(X_1(t)\right)\,\md t + \md L_1(t)\\
\label{eq:X2}
\md X_2(t) &= \phi\left(X_2(t)\right)\,\md t + h_{k+2}\left(X_2(t)\right)\,\md t + \md L_2(t).
\end{align}
In these equations, $W_0(t), W_1(t), W_2(t)$ are i.i.d. Brownian motions, and $L_1(t),\, L_2(t)$ are continuous nondecreasing processes that can increase only when $X_0(t) = X_1(t)$ and $X_1(t) = X_2(t)$, respectively. After applying the function $f$ from~\eqref{eq:f} and using Girsanov's theorem similarly to Case 1, we get an equivalent probability measure $\MQ$, under which the processes $Y_0(t) = f(X_0(t)),\, Y_1(t) = f(X_1(t)),\, Y_2(t) = f(X_2(t))$ form a system of reflected Brownian motions of type {\bf B1}. From Lemma~\ref{lemma:no-triple-collisions}, this system a.s. does not have triple collisions
$$
Y_0(t) = Y_1(t) = Y_2(t)\ \Leftrightarrow X_0(t) = X_1(t) = X_2(t).
$$
Taking a countable union of $\MQ$-zero (and therefore $\MP$-zero) events over all positive rational $q$ and positive integers $m$, we obtain a contradiction and complete the proof in Case~\eqref{eq:B1}. Cases~\eqref{eq:B2}, ~\eqref{eq:C1}, and~\eqref{eq:C2} similarly correspond to systems of three reflected Brownian motions of types {\bf B2}, {\bf C1}, and {\bf C2}. This completes the proof of~\eqref{eq:no-collisions} in Case 2. 

We now check that the resulting process has no double collisions.  Assume for the sake of contradiction that there is a double collision at time $t_0 > 0$.  This collision takes one of the four forms
\begin{align*}
j^n_i(t_0) &= j^{n - 1}_i(t_0) \text{ and } j^{k}_j(t_0) = j^{k - 1}_j(t_0)\\
j^n_i(t_0) &= j^{n - 1}_i(t_0) \text{ and } j^{k}_{j + 1}(t_0) = j^{k - 1}_j(t_0)\\
j^n_{i+1}(t_0) &= j^{n - 1}_i(t_0) \text{ and } j^{k}_j(t_0) = j^{k - 1}_j(t_0)\\
j^n_{i+1}(t_0) &= j^{n - 1}_i(t_0) \text{ and } j^{k}_{j+1}(t_0) = j^{k - 1}_j(t_0).
\end{align*}
Consider the first case.  As before, there exists $q \in (0, t_0)$ and a positive integer $m$ so that on $I = [q, t_0]$ there are no other collisions of these particles and they stay in $[m^{-1}, 1 - m^{-1}]$.  Using the same change of variables and application of Girsanov's theorem as in the proof of \eqref{eq:no-collisions}, we obtain a system of reflected Brownian motions of type \textbf{D}.  By Lemma \ref{lemma:no-triple-collisions}, this system a.s. does not have double collisions.  Taking a countable union over all positive rational $q$ and positive integers $m$ gives a contradiction and completes the proof in the first case.  The other three cases are similar, implying that there are no double collisions.
}

\bibliographystyle{alpha}
\bibliography{20170724-arxiv-lagwar}
\end{document}